\documentclass[a4paper,11pt]{amsart}
\newcommand\version{April 17, 2026}
\usepackage{amssymb,amsmath,amsthm,mathrsfs,enumerate,graphicx, color}
\usepackage[pdfpagelabels,colorlinks,linkcolor=blue,citecolor=black,urlcolor=blue]{hyperref}
\usepackage{esint}
\usepackage[autostyle=false, style=english]{csquotes}
\MakeOuterQuote{"}

\newtheorem{thm}{Theorem}[section]

\newtheorem{cor}[thm]{Corollary}
\newtheorem{lem}[thm]{Lemma}
\newtheorem{prop}[thm]{Proposition}

\newtheorem{rem}[thm]{Remark}

\newcommand{\lesi}{\lesssim}

\newcommand{\supp}{\operatorname{supp}}

\newcommand{\f}{\frac}

\newcommand{\vc}{\infty}
\newcommand{\Rd}{\mathbb{R}^d}

\newcommand{\La}{L_\lambda}

\newcommand{\me}[1]{\mathrm{e}^{#1}}
\newcommand{\one}{\mathbf{1}}
\newcommand{\ca}{\mathcal{A}}
\newcommand{\R}{\mathbb{R}}
\newcommand{\C}{\mathbb{C}}
\newcommand{\Z}{\mathbb{Z}}
\newcommand{\N}{\mathbb{N}}
\newcommand{\bs}{\mathbb{S}}
\DeclareMathOperator{\dom}{dom}

\begin{document}

\date{\version}
\title[$L^p$ norm inequalities for Hardy operators in a half-space---\version]{Equivalence of Sobolev norms in Lebesgue spaces for Hardy operators in a half-space}

\authors

\author[T. A. Bui]{The Anh Bui}
\address{The Anh Bui \\
School of Mathematical and Physical Sciences, Macquarie University, NSW 2109, Australia}
\email{the.bui@mq.edu.au, bt\_anh80@yahoo.com}

\author[K. Merz]{Konstantin Merz}
\address{Konstantin Merz\\ Institute for Theoretical Physics, ETH Zurich, Wolfgang-Pauli-Strasse 27, 8093 Zurich, Switzerland, and Institut f\"ur Analysis und Algebra, Technische Universit\"at Braunschweig, Universit\"atsplatz 2, 38106 Braun\-schweig, Germany, and Department of Mathematics, Graduate School of Science, Osaka University, Toyonaka, Osaka 560-0043, Japan}
\email{k.merz@tu-bs.de, k.merz@phys.ethz.ch}

\arraycolsep=1pt

\begin{abstract}
  We consider Hardy operators, i.e., homogeneous Schr\"odinger operators consisting of the ordinary or fractional Laplacian in a half-space plus a potential given by a function which only depends on the appropriate power of the distance to the boundary of the half-space multiplied with a coupling constant. We compare the scales of homogeneous $L^p$-Sobolev spaces generated by these Hardy operators with and without potential. To that end, we prove and use new square function estimates for operators whose heat kernels decay slowly and include singular weights. Our results hold for all admissible coupling constants in the local case and for repulsive potentials (positive coupling constants) in the fractional case, and extend those obtained recently in $L^2$. They also apply to attractive potentials (negative coupling constants) in the fractional case, provided the currently known heat kernel bounds for repulsive couplings extend to this regime.
\end{abstract}

\thanks{T.~A.~Bui is supported by the Australian Research Council via the grant ARC DP260101083. 
  K.~Merz is supported through the PRIME programme of the German Academic Exchange Service (DAAD) with funds from the German Federal Ministry of Education and Research (BMBF)}

\maketitle
\tableofcontents

\section{Introduction}

\subsection{Statement of the problem and main result}
For $d\in\N=\{1,2,...\}$ and $\alpha\in(0,2]$, we consider the following generalized Schr\"odinger operators in $L^{2}(\mathbb{R}^d_+)$,
\begin{equation}\label{defn-La}
  L_\lambda^{(\alpha)} := (-\Delta)^{\alpha/2}_{\Rd_+} + \lambda x_d^{-\alpha},
\end{equation}
where $\R_+^d:=\R^{d-1}\times\R_+$, $\R_+:=(0,\infty)$, $x=(x',x_d)\in\R_+^d$ with $x'\in\R^{d-1}$ and $x_d\in\R_+$, and $(-\Delta)^{\alpha/2}_{\Rd_+}$ denotes the regional fractional Laplacian, introduced by Bogdan, Burdzy, and Chen \cite{Bogdanetal2003}, which reduces to the Dirichlet Laplacian when $\alpha=2$. For $\alpha\in(0,2)$, it is the self-adjoint Friedrichs extension of the quadratic form
\begin{align}
  \langle u,(-\Delta)^{\frac\alpha2}_{\R_+^d}u\rangle_{L^2(\R_+^d)}
  = \frac{\ca(d,-\alpha)}{2} \iint_{\R_+^d\times\R_+^d} \frac{|u(x)-u(y)|^2}{|x-y|^{d+\alpha}}\,dx\,dy, \quad u\in C_c^1(\R_+^d),
\end{align}
with
\begin{align}
  \label{eq:defcdalpha}
  \ca(d,-\alpha) := \frac{\alpha}{2^{1-\alpha}\pi^{d/2}} \frac{\Gamma\left(\frac{d+\alpha}{2}\right)}{\Gamma\left(1-\frac\alpha2\right)},
\end{align}
and its pointwise action is given by
\begin{align*}
  (-\Delta)^{\alpha/2}_{\R_+^d}u(x)
  = \ca(d,-\alpha) \int_{\R_+^d}\frac{u(x)-u(y)}{|x-y|^{d+\alpha}}\,dy, \quad x\in\R_+^d.
\end{align*}
As stated above, the operator $L_0^{(2)}=(-\Delta)_{\R_+^d}$ means the Laplacian with a Dirichlet boundary condition on $(\R_+^d)^c$. Importantly, $(-\Delta)^{\alpha/2}_{\R_+^d}\neq \big((-\Delta)_{\R_+^d}\big)^{\alpha/2}$ when $\alpha\neq 2$.

By the decomposition $(-\Delta)_{\R_+^d}=(-\Delta)_{\R^{d-1}}+(-\Delta)_{\R_+}$ and the classical Hardy inequality \cite{Hardy1919,Hardy1920,OpicKufner1990,Kufneretal2006}, $L_\lambda^{(2)}$ is bounded from below if and only if $\lambda\ge -1/4$. Moreover, if $\lambda\ge-1/4$, then $L_\lambda^{(2)}$ is bounded from below by zero. Bogdan and Dyda \cite{BogdanDyda2011} showed the corresponding Hardy inequality for $\alpha\in(0,2)$. That is,
\begin{align}
  \label{eq:hardy}
  L_\lambda^{(\alpha)} \ge 0
\end{align}
if and only if $\lambda\ge\lambda_*(\alpha)$, with
\[
  \lambda_*(\alpha) := -\f{\Gamma(\f{1+\alpha}{2})}{\pi}\Big(\Gamma(\f{1+\alpha}{2})-\f{2^{\alpha-1}\sqrt \pi}{\Gamma(1-\f{\alpha}{2})}\Big).
\]
Since $\lambda_*(2)=-1/4$ and we fix $\alpha\in(0,2]$, we write $\lambda_*$ instead of $\lambda_*(\alpha)$ and $L_\lambda$ instead of $L_\lambda^{(\alpha)}$ from now on. Throughout the paper, we assume $\lambda\ge \lambda_*$, and we call $L_\lambda$ a Hardy operator.

\smallskip
The homogeneity of $L_\lambda$ motivates to ask to what extent $L_\lambda$ and $L_0$ are comparable with each other. Our main result in this paper, Theorem~\ref{mainresult} below, says that the scales of Sobolev norms $\|(L_\lambda)^{s/2} u\|_{L^p(\R_+^d)}$ are equivalent to the norms $\|(L_0)^{s/2} u\|_{L^p(\R_+^d)}$ for $s>0$ belonging to some interval depending on all parameters $d,\alpha,\lambda,p$.
This extends \cite{FrankMerz2023}, where Frank and the second-named author proved this equivalence of Sobolev norms for $p=2$, and contributes to a recently started and active research program concerning the characterization of the domains of power of homogeneous Schrödinger operators, whose origins lie in the influential papers of Killip, Miao, Visan, Zhang, and Zheng \cite{Killipetal2016,Killipetal2018}. The study of these function spaces is not only of theoretical interest, but also crucial in applications, e.g., mathematical physics, specifically the study of many-particle quantum systems with critically strong interactions, such as relativistically described atoms \cite{Franketal2020P,Franketal2023,Franketal2023T,MerzSiedentop2022}.
We refer to Subsection~\ref{s:commentsmainresult} for a brief survey of these and more recent results going even beyond the self-adjoint setting \cite{Buietal2026}.

\smallskip
To give a convenient presentation of our result, we introduce the following parameterization of the coupling constant $\lambda$ of the Hardy potential.
For $\alpha\in (0,2]$, let $M:=\alpha$ if $\alpha<2$ and $M:=\vc$ if $\alpha=2$, and
\begin{equation}
  \label{eq - C function}
  (-1,M) \ni \sigma \mapsto C(\sigma):=\f{1}{\pi}\left(\Gamma(\alpha)\sin\f{\pi\alpha}{2}+\Gamma(1+\sigma)\Gamma(\alpha-\sigma)\sin\f{\pi(2\sigma-\alpha)}{2}\right).
\end{equation}
Note that, if $\alpha=2$, then $C(\sigma)=\sigma^2-\sigma$ for all $\sigma>-1$ by using Euler's reflection formula and that the poles of $\Gamma(\alpha-\sigma)$ cancel with the zeros of $\sin(\pi(2\sigma-\alpha)/2)$.
As noted in \cite{FrankMerz2023} and the references therein, the function $\sigma\mapsto C(\sigma)$ is continuous and symmetric with respect to $\sigma=\f{\alpha-1}{2}$, strictly increasing on $[\f{\alpha-1}{2}, M)$ and  $C(\f{\alpha-1}{2})=\lambda_*$. Moreover, $\lim_{\sigma\to M} C(\sigma)=+\vc$ and $C(\alpha-1)=C(0)=0$. Therefore, for any $\lambda\in [\lambda_*,\vc)$ there is a unique
\begin{equation}
  \label{eq-sigma}
  \sigma \in \left[\f{\alpha-1}{2}, M\right) \ \text{with} \ \ C(\sigma) = \lambda.
\end{equation}
Hence, the case $\lambda=0$ corresponds to $\sigma =\max\{0, \alpha-1\}$ and the case $\lambda>0$ corresponds to $\sigma>(\alpha-1)_+$. For $\alpha=2$ one finds $\sigma =\f{1}{2}\big(1+\sqrt {1+4\lambda}\big)$. In the following, we always assume the relation \eqref{eq-sigma} between the coupling constant $\lambda$ and the parameter $\sigma$.

In the following, we always write $C$ and $c$ to denote positive constants that are independent of the main parameters involved but whose values may differ from line to line.
We write $A\lesssim B$ for $A,B\in\R_+$ whenever there is a constant $c>0$ such that $A\le cB$. The notation $A\simeq B$ means $A\lesssim B\lesssim A$ and in this case we say that $A$ and $B$ are comparable. If we want to emphasize that the constant $c$ may depend on some parameter, say $\tau$, we write $A\lesssim_\tau B$. The dependence on fixed parameters, like $d,\alpha,p$, is usually suppressed. Moreover, we use notations $A\wedge B := \min\{A,B\}$, $A\vee B := \max\{A,B\}$, and $A_+:=\max\{A,0\}$, $p' :=p/(p-1)$ for $p\in[1,\vc]$, and the convention $1/0\equiv\infty$.
The following is our main result.

\begin{thm}
  \label{mainresult}
  Let $p\in(1,\vc)$, $\alpha\in(0,2]$, and let $\lambda\ge0$ when $\alpha<2$ and $\lambda\ge-1/4$ when $\alpha=2$. Let $\sigma$ be defined by \eqref{eq-sigma}, and let $s\in(0,2]\cap (0,\frac{2d}\alpha)$.
  \begin{enumerate}
  \item[(a)] If 
  $1/p>\frac{\alpha s}{2}-\sigma$,
  then
    \begin{align}\label{eq:equivalencesobolev1}
      \| (L_0)^{s/2}u\|_{L^p(\R_+^d)}
      \lesssim_{d,\alpha,\lambda,s,p} \| (L_\lambda)^{s/2}u\|_{L^p(\R_+^d)}
      \quad \text{for all}\ u\in C_c^\infty(\R_+^d).
    \end{align}
	
  \item[(b)] If 
  $1/p>\frac{\alpha s}{2}-(\alpha-1)_+$,
  then
    \begin{align}
      \label{eq:equivalencesobolev2}
      \| (L_\lambda)^{s/2}u\|_{L^p(\R_+^d)}
      \lesssim_{d,\alpha,\lambda,s,p} \|(L_0)^{s/2}u\|_{L^p(\R_+^d)}
      \quad \text{for all}\ u\in C_c^\infty(\R_+^d).
    \end{align}
  \end{enumerate}
  In particular, if 
  $1/p>\max\left\{\frac{\alpha s}{2}-(\alpha-1)_+,\frac{\alpha s}{2}-\sigma\right\}$,
  then we have the equivalence
  \begin{align}
    \label{eq:equivalencesobolev3}
    \| (L_\lambda)^{s/2}u\|_{L^p(\R_+^d)} \simeq_{d,\alpha,\lambda,s,p} \|(L_0)^{s/2}u\|_{L^p(\R_+^d)}
    \quad \text{for all}\ u\in C_c^\infty(\R_+^d).
  \end{align}
\end{thm}

Theorem~\ref{mainresult}, specifically \eqref{eq:equivalencesobolev1}, and a density argument (Theorem~\ref{density}), yield the following result for the Riesz transforms associated to $L_0$ and $L_\lambda$ with $\lambda\neq0$.

\begin{cor}
  \label{maincor}
  Let $p\in(1,\infty)$, $\alpha\in(0,2]$ and let $\lambda\ge0$ when $\alpha<2$ and $\lambda\ge-1/4$ when $\alpha=2$. Let $\sigma$ be defined by \eqref{eq-sigma}, and let $s\in(0,2]\cap (0,\frac{2d}\alpha)$. 
  If 
  $1/p>\frac{\alpha s}{2}-\sigma$,
  then
  \begin{align}
    \label{eq:maincor}
    \|(L_0)^{s/2}(L_\lambda)^{-s/2}f\|_{L^p(\R_+^d)} \lesssim_{d,\alpha,\lambda,s,p} \|f\|_{L^p(\R_+^d)}.
  \end{align}
\end{cor}

In the next subsection, we comment on Theorem~\ref{mainresult} and put it into context with previous works on this subject. In the following subsection, we give its proof. 

\subsection{Comments on Theorem~\ref{mainresult}}
\label{s:commentsmainresult}

\noindent
(1) Frank and the second-named author proved \eqref{eq:equivalencesobolev1}--\eqref{eq:equivalencesobolev2} in \cite[Theorem~1]{FrankMerz2023} for $p=2$. As noted in this work, the assumptions $\lambda\ge0$ for $\alpha=1$ and $\lambda\ge-1/4$ for $\alpha=2$ in Theorem~\ref{mainresult} are best possible. For $\alpha\in(0,2)\setminus\{1\}$ we do not expect $\lambda\ge 0$ to be necessary. This assumption comes from our use of heat kernel bounds for $L_\lambda$, which are currently known only for $\lambda\ge0$ when $\alpha<2$. Since we expect these bounds to extend to $\lambda\in[\lambda_*,0)$, we accompany all results whose proofs use bounds for $\me{-tL_\lambda}$ with a remark stating the potential extension. In particular, we make the following remark.

\begin{rem}
  \label{equivalencesobolevrem}
  Let $\alpha\in(0,2)$, $\lambda\in[\lambda_*,0)$ and assume that $\me{-tL_\lambda}(x,y)$ satisfies the upper bound
  \begin{align}
    \label{eq:prematureheatkernelbound}
    \me{-tL_\lambda}(x,y) \lesssim \Big(1\wedge\f{x_d}{t^{1/\alpha}}\Big)^{\sigma}\Big(1\wedge\f{y_d}{t^{1/\alpha}}\Big)^{\sigma} t^{-d/\alpha}\Big(\f{t^{1/\alpha}+|x-y|}{t^{1/\alpha}}\Big)^{-d-\alpha},
  \end{align}
  with $\sigma$ defined by \eqref{eq-sigma}.
  (Note that this would extend the upper bound for $\me{-tL_\lambda}(x,y)$ in \eqref{eq:thm-heatkernelLa- alpha < 2}, which holds for $\lambda\geq0$, to all $\lambda\geq\lambda_*$.)
  Then \eqref{eq:equivalencesobolev1} in Theorem \ref{mainresult} and \eqref{eq:maincor} in Corollary~\ref{maincor} remain valid for this value of $\lambda$ under the additional restriction $1/p<1+\sigma\wedge0$. Moreover, \eqref{eq:equivalencesobolev2} and \eqref{eq:equivalencesobolev3} in Theorem~\ref{mainresult} remain valid under the additional assumptions $1/p<1+\sigma\wedge0$ and $1/p>\max\{\frac{\alpha s}{2}-(\alpha-1)_+,-\sigma\}$. These are additional assumptions only if $\alpha<1$.
  This follows by the same arguments as in the proof in Section~\ref{s:proofmain} below, taking into account Remarks~\ref{rem-thm-squarefunctions}, \ref{reversehardyrem}, and \ref{genhardyrem}. The resulting ranges on $p$ are nonempty in this case, which is clear if $\sigma>0$. Let us now discuss the case $\sigma<0$, which can only occur if $\alpha<1$. The assumption $\alpha s/2-\sigma<1+\sigma$ for the validity of \eqref{eq:equivalencesobolev1} is equivalent to $0>\sigma>\alpha s/4-1/2$. This interval is nonempty because $\alpha<1$ and $s\leq2$. We now consider the range of admissible $p$ in \eqref{eq:equivalencesobolev2}. If $(\alpha-1)/2\leq\sigma\leq-\alpha s/2$, which can only happen if $s\leq(1-\alpha)/2$, then $-\sigma<1+\sigma$ (i.e., $\sigma>-1/2$) is satisfied since $\sigma\geq(\alpha-1)/2$. If $-\alpha s/2\leq\sigma<0$, then the range of admissible $p$ is nonempty because $\alpha s/2<1+\sigma$.
\end{rem}

Note that in \cite{JakubowskiMaciocha2025}, Jakubowski and Maciocha prove the upper heat kernel bound \eqref{eq:prematureheatkernelbound} (which appears in Remarks~\ref{equivalencesobolevrem}, \ref{rem-thm-squarefunctions}, \ref{reversehardyrem}, and \ref{genhardyrem}) for $d=1$, all $\alpha\in(0,2)$, and all $\lambda\in[\lambda_*,0)$.

\medskip
\noindent
(2)
Theorem~\ref{mainresult} joins a line of recent research \cite{Killipetal2016,Killipetal2018,Franketal2021,Merz2021,BuiDAncona2023,BuiNader2022,FrankMerz2023,Buietal2026} on equivalence of Sobolev norms involving Hardy operators.
In the groundbreaking work \cite{Killipetal2016}, Killip, Visan, and Zhang showed the equivalence of $L^p$-Sobolev norms generated by the ordinary Laplacian and the Dirichlet--Laplacian on the complement of a convex set.
In \cite{Killipetal2018}, Killip, Miao, Visan, Zhang, and Zheng considered $(-\Delta)^{\alpha/2}+\lambda/|x|^\alpha$ and proved the equivalence of $L^p$-Sobolev norms generated by this operator with $\lambda\neq0$ and $\lambda=0$ when $\alpha=2$.
D'Ancona, Frank, Nader, Siedentop, and the authors of the present paper extended their result in \cite{Franketal2021,Merz2021,BuiDAncona2023,BuiNader2022} and successively treated all admissible $\alpha\in(0,2)$, $\lambda$ and $p$. For an alternative proof and an extension of this result when $\alpha=2$, see Miao, Su, and Zheng \cite{Miaoetal2023}.
In \cite{FrankMerz2023}, Frank and the second author showed the equivalence of $L^2(\R_+^d)$-Sobolev norms generated by $L_0$ and $L_\lambda$ with $\lambda\ge-1/4$ when $\alpha=2$ and $\lambda\ge0$ when $\alpha\in(0,2)$, i.e., Theorem~\ref{mainresult} with $p=2$.
Moreover, recently, together with Duong \cite{Buietal2026}, we proved the equivalence of Sobolev norms generated by the {\em non-self-adjoint} Kolmogorov operator with scaling critical drift, given by $(-\Delta)^{\alpha/2}-\kappa|x|^{-\alpha}x\cdot\nabla$ acting on functions on $\R^d$.
All of the above results crucially relied on heat kernel bounds for the operators in question. The restriction $\lambda>0$ in \cite{FrankMerz2023} is because heat kernel bounds for $L_\lambda$ with $\alpha\in(0,2)$ are so far only available when $\lambda\ge0$. However, as noted above, once it is shown that these bounds extend to all $\lambda\in[\lambda_*,0)$, the results in \cite{FrankMerz2023} automatically extend to $\lambda\in[\lambda_*,0]$.

\medskip
\noindent
(3)
As noted in the above-mentioned works on the equivalence of Sobolev norms, homogeneous operators appear frequently as scaling limits or model operators of more complicated problems, e.g., in quantum mechanics, general relativity, engineering, and finance. Thus, besides being of pure mathematical interest, the equivalences of Sobolev norms are powerful tools, because Hardy operators can be replaced with the easier to handle and better-understood fractional Laplacians.
For instance, Killip, Visan, and Zhang \cite{Killipetal2016Q} and Killip, Miao, Murphy, Visan, Zhang, and Zheng \cite{Killipetal2017,Killipetal2017T} used the equivalence of Sobolev norms in \cite{Killipetal2016,Killipetal2018} to study global well-posedness and scattering for nonlinear Schr\"odinger equations with the Dirichlet--Laplacian on the complement of a convex obstacle, or the Hardy operator $-\Delta+\lambda/|x|^2$. 
In the fractional case, the equivalence of Sobolev norms \cite{Franketal2021} played in important role in the proof of the strong Scott conjecture for relativistic atoms \cite{Franketal2020P} by Frank, Siedentop, Simon, and the second author; see also \cite{Franketal2023} for a shorter proof, \cite{Franketal2023T} and for a review.
This problem is a genuine many-body problem, but one can reduce it to a one-body problem involving the operator $C_\lambda:=\sqrt{1-\Delta}-1+\lambda/|x|$ and its powers $(C_\lambda)^{s/2}$ in $L^2(\R^3)$. Since $\sqrt{1-\Delta}-\sqrt{-\Delta}$ is bounded, the equivalence of Sobolev norms allowed us to deduce estimates for $(C_\lambda)^{s/2}$ from estimates for $(-\Delta)^{s/4}$, which are significantly easier to prove.
Furthermore, the techniques in \cite{Franketal2020P} and the equivalence of Sobolev norms in \cite{Franketal2021} are crucial for proving the Scott conjecture in the physically more relevant Furry model in \cite{MerzSiedentop2022}, where a single electron is described the Coulomb--Dirac operator instead of by $C_\lambda$.

In view of the applications of the equivalence of Sobolev norms in \cite{Killipetal2017,Killipetal2017T}, we are optimistic that the results in \cite{FrankMerz2023} and the present paper will be useful to study, e.g., nonlinear Schr\"odinger equations involving $L_\lambda$.

\subsection{Layout of the proofs of Theorem~\ref{mainresult} and Corollary~\ref{maincor}}
\label{s:proofmain}

The proof of Theorem~\ref{mainresult} follows the ideas laid out in \cite{Killipetal2018,Franketal2021,Merz2021,BuiDAncona2023,FrankMerz2023} and will be concluded in Subsection~\ref{s:actualproofmain} below.
Compared with the case $p=2$, a new difficulty in the proof of Theorem~\ref{mainresult} for $p\neq2$ arises from the question how to make sense of $(L_\lambda)^{s/2}f$ in $L^p(\R_+^d)$. A priori, it is not clear what the domain of $(L_\lambda)^{s/2}$, understood as an operator in $L^p$, is. A classical way to tackle this question is to express $(L_\lambda)^{s/2}$ using Littlewood--Paley square functions, i.e., to prove square function estimates such as $\|(L_\lambda)^{s/2}f\|_{L^p(\R_+^d)}\simeq\|\left(\sum_{N\in2^\Z}|N^{s/2}\Psi(L_\lambda/N)f|^2\right)^{1/2}\|_{L^p(\R_+^d)}$ for a bump function $\Psi\in C_c^\infty([1/2,3/2]:[0,1])$, say. However, as in the earlier works \cite{Killipetal2018,Merz2021,BuiDAncona2023}, we will work with square functions, where the Littlewood--Paley projection $\Psi(L_\lambda/N)$ is replaced with the well-studied heat kernel $\me{-tL_\lambda}$. See Section~\ref{s:preliminaries} below for estimates involving $\me{-tL_\lambda}$ that are most relevant for us.

\subsubsection{Continuous square function estimates}
A common approach to Littlewood--Paley square function estimates relies on a Mikhlin--H\"ormander spectral multiplier theorem, that is, the $L^p$-boundedness of sufficiently smooth functions of $L_\lambda$. This method was used in \cite{Killipetal2018,Merz2021} for operators of the form $(-\Delta)^{\alpha/2}+\lambda|x|^{-\alpha}$ when $\alpha=2$ with $\lambda\geq-(d-2)^2/4$ or $\alpha\in(0,2)$ with $\lambda\geq0$, and crucially depends on $L^p$--$L^q$ off-diagonal estimates or Poisson-type heat kernel bounds. These tools fail when $\lambda<0$ due to singular weights in the sharp heat kernel estimates. To overcome this difficulty, D'Ancona and the first author \cite{BuiDAncona2023} introduced a continuous square function approach based on square function boundedness for the underlying operator, using a theory of singular integrals beyond Calder\'on--Zygmund \cite[Chapter~6]{Auscher2007}. This yielded Sobolev norm equivalence for $\lambda\leq0$ and resolved the open case in \cite{Merz2021}. Motivated by this work, we construct analogous continuous square functions and establish the corresponding square function estimates for $L_\lambda$.
More precisely, for $\gamma>0$, we define the continuous square function associated to $L_\lambda$ as
\begin{align}
  \label{eq:defsquarefunction}
  (S_{\La,\gamma}f)(x) := \Big(\int_0^\vc|(t\La)^{\gamma}e^{-t\La}f|^2\f{dt}{t}\Big)^{1/2}.
\end{align}
In Section~\ref{s:squarefunctions}, we prove the following square function estimates.

\begin{thm}
  \label{squarefunctions}
  Let $\alpha\in (0,2]$ and let $\lambda \ge 0$ when $\alpha<2$ and $\lambda\ge -1/4$ when $\alpha=2$. Let $\gamma\in (0,1]$. Then for all $1<p<\vc$,  
  \begin{align}
    \label{eq:squarefunctions1}
    \|S_{\La,\gamma}f\|_{L^p(\R_+^d)} \simeq \|f\|_{L^p(\R_+^d)}.
  \end{align}
  In particular, for $s\in (0,2]$ and $1<p<\vc$,
  \begin{align}
    \label{eq:squarefunctions2}
    \left\|\Big(\int_0^\vc t^{-s}|t\La e^{-t\La}f|^2\f{dt}{t}\Big)^{1/2}\right\|_{L^p(\R_+^d)}
    \simeq \|(\La)^{s/2} f\|_{L^p(\R_+^d)}.
  \end{align}
\end{thm}

\begin{rem}
  \label{rem-thm-squarefunctions}
  Let $\alpha\in(0,2)$, $\lambda\in[\lambda_*,0)$, i.e., $\sigma\in[(\alpha-1)/2,(\alpha-1)_+)$, and assume that $\me{-tL_\lambda}(x,y)$ satisfies the bound in \eqref{eq:prematureheatkernelbound} with $\sigma$ defined by \eqref{eq-sigma}. Then Theorem  \ref{squarefunctions} remains valid for $\f{1}{1+\sigma\wedge0} =: (p_{-\sigma})' < p < p_{-\sigma} := -\f{1}{\sigma}$ if $\sigma<0$ and $1<p<\infty$ if $\sigma\geq0$. This follows by the same arguments as in the proof below, taking into account Remark \ref{rem-thm-ptk}.
\end{rem}

\subsubsection{Reversed and generalized Hardy inequalities}
With Theorem~\ref{squarefunctions} at hand, we prove the next ingredient in the proof of Theorem~\ref{mainresult}, which we call a reversed Hardy inequality. It gives a lower bound on the Hardy potential in terms of a square function arising in estimating the difference $(L_0)^{s/2}-(L_\lambda)^{s/2}$.

\begin{thm}[Reversed Hardy inequality]
  \label{thm-difference}
  Let $p\in(1,\vc)$, $\alpha\in (0,2]$ and let $\lambda \ge 0$ when $\alpha<2$ and $\lambda\ge -1/4$ when $\alpha=2$.
  Then, for $q=\min\{\sigma,(\alpha-1)_+\}$, $-r=\min\{0,q\}$,
  $p\in(\frac{1}{1-r},\frac1r)$
  and $s\in (0,2)$, we have
  \begin{align}
    \label{eq:thm-difference}
    \left\|\left(\int_0^\vc t^{-s}\left|\left(t\La e^{-t\La} -tL_0e^{-tL_0}\right)f\right|^2\f{dt}{t}\right)^{1/2}\right\|_{L^p(\R_+^d)}
    \lesi \left\|\f{f}{x_d^{\alpha s/2}}\right\|_{L^p(\R_+^d)}.
  \end{align}
\end{thm}

We prove Theorem~\ref{thm-difference} in Section~\ref{s:reversedhardy} below.

\begin{rem}
  \label{reversehardyrem}
  Let $\alpha\in(0,2)$, $\lambda\in[\lambda_*,0)$ and assume that $\me{-tL_\lambda}(x,y)$ satisfies the bound in \eqref{eq:prematureheatkernelbound} with $\sigma$ defined by \eqref{eq-sigma}. Then the assertions of Theorem \ref{thm-difference} remain valid. This follows by the same arguments as in the proof below, taking into account Remark \ref{rem-prop-difference}.
\end{rem}

The final ingredient to prove Theorem~\ref{mainresult} is the following generalized Hardy inequality giving an upper bound on powers of the Hardy potential in terms of powers of the Hardy operator $L_\lambda$.

\begin{thm}[Generalized Hardy inequality]
  \label{thm-HardyIneq}
  Let $1<p<\vc$, $\alpha\in (0,2]$, and let $\lambda \ge 0$ when $\alpha<2$ and $\lambda\ge -1/4$ when $\alpha=2$. Let $\sigma$ be defined by \eqref{eq-sigma} and suppose $s\in (0,\frac{2d}{\alpha})\cap(0,2]$ and $\tfrac{\alpha s}{2}<1+2\sigma$.
  Then,
  \begin{align}
    \label{eq:thm-HardyIneqPre}
    \|x^{-\alpha s/2}_d(L_\lambda)^{-s/2} f\|_{L^p(\R_+^d)}\lesi \|f\|_{L^p(\R_+^d)} \quad \text{for all}\ f\in L^p(\R_+^d)
  \end{align}
  holds if and only if $(\f{\alpha s}{2}-\sigma)_+<\f{1}{p}<1+\sigma\wedge0$. In that case, we have
  \begin{align}
    \label{eq:thm-HardyIneq}
    \|x^{-\alpha s/2}_d g\|_{L^p(\R_+^d)} \lesi \|\La^{s/2}g\|_{L^p(\R_+^d)} \quad \text{for all}\ g\in C_c^\infty(\R_+^d).
  \end{align}
\end{thm}

\begin{rem}
  \label{genhardyrem}
  Let $\alpha\in(0,2)$, $\lambda\in[\lambda_*,0)$ and assume that $\me{-tL_\lambda}(x,y)$ satisfies the bound in \eqref{eq:prematureheatkernelbound} with $\sigma$ defined by \eqref{eq-sigma}. Then the assertions of Theorem \ref{thm-HardyIneq} remain valid. This follows by the same arguments as in the proof below, taking into account Remark \ref{rieszrem}.
\end{rem}

Our proof of Theorem~\ref{thm-HardyIneq} (see Section~\ref{s:generalizedhardy}) is similar to that in \cite[Theorems~3 and 13]{FrankMerz2023}, where $p=2$ was considered. 

\subsection{Proofs of Theorem~\ref{mainresult} and Corollary~\ref{maincor}}
\label{s:actualproofmain}

With the above ingredients at hand, we give the
\begin{proof}[Proof of Theorem~\ref{mainresult}]
  We only consider the proof of \eqref{eq:equivalencesobolev1}; the proof of \eqref{eq:equivalencesobolev2} is analogous.
  For $s\in(0,2)$, the square function estimates in \eqref{eq:squarefunctions2}, the triangle inequality, the reversed and generalized Hardy inequalities (Theorems~\ref{thm-difference}--\ref{thm-HardyIneq}) yield
  \begin{align}
    \begin{split}
      \| (L_0)^{\frac s2}u\|_{L^p(\R_+^d)}
      & \simeq \left\|\Big(\int_0^\vc t^{-s}|tL_0 e^{-tL_0}u|^2\f{dt}{t}\Big)^{1/2}\right\|_{L^p(\R_+^d)} \\
      & \le \left\|\Big(\int_0^\vc t^{-s}|t\La e^{-t\La}u|^2\f{dt}{t}\Big)^{1/2}\right\|_{L^p(\R_+^d)} \\
      & \quad + \left\|\Big(\int_0^\vc t^{-s}|(t L_0 e^{-tL_0} -t L_\lambda \me{-tL_\lambda})u|^2\f{dt}{t}\Big)^{\frac12} \right\|_{L^p(\R_+^d)} \\
      & \lesssim \| (L_\lambda)^{s/2}u\|_{L^p(\R_+^d)} + \left\|\f{u}{x_d^{\alpha s/2}}\right\|_{L^p(\R_+^d)}
        \lesssim \| (L_\lambda)^{s/2}u\|_{L^p(\R_+^d)}.
    \end{split}
  \end{align}
  For $s=2$, the square function estimates are not necessary; we simply use the triangle inequality and the generalized Hardy inequality to get the desired estimate.
\end{proof}

To prove Corollary~\ref{maincor}, we use the following density result, which is proved in \cite[Theorem~24]{FrankMerz2023} for $p=2$.

\begin{thm}
  \label{density}
  Let $p\in(1,\vc)$, $\alpha\in(0,2]$ and let $\lambda\ge 0$ when $\alpha<2$ and $\lambda\ge-1/4$ when $\alpha=2$. Let $\sigma$ be defined by \eqref{eq-sigma}, and let $s\in(0,2]$. Assume that $s<2(1/p+\sigma)/\alpha$. Then for any $f\in L^p(\R^d_+)$ there is a sequence $(\phi_n)\subset C^\infty_c(\R^d_+)$ such that
  $$
  L_\lambda^{s/2}\phi_n\to f
  \ \text{in}\ L^p(\R^d_+) \,.
  $$
\end{thm}

Theorem~\ref{density} is proved in Section~\ref{s:density}.

\begin{rem}
  \label{densityrem}
  Let $\alpha\in(0,2)$, $\lambda\in[\lambda_*,0)$ and assume that $\me{-tL_\lambda}(x,y)$ satisfies the upper bound in \eqref{eq:prematureheatkernelbound} with $\sigma$ defined by \eqref{eq-sigma}. Then Theorem \ref{density} remains valid for this value of $\lambda$. This follows by the same arguments as in the proof below, since Lemma \ref{pointwise} remains valid for this value of $\lambda$.
\end{rem}

We can now give the
\begin{proof}[Proof of Corollary~\ref{maincor}]
  It suffices to show \eqref{eq:maincor} for a dense set of $f\in L^p(\R_+^d)$. By Theorem~\ref{density}, we may choose $f=(L_\lambda)^{s/2}g$ with $g\in C_c^\infty(\R_+^d)$. Then, by Theorem~\ref{mainresult},
  \begin{align}
    \|(L_0)^{s/2}g\|_{L^p(\R_+^d)} \lesssim_{d,\alpha,\lambda,s,p} \|(L_\lambda)^{s/2}g\|_{L^p(\R_+^d)} = \|f\|_{L^p(\R_+^d)},
  \end{align}
  which concludes the proof of \eqref{eq:maincor}.
\end{proof}

\subsection*{Organization}

In Section~\ref{s:preliminaries}, we recall kernel bounds for $\me{-tL_\lambda}$ and use them, together with a Phragm\'en--Lindel\"of argument, to prove bounds for complex-time heat kernels. We use these bounds to prove new bounds for the kernels of $(tL_\lambda)^k\me{-tL_\lambda}$ and prove corresponding $L^p\to L^q$-estimates. We use these estimates in Section~\ref{s:squarefunctions} to prove the square function estimates for $(L_\lambda)^{s/2}$ in Theorem~\ref{squarefunctions}.
In Section~\ref{s:newboundsdifferenceskernels}, we extend bounds for the kernel of the difference $e^{-t\La} -e^{-tL_0}$ considered in \cite{FrankMerz2023} and prove bounds for the kernel of the difference $tL_0 e^{-tL_0}- t\La e^{-t\La}$.
These bounds are crucial to prove the reversed Hardy inequality in Section~\ref{s:reversedhardy}.
In Section~\ref{s:generalizedhardy}, we prove the generalized Hardy inequality (Theorem~\ref{thm-HardyIneq}).
Finally, in Section~\ref{s:density}, we show the density of $(L_\lambda)^{s/2}C_c^\infty(\R^d_+)$ in $L^p(\R_+^d)$ (Theorem~\ref{density}).

\section{Estimates involving the heat kernel of $L_\lambda$}
\label{s:preliminaries}

In this section, we collect known pointwise estimates for the heat kernel $\me{-tL_\lambda}$ and derive novel bounds for the complex-time heat kernel and $(tL_\lambda)^k\me{-tL_\lambda}$, $k\in\N$. In the first subsection, we discuss pointwise and in the following subsection $L^p\to L^q$-estimates.

\subsection{Pointwise estimates}
\label{ss:kernelestimates}
We recall pointwise estimates for the kernel $p_t(x,y)$ of $\me{-tL_\lambda}$.

\begin{thm}
  \label{thm-heatkernelLa- alpha < 2}
  \begin{enumerate}[\rm (a)]
    \item Let $\alpha\in (0,2)$, $\lambda\ge 0$, and let $\sigma$ be defined by \eqref{eq-sigma}.  Then for all $t>0$ and $x,y \in \Rd_+$,
      \begin{align}
        \label{eq:thm-heatkernelLa- alpha < 2}
        p_t(x,y)\  {\simeq} \  \Big(1\wedge\f{x_d}{t^{1/\alpha}}\Big)^{\sigma}\Big(1\wedge\f{y_d}{t^{1/\alpha}}\Big)^{\sigma} t^{-d/\alpha}\Big(\f{t^{1/\alpha}+|x-y|}{t^{1/\alpha}}\Big)^{-d-\alpha}.
      \end{align}
      
    \item Let $\alpha=2$, $\lambda\ge -\f{1}{4}$, and let $\sigma$ be defined by \eqref{eq-sigma}. Then for all $t>0$ and $x,y \in \Rd_+$,
  \begin{align}
    \label{eq:thm-heatkernelLa- alpha = 2}
    p_t(x,y)\  {\simeq} \  \Big(1\wedge\f{x_d}{\sqrt t}\Big)^{\sigma}\Big(1\wedge\f{y_d}{\sqrt t}\Big)^{\sigma} t^{-d/2}\exp\Big(-\f{|x-y|^2}{ct}\Big).
  \end{align}
\end{enumerate}
\end{thm}

For $\lambda\!=\!0$ and $\alpha\!\le\! 1$ the bounds were proved in \cite[Theorem~1.1]{ChenKumagai2003}, while \cite[Theorem~1.1]{Chenetal2010} proved the estimates for $\alpha\!\in\!(1,2)$.
The bounds for $\alpha\in(0,2)$ and $\lambda\ge0$ were recently proved in \cite[Theorem~3.2, Remark~3.3]{Choetal2020} and extended in \cite[Theorem~1.1]{Songetal2025}. Note that in \cite{Choetal2020}, the parameter $\sigma$ here is denoted by $p$ there. Moreover, $C(\sigma)$ in \eqref{eq - C function} here equals
\begin{align*}
  C(d,\alpha,p) & = \ca(d,-\alpha)\frac{|\bs^{d-2}|}{2} B\left(\frac{\alpha+1}{2},\frac{d-1}{2}\right) \gamma(\alpha,p), \quad p\in(-1,\alpha), \ d\geq2, \\
  C(1,\alpha,p) & = \ca(1,-\alpha)\gamma(\alpha,p), \quad p\in(-1,\alpha), \ d=1,
\end{align*}
with $\ca(d,-\alpha)$ as in \eqref{eq:defcdalpha},
\begin{align*}
  \gamma(\alpha,p) & = \int_0^1 \frac{(t^p-1)(1-t^{\alpha-p-1})}{(1-t)^{1+\alpha}}\,dt,
\end{align*}
and $p=\sigma$ in \cite[Remark~3.3]{Choetal2020}, as shown in \cite[Appendix~A]{FrankMerz2023}. Note that the function $C(d,\alpha,p)$ is indeed independent of $d$ since $\ca(d,-\alpha)\frac{|\bs^{d-2}|}{2} B\left(\frac{\alpha+1}{2},\frac{d-1}{2}\right)=\ca(1,-\alpha)$.

For $\alpha=2$, the bound  \eqref{eq:thm-heatkernelLa- alpha = 2} follows from the factorization of the Gaussian heat kernel with respect to $\R_+^d\!=\!\R^{d-1}\!\times\!\R_+$ and the explicit expression for the one-dimensional heat kernel, i.e., the Bessel heat kernel, which is contained, e.g., in \cite[p.~75]{BorodinSalminen2002}; we refer to \cite[Appendix~B]{FrankMerz2023} for a proof.
We also refer to \cite{BogdanMerz2025,GrzywnyTrojan2021} for sharp bounds for the $\tfrac\alpha2$-stable subordinated Bessel heat kernels.

\smallskip
Next, by a Phragm\'en--Lindel\"of argument, we extend the upper bound for $p_t(x,y)$ to complex times $z\in \mathbb{C}_{\theta}:=\{z\in \mathbb{C}: |\arg z|<\theta\}$, $\theta\in[0,\pi/2]$. These bounds are the key to estimate the kernels of $(tL_\lambda)^k\me{-tL_\lambda}$, $k\in\N$.
\begin{prop}
  \label{heatkernelestimates-halfplane}
  \begin{enumerate}[\rm (a)]
  \item Let $\alpha\in (0,2)$, $\lambda\ge0$, and let $\sigma$ be defined by \eqref{eq-sigma}. Let $p_z(x,y)$ be the kernels associated to the semigroups $e^{-zL_\lambda}$ with $z\in \mathbb{C}_{\pi/2}=\{z\in \mathbb{C}: |\arg z|<\pi/2\}$. Then for any $\epsilon\in (0,1)$, there exists a constant $C=C_\epsilon$ such that
    \begin{equation}
      \label{boundedp_t(x,y)-complex-alpha < 2}
      |p_z(x,y)|\le C  \Big(1\wedge\f{x_d}{|z|^{1/\alpha}}\Big)^{\sigma}\Big(1\wedge\f{y_d}{|z|^{1/\alpha}}\Big)^{\sigma}|z|^{-d/\alpha}\Big(\f{|z|^{1/\alpha}}{|z|^{1/\alpha}+|x-y|}\Big)^{(d+\alpha)(1-\epsilon)}
    \end{equation}
    for all $x,y \in \Rd_+$ and $z\in \mathbb{C}_{\epsilon\pi/4}$.
    
    \medskip
    
  \item Let $\alpha=2$, $\lambda\ge-1/4$, and let $\sigma$ be defined by \eqref{eq-sigma}. Let $p_z(x,y)$ be the kernels associated to the semigroups $e^{-zL_\lambda}$ with $z\in \mathbb{C}_{\pi/2}$. Then there exist constants $C, c>0$ such that
    \begin{equation}
      \label{boundedp_t(x,y)-complex-alpha = 2}
      |p_z(x,y)|\le C  \Big(1\wedge\f{x_d}{\sqrt{|z|}}\Big)^{\sigma}\Big(1\wedge\f{y_d}{\sqrt{|z|}}\Big)^{\sigma}|z|^{-d/2}\exp\Big(-\f{|x-y|^2}{c|z|}\Big)
    \end{equation}
    for all $x,y \in \Rd_+$ and $z\in \mathbb{C}_{\pi/4}$.
  \end{enumerate}
\end{prop}

\begin{rem}
  \label{rem-heatkernelestimates-halfplane}
  Let $\alpha\in(0,2)$, $\lambda\in[\lambda_*,0)$ and assume that $\me{-tL_\lambda}(x,y)$ satisfies the bound in \eqref{eq:thm-heatkernelLa- alpha < 2} with $\sigma$ defined by \eqref{eq-sigma}. Then \eqref{boundedp_t(x,y)-complex-alpha < 2} remains valid. This follows by the same arguments as in the proof below.
\end{rem}

\begin{proof} 
  We only to give the proof for \eqref{boundedp_t(x,y)-complex-alpha < 2} since the proof of  \eqref{boundedp_t(x,y)-complex-alpha = 2} can be done similarly. To that end, we use a Phragm\'en--Lindel\"of argument as in \cite[Section~3.4]{Davies1990}. We set
  \[
    w_z(x)= \Big(1+\f{ {|z|^{1/\alpha}}}{x_d}\Big)^{\sigma}, \quad x\in  \Rd_+,\ z\in\C
  \]
  and first claim that there exists $C>0$ such that
  \begin{equation}\label{boundedp_t(x,y)-complex1'}
    \left|w_z(x) p_z(x,y)w_z(y)\right|\le \f{C}{|z|^{d/\alpha}}
  \end{equation}
  holds for all $x,y\in \Rd_+$ and $z\in\mathbb{C}_{\epsilon\pi/4}$.
  To prove this, we define, for $f: \mathbb{R}^d_+\rightarrow \mathbb{R}$, the norms
  $$
  |f|_{L_{w_{z}}^\vc(\Rd_+)}=\sup_{x\in \Rd_+} \left|f(x)w_z(x)\right|
  $$
  and
  \[
    |f|_{L_{w_{z}}^1(\Rd_+)}=\int_{\Rd_+}|f(x)|w_z(x)dx.
  \]
  Hence, the inequality  (\ref{boundedp_t(x,y)-complex1'}) is equivalent to
  $$
  \|e^{-z\La}\|_{L^1_{w_z^{-1}}(\Rd_+)\rightarrow L^\vc_{w_z}(\Rd_+)}\le \f{C}{|z|^{d/\alpha}}.
  $$
  For $z\in\mathbb{C}_{\pi/4}$ we can write $z=2t+is$ for some $t\ge 0$ and $s\in \mathbb{R}$ such that $t\simeq |z|$. Then, 
  \begin{align*}
    & \|e^{-z\La}\|_{L^1_{w_z^{-1}}(\Rd_+)\rightarrow L^\vc_{w_z}(\Rd_+)} \\
    & \quad \le \|e^{-t\La}\|_{L^2(\Rd_+)\rightarrow L^\vc_{w_z}(\Rd_+)}\|e^{-is\La}\|_{L^2(\Rd_+)\rightarrow L^2(\Rd_+)}\|e^{-t\La}\|_{L^1_{{w_z}^{-1}}(\Rd_+)\rightarrow L^2(\Rd_+)}.
  \end{align*}
  We first estimate $\|e^{-t\La}\|_{L^1_{{w_z}^{-1}}(\Rd_+)\rightarrow L^2(\Rd_+)}$. For $f\in L^1_{{w_z}^{-1}}(\Rd_+)$  we have,  by Theorem \ref{thm-heatkernelLa- alpha < 2} and $\displaystyle \Big(1\wedge\f{y_d}{t^{1/\alpha}}\Big)^{\sigma}\simeq \Big(1\wedge\f{y_d}{|z|^{1/\alpha}}\Big)^{\sigma}\simeq w_z(y)^{-1}$,
  $$
  \begin{aligned}
    \|e^{-t\La}f\|_{L^2(\Rd_+)}&\lesi \Big[\int_{\mathbb{R}^d_+}\Big|\int_{\mathbb{R}^d_+}   {t^{-\frac d\alpha}  \Big(\f{t^{1/\alpha}+|x-y|}{t^{1/\alpha}}\Big)^{-d-\alpha}w_z(y)^{-1}|f(y)|}dy\Big|^2 \left(1\wedge\frac{x_d}{t^{1/\alpha}}\right)^{2\sigma} \,dx\Big]^{\frac12}.
  \end{aligned}
  $$
  By Minkowski's inequality,
  $$
  \begin{aligned}
    \|e^{-t\La}f\|_{L^2(\Rd_+)}
    &\lesi \int_{\mathbb{R}^d}\Big[\int_{\mathbb{R}^d_+}\Big| t^{-\frac d\alpha}   \Big(\f{t^{1/\alpha}+|x-y|}{t^{1/\alpha}}\Big)^{-d-\alpha}\Big|^2 \left(1\wedge\frac{x_d}{t^{1/\alpha}}\right)^{2\sigma}\,dx\Big]^{\frac12} w_z(y)^{-1}|f(y)|dy\\
    &\lesi  {\sup_{y\in \Rd}}\Big[\int_{\mathbb{R}^d}\Big| t^{-\frac d\alpha}\Big(\f{t^{1/\alpha}+|x-y|}{t^{1/\alpha}}\Big)^{-d-\alpha}\Big|^2 \left(1\wedge\frac{x_d}{t^{1/\alpha}}\right)^{2\sigma}\,dx\Big]^{\frac12}\|f\|_{L^1_{w_z^{-1}}}.
  \end{aligned}
  $$
  This, along with the fact that
  \begin{equation}
    \label{eq1-norm}
    {\sup_{y\in \Rd}}\Big[\int_{\mathbb{R}^d}\Big| t^{-d/\alpha} \Big(\f{t^{1/\alpha}+|x-y|}{t^{1/\alpha}}\Big)^{-d-\alpha}\Big|^2 \left(1\wedge\frac{x_d}{t^{1/\alpha}}\right)^{2\sigma}\, dx\Big]^{1/2}
    \lesi t^{-\f{d}{2\alpha}},
  \end{equation}
  which can be seen by distinguishing between $x_d\lessgtr t^{1/\alpha}$ and using $\sigma>-1/2$, implies
  \[
    \|e^{-t\La}f\|_{L^2(\Rd_+)}\lesi t^{-\f{d}{2\alpha}}\|f\|_{L^1_{w_z^{-1}}(\R_+^d)}.
  \]
  It follows that 
  \[
    \|e^{-t\La}\|_{L^1_{{w_z}^{-1}}(\Rd_+)\rightarrow L^2(\Rd_+)}\lesi t^{-\f{d}{2\alpha}}.
  \]
  Similarly,
  \[
    \|e^{-t\La}\|_{L^2(\Rd_+)\to L^\vc_{{w_z}}(\Rd_+)}\lesi t^{-\f{d}{2\alpha}}.
  \]
  Moreover, since $L_\lambda$ is a (non-negative) self-adjoint operator, $\|e^{-is\La}\|_{L^2(\Rd_+)\rightarrow L^2(\Rd_+)}\le 1$. Taking these three estimates into account, we get  
  $$
  \|e^{-z\La}\|_{L^1_{w_z^{-1}}(\Rd_+)\rightarrow L^\vc_{w_z}(\Rd_+)}\le t^{-\f{d}{\alpha}} \simeq \f{1}{|z|^{d/\alpha}},
  $$
  which implies \eqref{boundedp_t(x,y)-complex1'}, i.e.,
  \[
    |p_z(x,y)|\lesi \f{1}{|z|^{d/\alpha}}\Big(1\wedge\f{x_d}{|z|^{1/\alpha}}\Big)^{\sigma}\Big(1\wedge\f{y_d}{|z|^{1/\alpha}}\Big)^{\sigma}
  \]
  for all $x,y\in \Rd_+$ and $z\in\mathbb{C}_{\epsilon\pi/4}$. By the Phragm\'en--Lindel\"of theorem in \cite[Proposition 3.3]{DuongRobinson1996} (see also \cite[Theorem~2.1]{Merz2022}), we get \eqref{boundedp_t(x,y)-complex-alpha < 2}.
\end{proof}

For each $k\in \mathbb N$, we denote by $p_{t,k}(x,y)$ the kernel of $(tL_\lambda)^ke^{-tL_\lambda}$. As a direct consequence of Proposition \ref{heatkernelestimates-halfplane} and Cauchy's formula, we obtain the following result.

\begin{prop}
  \label{thm-ptk}
  \begin{enumerate}[\rm (a)]
  \item Let $\alpha\in (0,2)$, $\lambda\ge0$, and let $\sigma$ be defined by \eqref{eq-sigma}.  Then for any $\epsilon>0$ and $k\in \mathbb N$, there exists a constant  $C=C(k,\epsilon)$ such that
    \begin{equation}
      \label{boundedp_t k(x,y)-complex-alpha < 2}
      |p_{t,k}(x,y)|\le C  \Big(1\wedge\f{x_d}{t^{1/\alpha}}\Big)^{\sigma}\Big(1\wedge\f{y_d}{t^{1/\alpha}}\Big)^{\sigma} t^{-(k+d/\alpha)}\Big(\f{t^{1/\alpha}+|x-y|}{t^{1/\alpha}}\Big)^{-d-\alpha+\epsilon}
    \end{equation}
    for all $x,y \in \Rd_+$ and $t>0$.
    
    \medskip
    
  \item Let $\alpha=2$ and let $\sigma$ be defined by \eqref{eq-sigma}.  Then for each $k\in \mathbb N$ there exist  constants  $C, c>0$ such that
    \begin{equation}\label{boundedp_t k(x,y)-complex-alpha = 2}
      |p_{t,k}(x,y)|\le C  \Big(1\wedge\f{x_d}{\sqrt t}\Big)^{\sigma}\Big(1\wedge\f{y_d}{\sqrt t}\Big)^{\sigma} t^{-(k+d/2)}\exp\Big(-\f{|x-y|^2}{ct}\Big)
    \end{equation}
    for all $x,y \in \Rd_+$ and $t>0$.
  \end{enumerate}		
\end{prop}

\begin{rem}
  \label{rem-thm-ptk}
  Let $\alpha\in(0,2)$, $\lambda\in[\lambda_*,0)$ and assume that $\me{-tL_\lambda}(x,y)$ satisfies the bound in \eqref{eq:thm-heatkernelLa- alpha < 2} with $\sigma$ defined by \eqref{eq-sigma}. Then \eqref{boundedp_t k(x,y)-complex-alpha < 2} remains valid. This follows by the same arguments as in the proof below, taking into account Remark~\ref{rem-heatkernelestimates-halfplane}.
\end{rem}

\begin{proof}
  Applying Cauchy's formula, for every $t>0$ and $k\in \mathbb N$,
  \[
    \La^ke^{-t\La} = \f{(-1)^kk!}{2\pi i}\int_{|\xi-t|=\eta t} e^{-\xi \La}\f{d\xi}{(\xi-t)^{k+1}},
  \]
  where $\eta>0$ is small enough so that $\{\xi: |\xi-t|=\eta t\}\subset \C_{\tilde\epsilon\pi/4}$ with $\tilde \epsilon = \f{\epsilon}{d+\alpha}$, and the integral does not depend on the choice of $\eta$. Thus, by Proposition \ref{heatkernelestimates-halfplane} and $|\xi|\simeq |\xi-t|\simeq t$,
  \[
    \begin{aligned}
      |p_{t,k}(x,y)|
      &\le C_{k,\tilde\epsilon}  \Big(1\wedge\f{x_d}{t^{1/\alpha}}\Big)^{\sigma}\Big(1\wedge\f{y_d}{t^{1/\alpha}}\Big)^{\sigma}t^{-(k+d/\alpha)} \Big(\f{t^{1/\alpha}+|x-y|}{t^{1/\alpha}}\Big)^{-(d+\alpha)(1-\tilde\epsilon)} \\
      &\le C_{k,\epsilon}  \Big(1\wedge\f{x_d}{t^{1/\alpha}}\Big)^{\sigma}\Big(1\wedge\f{y_d}{t^{1/\alpha}}\Big)^{\sigma}t^{-(k+d/\alpha)} \Big(\f{t^{1/\alpha}+|x-y|}{t^{1/\alpha}}\Big)^{-(d+\alpha-\epsilon)}
    \end{aligned}
  \]
  holds for all $x,y\in \Rd$ and $t>0$.
\end{proof}

\subsection{Spatially averaged estimates}
\label{s:averagedestimates}

In this section, we use the pointwise estimates from the previous section and tools from singular integral theory to prove spatially weighted estimates involving the heat kernel bounds of $L_\lambda$. We will use them to prove the square function estimates in Section~\ref{s:squarefunctions} below. To do so, we introduce further notation.
We write $\|f\|_p:=\|f\|_{L^p(\R_+^d)}$ and denote, for $\theta<1/2$,
\[
  p_\theta = \begin{cases}
               \f{1}{\theta}, \ \ & \text{if $\theta>0$}, \\
               \vc, \ \ & \text{if $\theta\le0$}.
             \end{cases}
\]
We denote the average of a measurable function $f$ over a measurable set $E\subseteq\R^N$ with $N$-dimensional Lebesgue measure $|E|\in(0,\infty)$ by
$$
\fint_E f(x)dx=\f{1}{|E|}\int_E f(x)dx.
$$
For any ball $B\subseteq\R^d$, we associate annuli $S_j(B):=2^{j}B\backslash 2^{j-1}B$, $j\in\N$, and write $S_0(B)=B$.

\begin{thm}
  \label{thm-Tt}
  Let $\{T_t\}_{t>0}$ be a family of linear integral operators on $L^2(\Rd_+)$ with their associated kernels $T_t(x,y)$. Assume that there exist $0<\beta<\alpha$  and $\theta<1/2$ such that for all $t>0$ and $x,y \in \Rd_+$,
  \begin{equation}\label{kernelTt}
    |T_t(x,y)|\le t^{-d/\alpha}\Big(\f{t^{1/\alpha}+|x-y|}{t^{1/\alpha}}\Big)^{-d-\beta}\Big(1\wedge\f{x_d}{t^{1/\alpha}}\Big)^{-\theta}\Big(1\wedge\f{y_d}{t^{1/\alpha}}\Big)^{-\theta}.
  \end{equation}
  Assume that $(p_\theta)'<p\le q< p_\theta$. Then for any ball $B$ with radius $r_B$, for every  {$t>0$} and $j\in \mathbb{N}$,  we have
  \begin{equation}\label{eq1-Tt}
    \Big(\fint_{S_j(B)}|T_tf|^q\Big)^{\frac1q}
    \le \max\Big\{\Big(\f{r_B}{t^{1/\alpha}}\Big)^{d-\frac{1}{p'}}, \Big(\f{r_B}{t^{1/\alpha}}\Big)^{d}\Big\}\Big(1+\f{t^{1/\alpha}}{2^jr_B}\Big)^{\frac1q} \Big(1+\f{2^jr_B}{t^{1/\alpha}}\Big)^{-d-\beta} \Big(\fint_B|f|^p\Big)^{\frac1p}
  \end{equation}
  for all $f\in L^p(\Rd_+)$ supported in $B$, and
  \begin{equation}\label{eq2-Tt}
    \Big(\fint_{B}|T_tf|^q\Big)^{\frac1q}
    \le \max\Big\{\Big(\f{2^jr_B}{t^{1/\alpha}}\Big)^{d-\frac{1}{p'}},\Big(\f{2^jr_B}{t^{1/\alpha}}\Big)^{d} \Big\} \Big(1+\f{t^{1/\alpha}}{r_B}\Big)^{\frac1q}\Big(1+\f{2^jr_B}{t^{1/\alpha}}\Big)^{-d-\beta} \Big(\fint_{S_j(B)}|f|^p\Big)^{\frac1p}
  \end{equation}
  for all $f\in L^p(\R_+^d)$ supported in $S_j(B)$.
\end{thm}

\begin{proof}
  For $\theta\le0$, the proof is immediate. The following more elaborate arguments cover this case nevertheless. In the following, we only prove \eqref{eq1-Tt}; the proof of \eqref{eq2-Tt} is analogous.
  Let $B=B(x_B, r_B)$ be a ball in $\Rd_+$ with $x_B = (x_B', x_{B}^d)$ with $x_B' \in \R^{d-1}$ and $x_B^d \in (0,\vc)$. We write $2^j B:=B(x_B,2^j r_B)$ and
  \[
    B' := \{y'\in \R^{d-1}: \|y'-x_B'\|<r_B\}.
  \]
  With the abbreviation
  \[
    D_\theta(x_d,t) := \Big(1\wedge\f{x_d}{t^{1/\alpha}}\Big)^{-\theta}
    \quad \text{for all} \ x_d,t>0, \ \theta\in\R,
  \]
  we have
  \[
    \begin{aligned}
      \|T_tf\|_{L^q(S_j(B))}
      &\le \left\{\int_{S_j(B)}\Big[\int_B  t^{-\frac d\alpha}\Big(\f{t^{1/\alpha}+|x-y|}{t^{1/\alpha}}\Big)^{-d-\alpha}D_\theta(x_d,t)D_\theta(y_d,t)|f(y)|dy\Big]^qdx\right\}^{\frac1q}\\
      &\lesi I_1 + I_2 + I_3 + I_4,
    \end{aligned}
  \]
  where
  $$
  \begin{aligned}
    I_1&= \left\{\int_{S_j(B)\cap \{x_d\le t^{1/\alpha}\}}\Big[\int_{B\cap \{y_d\le t^{1/\alpha}\}}  t^{-\frac d\alpha}\Big(\f{t^{1/\alpha}+|x-y|}{t^{1/\alpha}}\Big)^{-d-\beta}D_\theta(x_d,t)D_\theta(y_d,t)|f(y)|dy\Big]^qdx\right\}^{\frac1q},\\
    I_2&= \left\{\int_{S_j(B)\cap \{x_d\le t^{1/\alpha}\}}\Big[\int_{B\cap \{y_d> t^{1/\alpha}\}}   t^{-\frac d\alpha}\Big(\f{t^{1/\alpha}+|x-y|}{t^{1/\alpha}}\Big)^{-d-\beta}D_\theta(x_d,t)D_\theta(y_d,t)|f(y)|dy\Big]^qdx\right\}^{\frac1q},\\
    I_3&= \left\{\int_{S_j(B)\cap \{x_d> t^{1/\alpha}\}}\Big[\int_{B\cap \{y_d\le t^{1/\alpha}\}}   t^{-\frac d\alpha}\Big(\f{t^{1/\alpha}+|x-y|}{t^{1/\alpha}}\Big)^{-d-\beta}D_\theta(x_d,t)D_\theta(y_d,t)|f(y)|dy\Big]^qdx\right\}^{\frac1q},\\
    \text{and}& \\
    I_4&= \left\{\int_{S_j(B)\cap \{x_d> t^{1/\alpha}\}}\Big[\int_{B\cap \{y_d> t^{1/\alpha}\}}  t^{-\frac d\alpha}\Big(\f{t^{1/\alpha}+|x-y|}{t^{1/\alpha}}\Big)^{-d-\beta}D_\theta(x_d,t)D_\theta(y_d,t)|f(y)|dy\Big]^qdx\right\}^{\frac1q}.
  \end{aligned}
  $$
  
  We first consider $I_1$. Here, we have
  \[
    D_\theta(x_d,t) = \Big(\f{x_d}{t^{1/\alpha}}\Big)^{-\theta }, \qquad
    D_\theta(y_d,t) = \Big(\f{y_d}{t^{1/\alpha}}\Big)^{-\theta }.
  \]
  By H\"older's inequality, 
  $$
  \begin{aligned}
    I_1&\le t^{-d/\alpha} \Big(1+\f{d(B,S_j(B))}{t^{1/\alpha}}\Big)^{-d-\beta}\|f\|_p\\
       & \ \ \ \times \Big[\int_{S_j(B)\cap  \{x_d\le t^{1/\alpha}\}}\Big(\f{x_d}{t^{1/\alpha}}\Big)^{-\theta q}dx\Big]^{1/q} \Big[\int_{B\cap \{y_d\le t^{1/\alpha}\}}\Big( \f{y_d}{t^{1/\alpha}}\Big)^{-\theta p'}dy\Big]^{1/p'}.
  \end{aligned}
  $$
  Using $d(B,S_j(B))\simeq 2^j r_B$,
  \begin{equation}\label{eq1 - proof of Theorem 2.2}
    \begin{aligned}
      \Big[\int_{S_j(B)\cap \{x_d\le t^{1/\alpha}\}}\Big(\f{x_d}{t^{1/\alpha}}\Big)^{-\theta q}dx\Big]^{1/q}
      &\le \Big[  \int_{2^jB'}\int_0^{t^{1/\alpha}} \Big(\f{x_d}{t^{1/\alpha}}\Big)^{-\theta q}dx_ddx' \Big]^{1/q}\\
      &\lesi t^{\f{1}{\alpha q}}|2^jB'|^{1/q}.
    \end{aligned}
  \end{equation}
  since $\theta q<1$, and
  \begin{equation}\label{eq2 - proof of Theorem 2.2}
    \begin{aligned}
      \Big[\int_{B\cap \{y_d\le t^{1/\alpha}\}}\Big( \f{y_d}{t^{1/\alpha}}\Big)^{-\theta p'}dy\Big]^{1/p'}
      &\lesi t^{\f{1}{\alpha p'}}|B'|^{1/p'},
    \end{aligned}
  \end{equation}
  we get, for some $c>0$,
  \begin{equation}
    \label{eq-I1}
    \begin{aligned}
      |S_j(B)|^{-1/q}\times I_1
      &\lesi t^{-\f{d}{\alpha} +\f{1}{\alpha q}+\f{1}{\alpha p'}} \Big(1+\f{2^jr_B}{t^{1/\alpha}}\Big)^{-d-\beta}|2^jB|^{-1/q}|2^jB'|^{1/q}|B'|^{1/p'}\|f\|_{L^p(B)}\\
      &= c t^{-\f{d}{\alpha} +\f{1}{\alpha q}+\f{1}{\alpha p'}} (2^jr_B)^{-1/q}r_B^{d-1/p'}\Big(1+\f{2^jr_B}{t^{1/\alpha}}\Big)^{-d-\beta} \Big(\fint_B|f|^p\Big)^{1/p}\\
      &\lesi  {\Big(\f{r_B}{t^{1/\alpha}}\Big)^{ d-1/p' }\Big(1+\f{t^{1/\alpha}}{2^jr_B}\Big)^{1/q} \Big(1+\f{2^jr_B}{t^{1/\alpha}}\Big)^{-d-\beta} \Big(\fint_B|f|^p\Big)^{1/p}}.
    \end{aligned}
  \end{equation}
  This concludes the estimate for $I_1$. We now bound $I_2$. Here, we have
  \[
    D_\theta(x_d,t) = \Big(\f{x_d}{t^{1/\alpha}}\Big)^{-\theta }, \quad D_\theta(y_d,t)= 1.
  \]
  Hence, using H\"older's inequality and \eqref{eq1 - proof of Theorem 2.2}, we have
  $$
  \begin{aligned}
    I_2&\lesi \left\{\int_{S_j(B)\cap \{x_d\le t^{1/\alpha}\}}\Big[\int_{B\cap \{y_d> t^{1/\alpha}\}} \Big( \f{x_d}{t^{1/\alpha}}\Big)^{-\theta} t^{-d/\alpha}\Big(\f{t^{1/\alpha}+|x-y|}{t^{1/\alpha}}\Big)^{-d-\beta}|f(y)|dy\Big]^qdx\right\}^{1/q}\\
       &\lesi   {t^{-\f{d}{\alpha}}}\Big(1+\f{d(B,S_j(B))}{t^{1/\alpha}}\Big)^{-d-\beta}\|f\|_{L^1(B)}\Big(\int_{S_j(B)\cap \{x_d\le t^{1/\alpha}\}}  \Big( \f{x_d}{t^{1/\alpha}}\Big)^{-\theta q}dx\Big)^{1/q}\\
       &\lesi  t^{-\f{d}{\alpha}+\f{1}{\alpha q}}\Big(1+\f{2^jr_B}{t^{1/\alpha}}\Big)^{-d-\beta}|2^jB'|^{1/q}|B|\Big(\fint_B|f|^p\Big)^{1/p}.
  \end{aligned}
  $$
  It follows that 
  \begin{equation}
    \label{eq-I2}
    \begin{aligned}
      |S_j(B)|^{-1/q}\times I_2
      &\lesi t^{-\f{d}{\alpha}+\f{1}{\alpha q}}\Big(1+\f{2^jr_B}{t^{1/\alpha}}\Big)^{-d-\beta}|2^jB|^{-1/q}|2^jB'|^{1/q}|B| \Big(\fint_B|f|^p\Big)^{1/p}\\
      &\lesi   {\Big(\f{r_B}{t^{1/\alpha}}\Big)^d\Big(1+\f{t^{1/\alpha}}{2^jr_B}\Big)^{1/q}\Big(1+\f{2^jr_B}{t^{1/\alpha}}\Big)^{-d-\beta} \Big(\fint_B|f|^p\Big)^{1/p}}.
    \end{aligned}
  \end{equation}
  
  Considering the term $I_3$, we have
  \[
    D_\theta(x_d,t) = 1, \ \ \ D_\theta(y_d,t)= \Big(\f{y_d}{t^{1/\alpha}}\Big)^{-\theta }.
  \]
  Hence, 
  $$
  \begin{aligned}
    I_3&\lesi \left\{\int_{S_j(B)\cap \{x_d> t^{1/\alpha}\}}\Big[\int_{B\cap \{y_d\le t^{1/\alpha}\}}  t^{-\frac d\alpha}\Big(\f{t^{1/\alpha}+|x-y|}{t^{1/\alpha}}\Big)^{-d-\beta}\Big( \f{y_d}{t^{1/\alpha}}\Big)^{-\theta }|f(y)|dy\Big]^qdx\right\}^{\frac1q}\\
       &\lesi  t^{-\f{d}{\alpha}}\Big(1+\f{d(B,S_j(B))}{t^{1/\alpha}}\Big)^{-d-\beta}|2^jB|^{1/q} \int_{B\cap \{y_d\le t^{1/\alpha}\}}\Big( \f{y_d}{t^{1/\alpha}}\Big)^{-\theta }|f(y)|dy.
  \end{aligned}
  $$
  By H\"older's inequality and \eqref{eq2 - proof of Theorem 2.2},
  $$
  \begin{aligned}
    \int_{B\cap \{y_d\le t^{1/\alpha}\}}\Big( \f{y_d}{t^{1/\alpha}}\Big)^{-\theta }|f(y)|dy&\lesi  \Big(\int_{B\cap \{y_d\le t^{1/\alpha}\}}\Big( \f{y_d}{t^{1/\alpha}}\Big)^{-\theta p' }dy\Big)^{1/p'}\|f\|_{L^p(B)}\\
                                                                                           &\lesi t^{\f{1}{p'\alpha}} |B'|^{1/p'}\|f\|_{L^p(B)}.
  \end{aligned}
  $$
  Consequently,
  \begin{equation}
    \label{eq-I3}
    \begin{aligned}
      |S_j(B)|^{-1/q}\times I_3
      &\lesi \ t^{-\f{d}{\alpha}+\f{1}{\alpha p'}}|B'|^{1/p'}|B|^{1/p} \Big(1+\f{2^jr_B}{t^{1/\alpha}}\Big)^{-d-\beta} \Big(\fint_B|f|^p\Big)^{1/p}\\
      & {\simeq}  \  \Big(\f{r_B}{t^{1/\alpha}}\Big)^{d-1/p'} \Big(1+\f{2^jr_B}{t^{1/\alpha}}\Big)^{-d-\beta} \Big(\fint_B|f|^p\Big)^{1/p}.
    \end{aligned}
  \end{equation}
  
  It remains to estimate the term $I_4$. Since in this case 
  \[
    D_\theta(x_d,t)= D_\theta(y_d,t)= 1,
  \]
  we have,
  \[
    \begin{aligned}
      I_4&\lesi \left\{\int_{S_j(B)}\Big[\int_{B}  t^{-d/\alpha}\Big(\f{t^{1/\alpha}+|x-y|}{t^{1/\alpha}}\Big)^{-d-\alpha}|f(y)|dy\Big]^qdx\right\}^{1/q}\\
         &\lesi t^{-d/\alpha} \Big(1+\f{d(B,S_j(B))}{t^{1/\alpha}}\Big)^{-d-\beta}|S_j(B)|^{1/q}\|f\|_{L^1(B)}.
    \end{aligned}
  \]
  Thus, by H\"older's inequality, we obtain
  \begin{equation}
    \label{eq-I4}
    \begin{aligned}
      |S_j(B)|^{-1/q}\times I_4
      &\lesi t^{-\f{d}{\alpha}}|B| \Big(1+\f{2^jr_B}{t^{1/\alpha}}\Big)^{-d-\beta} \Big(\fint_B|f|^p\Big)^{1/p}\\
      &\lesi   \Big(\f{r_B}{t^{1/\alpha}}\Big)^{d} \Big(1+\f{2^jr_B}{t^{1/\alpha}}\Big)^{-d-\beta} \Big(\fint_B|f|^p\Big)^{1/p}.
    \end{aligned}
  \end{equation}
  This completes the proof of \eqref{eq1-Tt}.
\end{proof}

In the following, we denote the operator norm of a bounded linear operator $T:L^p(\R_+^d)\to L^q(\R_+^d)$ by $\|T\|_{p\to q}$.

\begin{lem}\label{lem-Lp boundedness of Tt}
  Let $\theta<1/2$ and $\{T_t\}_{t>0}$ be as in Theorem \ref{thm-Tt}. Then for $(p_\theta)'<p\le q<p_\theta$ we have
  \[
    \|T_t\|_{p\to q} \lesi t^{-\f{d}{\alpha}(\f{1}{p}-\f{1}{q})}
  \]
  for all $t>0$.
\end{lem}

\begin{proof}
  For $\theta\le0$, we use
  \begin{align*}
    |T_t(x,y)| \le C t^{-\frac d\alpha}\Big(\f{t^{1/\alpha}+|x-y|}{t^{1/\alpha}}\Big)^{-d-\beta}
  \end{align*}
  and that, for $1\le p\le q\le \infty$, 
  \begin{align}
    \label{eq:lem1-Tt}
    \begin{split}
      \left\|\int_{\R^N} t^{-\frac N\alpha}\Big(\f{t^{1/\alpha}+|x-y|}{t^{1/\alpha}}\Big)^{-N-\beta} f(y)\,dy\right\|_{L^q(\R^N,dx)}
      \lesi t^{-\frac{N}{\alpha}(\frac1p-\frac1q)}\|f\|_{L^p(\R^N)}
    \end{split}
  \end{align}
  holds for all $N\in\N$ and $t>0$. This concludes the case $\theta \le 0$.
  
  Thus, from now on, we consider $\theta \in (0,1/2)$. We use the same notations as in the proof of Theorem \ref{thm-Tt}. Similarly as in that proof, we write
  \[
    \begin{aligned}
      \|T_tf\|_{q}
      &\le \left\{\int_{\Rd_+}\Big[\int_{\Rd_+}  t^{-d/\alpha}\Big(\f{t^{1/\alpha}+|x-y|}{t^{1/\alpha}}\Big)^{-d-\alpha}D_\theta(x_d,t)D_\theta(y_d,t)|f(y)|dy\Big]^qdx\right\}^{1/q}\\
      &\lesi I_1 + I_2 + I_3 + I_4,
    \end{aligned}
  \]
  where
  $$
  \begin{aligned}
    I_1&= \left\{\int_{\Rd_+\cap \{x_d\le t^{1/\alpha}\}}\Big[\int_{\Rd_+\cap \{y_d\le t^{1/\alpha}\}}  t^{-\frac d\alpha}\Big(\f{t^{1/\alpha}+|x-y|}{t^{1/\alpha}}\Big)^{-d-\beta}D_\theta(x_d,t)D_\theta(y_d,t)|f(y)|dy\Big]^qdx\right\}^{\frac1q},\\
    I_2&= \left\{\int_{\Rd_+\cap \{x_d\le t^{1/\alpha}\}}\Big[\int_{\Rd_+\cap \{y_d> t^{1/\alpha}\}}   t^{-\frac d\alpha}\Big(\f{t^{1/\alpha}+|x-y|}{t^{1/\alpha}}\Big)^{-d-\beta}D_\theta(x_d,t)D_\theta(y_d,t)|f(y)|dy\Big]^qdx\right\}^{\frac1q},\\
    I_3&= \left\{\int_{\Rd_+\cap \{x_d> t^{1/\alpha}\}}\Big[\int_{\Rd_+\cap \{y_d\le t^{1/\alpha}\}}   t^{-\frac d\alpha}\Big(\f{t^{\frac 1\alpha}+|x-y|}{t^{1/\alpha}}\Big)^{-d-\beta}D_\theta(x_d,t)D_\theta(y_d,t)|f(y)|dy\Big]^qdx\right\}^{\frac1q},\\
    \text{and}& \\
    I_4&= \left\{\int_{\Rd_+\cap \{x_d> t^{1/\alpha}\}}\Big[\int_{\Rd_+\cap \{y_d> t^{1/\alpha}\}}  t^{-\frac d\alpha}\Big(\f{t^{1/\alpha}+|x-y|}{t^{1/\alpha}}\Big)^{-d-\beta}D_\theta(x_d,t)D_\theta(y_d,t)|f(y)|dy\Big]^qdx\right\}^{\frac1q}.
  \end{aligned}
  $$
  We only estimate $I_1$ and $I_2$ since $I_3$ and $I_4$ can be treated similarly as the corresponding terms in the proof of Theorem \ref{thm-Tt} with minor modifications. To estimate $I_1$, we note
  \[
    D_\theta(x_d,t) = \Big(\f{x_d}{t^{1/\alpha}}\Big)^{-\theta }, \quad D_\theta(y_d,t)= \Big(\f{y_d}{t^{1/\alpha}}\Big)^{-\theta }.
  \]
  Then, applying \eqref{eq:lem1-Tt} with $N=d-1$ and Minkowski's inequality,
  \[
    \begin{aligned}
      I_1&\lesi t^{-\f{d-1}{\alpha}(\f{1}{p}-\f{1}{q})}\left\{\int_0^{ t^{1/\alpha}}\Big[\int_0^{t^{1/\alpha}}  t^{-1/\alpha}\Big(\f{x_d}{t^{1/\alpha}}\Big)^{-\theta}\Big(\f{y_d}{t^{1/\alpha}}\Big)^{-\theta}\|f(\cdot,y_d)\|_{L^p(\R^{d-1})} dy_d\Big]^qdx_d\right\}^{1/q}\\
         &\lesi t^{-\f{d-1}{\alpha}(\f{1}{p}-\f{1}{q})-\f{1}{\alpha}}\Big[\int_0^{ t^{1/\alpha}}\Big(\f{x_d}{t^{1/\alpha}}\Big)^{-\theta q}dx_d\Big]^{1/q}  \int_0^{t^{1/\alpha}}   \Big(\f{y_d}{t^{1/\alpha}}\Big)^{-\theta}\|f(\cdot,y_d)\|_{L^p(\R^{d-1})} dy_d\\
         &\lesi t^{-\f{d-1}{\alpha}(\f{1}{p}-\f{1}{q})-\f{1}{\alpha}}t^{\f{1}{\alpha q}} \int_0^{t^{1/\alpha}}   \Big(\f{y_d}{t^{1/\alpha}}\Big)^{-\theta}\|f(\cdot,y_d)\|_{L^p(\R^{d-1})} dy_d.
    \end{aligned}
  \]
  By H\"older's inequality,
  \[
    \begin{aligned}
      \int_0^{t^{1/\alpha}}   \Big(\f{y_d}{t^{1/\alpha}}\Big)^{-\theta}\|f(\cdot,y_d)\|_{L^p(\R^{d-1})} dy_d
      &\lesi \|f\|_{L^p(\Rd_+)}\Big[\int_0^{t^{1/\alpha}} \Big(\f{y_d}{t^{1/\alpha}}\Big)^{-\theta p'} dy_d\Big]^{\frac{1}{p'}}
        \lesi t^{\f{1}{\alpha p'}}\|f\|_{L^p(\Rd_+)}.
    \end{aligned}
  \]
  Therefore,
  \[
    \begin{aligned}
      I_1 \lesi t^{-\f{d}{\alpha}(\f{1}{p}-\f{1}{q})}\|f\|_{L^p(\Rd_+)}.
    \end{aligned}
  \]
  We now consider $I_2$. Here, we have
  \[
    D_\theta(x_d,t) = \Big(\f{x_d}{t^{1/\alpha}}\Big)^{-\theta }, \quad D_\theta(y_d,t)= 1.
  \]
  By \eqref{eq:lem1-Tt} with $N=d-1$, we have, for any $\beta'\in(0, \beta)$,
  \[
    \begin{aligned}
      I_2&\lesi t^{-\f{d-1}{\alpha}(\f{1}{p}-\f{1}{q})}\left\{\int_0^{ t^{1/\alpha}}\Big[\int_0^{t^{1/\alpha}}  t^{-\frac1\alpha}\Big(\f{x_d}{t^{1/\alpha}}\Big)^{-\theta}\Big(\f{t^{1/\alpha}+|x_d-y_d|}{t^{1/\alpha}}\Big)^{-1-\beta'} \|f(\cdot,y_d)\|_{L^p(\R^{d-1})} dy_d\Big]^qdx_d\right\}^{\frac1q}.
    \end{aligned}
  \]
  By H\"older's inequality and \eqref{eq:lem1-Tt} with $N=1$,
  \[
    \begin{aligned}
      & \int_0^{t^{1/\alpha}}   \Big(\f{t^{1/\alpha}+|x_d-y_d|}{t^{1/\alpha}}\Big)^{-1-\beta'}\|f(\cdot,y_d)\|_{L^p(\R^{d-1})} dy_d\\
      & \quad \lesi \|f\|_{L^p(\Rd_+)}\Big[\int_0^{\vc}   \Big(\f{t^{1/\alpha}+|x_d-y_d|}{t^{1/\alpha}}\Big)^{-(1+\beta')p'} dy_d\Big]^{1/p'}
        \lesi t^{\f{1}{\alpha p'}}\|f\|_{L^p(\Rd_+)}.
    \end{aligned}
  \] 
  Therefore,
  \[
    \begin{aligned}
      I_2&\lesi t^{-\f{d-1}{\alpha}(\f{1}{p}-\f{1}{q}) -\f{1}{\alpha}+\f{1}{\alpha p'} } \|f\|_{L^p(\Rd_+)}\Big[\int_0^{t^{1/\alpha}}  t^{-\frac1\alpha}\Big(\f{x_d}{t^{1/\alpha}}\Big)^{-\theta q} dx_d\Big]^{1/q}\\
         &\lesi t^{-\f{d-1}{\alpha}(\f{1}{p}-\f{1}{q}) -\f{1}{\alpha}+\f{1}{\alpha p'} +\f{1}{\alpha q}} \|f\|_{L^p(\Rd_+)}
           \lesi t^{-\f{d}{\alpha}(\f{1}{p}-\f{1}{q})} \|f\|_{L^p(\Rd_+)}.
    \end{aligned}
  \]
  Similarly, we also have
  \[
    I_3+I_4 \lesi t^{-\f{d}{\alpha}(\f{1}{p}-\f{1}{q})} \|f\|_{L^p(\Rd_+)}.
  \]
  This completes our proof.
\end{proof}

The following two lemmas will be important to prove Theorem~\ref{density}.

\begin{lem}\label{lem-boundedness of LsetL}
  Let $\alpha\in (0,2]$ and let $\lambda \ge 0$ when $\alpha<2$ and $\lambda\ge -1/4$ when $\alpha=2$. Let  $\sigma$ be defined by \eqref{eq-sigma}. Let $\gamma\in (0,1]$. Then for $1<p<\vc$ we have
  \begin{align}
    \label{eq:lem-boundedness of LsetL1}
    \|L_\lambda^\gamma e^{-tL_\lambda} \|_{p\to p} \lesi  t^{-\gamma},
  \end{align}
  and
  \begin{align}
    \label{eq:lem-boundedness of LsetL2}
    \|L_\lambda^{-\gamma} \big[e^{-tL_\lambda}-e^{-sL_\lambda}\big] \|_{p\to p}\lesi (t^\gamma -s^\gamma)
  \end{align}
  for all $t\ge s>0$.
\end{lem}

\begin{rem}
  \label{rem-lem-boundedness of LsetL}
  Let $\alpha\in(0,2)$, $\lambda\in[\lambda_*,0)$ and assume that $\me{-tL_\lambda}(x,y)$ satisfies the bound in \eqref{eq:thm-heatkernelLa- alpha < 2} with $\sigma$ defined by \eqref{eq-sigma}. Then Lemma \ref{lem-boundedness of LsetL}  remains valid for $p\in(1,\infty)$ obeying $\f{1}{1+\sigma}=(p_{-\sigma})'<p< p_{-\sigma}=-\f{1}{\sigma}$. This follows by the same arguments as in the proof below, taking into account Remark~\ref{rem-thm-ptk}.
\end{rem}

\begin{proof}
  Since
  \[
    \La^{\gamma}=\f{1}{\Gamma(1-\gamma)}\int_{0}^\vc u^{1-\gamma} \La e^{-u\La}\f{du}{u},
  \]
  we get
  \[
    \La^{\gamma}e^{-tL_\lambda}=\f{1}{\Gamma(1-\gamma)}\int_{0}^\vc \f{u^{1-\gamma}}{u+t} (u+t)\La e^{-(u+t)\La}\f{du}{u}.
  \]
  From Lemma \ref{lem-Lp boundedness of Tt} and Proposition \ref{thm-ptk},  $\|(u+v)L_\lambda e^{-(u+v)L_\lambda}\|_{p\to p}\lesi 1$. Therefore,
  \[
    \begin{aligned}
      \|\La^{\gamma}e^{-tL_\lambda}\|_{p\to p}&\lesi  \int_{0}^\vc \f{u^{1-\gamma}}{u+t} \f{du}{u} \lesi t^{-\gamma}.
    \end{aligned}
  \]
  This shows \eqref{eq:lem-boundedness of LsetL1}.
  To show \eqref{eq:lem-boundedness of LsetL2}, we use, for $\gamma\in (0,1)$,
  \[
    L_\lambda^{-\gamma} = \f{1}{\Gamma(\gamma)}\int_0^\vc u^\gamma e^{-uL_\lambda}\f{du}{u}
  \]
  and Duhamel's formula
  \[
    e^{-tL_\lambda}-e^{-sL_\lambda} =-\int_{s}^t vL_\lambda e^{-vL_\lambda}\f{dv}{v}
  \]
  to write
  \[
    L_\lambda^{-\gamma} \big[e^{-tL_\lambda}-e^{-sL_\lambda}\big] =-\f{1}{\Gamma(\gamma)}\int_s^t\int_0^\vc \f{vu^\gamma}{u+v} (u+v)L_\lambda e^{-(u+v)L_\lambda}\f{du}{u}\f{dv}{v}.
  \]
  This, together with $\|(u+v)L_\lambda e^{-(u+v)L_\lambda}\|_{p\to p}\lesi 1$, implies
  \[
    \begin{aligned}
      \big\|L_\lambda^{-\gamma} \big[e^{-tL_\lambda}-e^{-sL_\lambda}\big]\big\|_{p\to p}
      &\lesi  \int_s^t\int_0^\vc \f{vu^\gamma}{u+v} \f{du}{u}\f{dv}{v}\\
      &\lesi  \int_s^t\int_0^\vc \f{vu^\gamma}{u+v} \f{du}{u}\f{dv}{v}
      &\lesi  \int_s^t v^\gamma  \f{dv}{v}\simeq t^\gamma -s^\gamma.
    \end{aligned}
  \]
  This concludes the proof.
\end{proof}

\begin{lem}
  \label{Calderon-reproducing formula}
  Let $\alpha\in (0,2]$ and let $\lambda \ge 0$ when $\alpha<2$ and $\lambda\ge -1/4$ when $\alpha=2$. Let $\sigma$ be defined by \eqref{eq-sigma}. Then for $f\in L^p(\Rd_+), 1<p<\vc$, we have
  \begin{enumerate}[{\rm (a)}]
  \item $\displaystyle \lim_{t\to 0} \|(I-e^{-tL_\lambda})f\|_{p} =0 $ and
  \item $\displaystyle \lim_{t\to \vc} \|e^{-tL_\lambda}f\|_{p} =0 $.
  \end{enumerate}
\end{lem}

\begin{rem}
  \label{rem-lem-Calderon-reproducing formula}
  Let $\alpha\in(0,2)$, $\lambda\in[\lambda_*,0)$ and assume that $\me{-tL_\lambda}(x,y)$ satisfies the bound in \eqref{eq:thm-heatkernelLa- alpha < 2} with $\sigma$ defined by \eqref{eq-sigma}. Then Lemma \ref{Calderon-reproducing formula}  remains valid for $p\in(1,\infty)$ obeying $\f{1}{1+\sigma}=(p_{-\sigma})'<p< p_{-\sigma}=-\f{1}{\sigma}$. This follows by the same arguments as in the proof below, taking into account Remark \ref{rem-thm-ptk}.
\end{rem}

\begin{proof}
  (a) It suffices to prove for $p\le 2$ since the case $p> 2$ can be obtained by duality. Since $\|e^{-tL_\lambda}\|_{p\to p}\lesi 1$ for all $t$ and $C^\vc_c(\Rd_+)$ is dense in $L^p$, it suffices to prove the lemma for $f\in C^\vc_c(\Rd_+)$. Assume that $\supp f \subset B$ for some ball $B\subset \Rd_+$. By
  \[ 
    I-e^{-tL_\lambda} =-\int_0^t L_\lambda e^{-sL_\lambda} ds,
  \]
  we have
  \[
    \|(I-e^{-tL_\lambda})f\|_{p}\le \Big\|\int_0^t L_\lambda e^{-sL_\lambda}f ds\Big\|_{L^p(4B)} +\Big\|\int_0^t L_\lambda e^{-sL_\lambda}f ds\Big\|_{L^p(\Rd_+\backslash 4B)}.
  \]
  By H\"older's inequality,
  \[
    \begin{aligned}
      \Big\|\int_0^t L_\lambda e^{-sL_\lambda}f ds\Big\|_{L^p(4B)}\le |4B|^{\f{2-p}{2p}}\Big\|\int_0^t L_\lambda e^{-sL_\lambda}f ds\Big\|_{L^2(4B)}.
    \end{aligned}
  \]
  Applying Minkowski's inequality,
  \[
    \begin{aligned}
      \|(I-e^{-tL_\lambda})f\|_{L^p(4B)}
      &\le |4B|^{\f{2-p}{2p}}\Big\| \int_0^t  e^{-sL_\lambda}L_\lambda f ds\Big\|_{L^2(4B)}
      \le t|4B|^{\f{2-p}{2p}}\|L_\lambda f\|_{2}.
    \end{aligned}
  \]
  Since $L_\lambda f \in L^2(\Rd_+)$ by \cite[Lemma~15]{FrankMerz2023}, we have
  \[
    \lim_{t\to 0}\|(I-e^{-tL_\lambda})f\|_{L^p(4B)} = 0.
  \]
  For the second term we write
  \[
    \begin{aligned}
      \Big\|\int_0^t L_\lambda e^{-sL_\lambda}f ds\Big\|_{L^p(\Rd_+\backslash 4B)}
      &\le  \int_0^t \|sL_\lambda e^{-sL_\lambda}f\|_{L^p(\Rd_+\backslash 4B)} \f{ds}{s} \\
      &\le \int_0^t\sum_{j\ge 2} \|sL_\lambda e^{-sL_\lambda}f\|_{L^p(S_j(B))} \f{ds}{s}.
    \end{aligned}
  \]
  Then by Proposition \ref{thm-ptk} and Theorem \ref{thm-Tt}, we have for $t<r_B^{\alpha}$,
  \[
    \begin{aligned}
      \sum_{j\ge 2} \|L_\lambda e^{-sL_\lambda}f\|_{L^p(S_j(B))}
      &\le\sum_{j\ge 2} |2^jB|^{1/p}   \Big(\f{r_B}{s^{1/\alpha}}\Big)^{d} \Big(\f{2^jr_B}{s^{1/\alpha}}\Big)^{-d-\beta} \Big(\fint_B|f|^r\Big)^{1/r}\\
      &\le C(|B|,p,r,\|f\|_r) s^{\beta/\alpha}.
    \end{aligned}
  \]
  It follows that
  \[
    \begin{aligned}
      \Big\|\int_0^t L_\lambda e^{-sL_\lambda}f ds\Big\|_{L^p(\Rd_+\backslash \int_0^t  4B)}&\lesi \int_0^t s^{\beta/\alpha}\f{ds}{s} \simeq t^{\beta/\alpha}\to 0 \ \  \text{as $t\to 0$}.
    \end{aligned}
  \]
  
  (b) Similarly as in (a), we assume $f\in C^\vc_c(\Rd_+)$ without loss of generality. Fix $1<r<p$. Then,
  \[
    \|e^{-tL_\lambda}f\|_{p}\lesi t^{-\f{d}{\alpha}(\frac{1}{r}-\f{1}{p})}\|f\|_r,
  \]
  which implies
  \[
    \displaystyle \lim_{t\to \vc} \|e^{-tL_\lambda}f\|_{p} =0. 
  \]
  This completes our proof.
\end{proof}

\section{Proof of the square function estimates (Theorem~\ref{squarefunctions})}
\label{s:squarefunctions}

For $\gamma>0$, recall the square function
\begin{align*}
  (S_{\La,\gamma}f)(x) := \Big(\int_0^\vc|(t\La)^{\gamma}e^{-t\La}f|^2\f{dt}{t}\Big)^{1/2}.
\end{align*}

We recall two criteria for singular integrals to be bounded on Lebesgue spaces. These will play an important role in the proof of the boundedness of the square functions. The first theorem gives a criterion on the boundedness on $L^p(\Rd)$ spaces with $p\in (1,2)$, while the second one covers the range $p>2$. The latter also involves, for $r>0$, the Hardy--Littlewood maximal operator $\mathcal{M}_r$, which acts as
$$
(\mathcal{M}_rf)(x) := \sup_{B\ni x}\Big(\f{1}{|B|}\int_B|f(y)|^r\,dy\Big)^{1/r}, \quad x\in \mathbb{R}^d.
$$
Here the supremum is taken over all balls $B\subseteq\R^d$ containing $x$.

\begin{thm}[{\cite[Theorem~1.1]{Auscher2007}}]
  \label{thm1-Auscher}
  Let $1\le p_0< 2$ and let $T$ be a sublinear operator which is bounded on $L^2(\Rd)$.  Assume that there exists a family of operators $\{\mathcal{A}_t\}_{t>0}$ satisfying that for $j\ge 2$ and every ball $B\subseteq\R^d$
  \begin{equation}\label{eq1-BZ}
    \Big(\fint_{S_j(B)}|T(I-\mathcal{A}_{r_B})f|^{2}\Big)^{1/2}\le
    \alpha(j)\Big(\fint_B |f|^{p_0} \Big)^{1/p_0},
  \end{equation}
  and
  \begin{equation}\label{eq1-BZ-bis}
    \Big(\fint_{S_j(B)}|\mathcal{A}_{r_B}f|^{2} \Big)^{1/2}\le
    \alpha(j)\Big(\fint_B |f|^{p_0} \Big)^{1/p_0},
  \end{equation}
  for all $f$ supported in $B$. If $\sum_{j\geq2} \alpha(j)2^{jd}<\vc$, then $T$ is bounded on $L^p(\mathbb{R}^d)$ for all $p\in (p_0,2)$.
\end{thm}

\begin{thm}[{\cite[Theorem~1.2]{Auscher2007}}]
  \label{thm2-Auscher}
  Let $2< q_0\le \infty.$  Let $T$ be a bounded sublinear operator on $L^{2}(\mathbb{R}^d)$. Assume that there exists  a family of operators $\{\mathcal{A}_t\}_{t>0}$  satisfying that
  \begin{eqnarray}\label{e1-Martell}
    \Big( \fint_{B} \big| T(I-\mathcal{A}_{r_B})f\big|^{2}dx\Big)^{1/2} \le
    C \mathcal{M}_{2}(f)(x),
  \end{eqnarray}
  and
  \begin{eqnarray}\label{e2-Martell}
    \Big( \fint_{B} \big| T\mathcal{A}_{r_B}f\big|^{q_0}dx\Big)^{1/q_0} \le
    C \mathcal{M}_{2}(Tf)(x),
  \end{eqnarray}
  \noindent
  for all balls $B\subseteq\R^d$ with radius $r_B$, all $f \in C^{\infty}_c(\mathbb{R}^d) $ and all $x\in B$. Then $T$ is bounded on $L^p(\mathbb{R}^d)$ for all
  $2<p<q_0$.
\end{thm}

See \cite[p.~2]{Auscher2007} for references, where these theorems were originally proved. Using Theorems \ref{thm1-Auscher}--\ref{thm2-Auscher} with the $T=S_{L_\lambda,\gamma}$ on the left-hand side of \eqref{eq:defsquarefunction}, we give the

\begin{proof}[Proof of Theorem~\ref{squarefunctions}]
  We begin with some reductions.
  First, Formula~\eqref{eq:squarefunctions2} immediately follows from \eqref{eq:squarefunctions1}.
  Since the statement for $p=2$ follows from the spectral theorem, it suffices to consider $p\in(1,\infty)\setminus\{2\}$.
  By a duality argument, the lower bound in \eqref{eq:squarefunctions1} follows from the upper bound; see the end of this proof for an argument. Thus, it suffices to prove the upper bound.  
  To that end, we fix $\beta \in (0,\alpha)$ such that $\beta/\alpha>1-\gamma$.
  We now distinguish between $p<2$ and $p>2$.
  
  \bigskip
   
  \textbf{Step 1: The proof of the $L^p$-boundedness for $p<2$}
  
  \medskip
  
  Fix $1<p< 2$. We use Theorem \ref{thm1-Auscher} with $T=S_{L_\lambda,\gamma}$ and
  \[
    \mathcal A_{r_B} = I-(I-e^{-r_B^\alpha \La})^m, \ \ \  m>\beta +1.
  \]
  Thus, it suffices to prove
  \begin{equation}\label{eq1-thmSquareFunctions}
    \Big(\fint_{S_j(B)}|S_{\La,\gamma}(I-\mathcal A_{r_B})f(x)|^{2}dx\Big)^{1/2}\lesi 2^{-(d+\beta)j}\Big(\fint_B|f(x)|^{p}dx\Big)^{1/p}
  \end{equation}
  and
  \begin{equation}\label{eq2-thmSquareFunctions}
    \Big(\fint_{S_j(B)}| \mathcal A_{r_B}f(x)|^{2}dx\Big)^{1/2}\lesi 2^{-(d+\beta)j}\Big(\fint_B|f(x)|^{p}dx\Big)^{1/p}
  \end{equation}
  for $j\ge 2$ and for every function $f$ supported in $B$. Since 
  \[
    \mathcal A_{r_B} =\sum_{k=1}^m C^m_ke^{-kr_B^\alpha\La}
  \]
  for combinatorial factors $C_k^m\in\R$ with $C_0^m=1$, the estimate \eqref{eq2-thmSquareFunctions} follows directly from Theorem \ref{thm-heatkernelLa- alpha < 2} and Theorem \ref{thm-Tt}. It remains to prove \eqref{eq1-thmSquareFunctions}. To do this, we bound
  \begin{equation}
    \label{eq1-squarefunctions}
    \begin{aligned}
      \Big(\fint_{S_j(B)}|S_{\La,\gamma}&(I-e^{-r_B^\alpha\La})^mf|^{2}dx\Big)^{1/2}\\
                                        &\lesi \Big(\int_0^{r_B^\alpha}\left\|(t\La)^{\gamma}e^{-t\La}(I-e^{-r_B^\alpha\La})^mf\right\|_{L^{2}(S_j(B),\f{dx}{|S_j(B)|})}^2\f{dt}{t}\Big)^{1/2}\\
                                        & \quad+\Big(\int_{r_B^\alpha}^\vc\left\|(t\La)^{\gamma}e^{-t\La}(I-e^{-r_B^\alpha\La})^mf\right\|_{L^{2}(S_j(B),\f{dx}{|S_j(B)|})}^2\f{dt}{t}\Big)^{1/2}\\
                                        &=:E_1+E_2.
    \end{aligned}
  \end{equation}
  We first estimate $E_1$. Using
  $$
  \La^{\gamma}=\f{1}{\Gamma(1-\gamma)}\int_{0}^\vc u^{1-\gamma} \La e^{-u\La}\f{du}{u}.
  $$
  
  Using this and  Minkowski's inequality, we have 
  $$
  \begin{aligned}
    E_1&\lesi \Big(\int_0^{r_B^\alpha}\Big[\int_{0}^{r_B^\alpha}\Big(\f{u}{t}\Big)^{1-\gamma}\left\|t\La e^{-(t+u)\La}(I-e^{-r_B^\alpha\La})^mf\right\|_{L^{2}(S_j(B),\f{dx}{|S_j(B)|})}\f{du}{u}\Big]^2\f{dt}{t}\Big)^{1/2}\\
       & \ \ + \Big(\int_0^{r_B^\alpha}\Big[\int_{r_B^\alpha}^\vc\Big(\f{u}{t}\Big)^{1-\gamma}\Big\|t\La e^{-(t+u)\La}(I-e^{-r_B^\alpha\La})^mf\Big\|_{L^{2}(S_j(B),\f{dx}{|S_j(B)|})}\f{du}{u}\Big]^2\f{dt}{t}\Big)^{1/2}\\
       &=:E_{11}+E_{12}.
  \end{aligned}
  $$
  Using the identity 
  $$
  (I-e^{-r_B^\alpha\La})^m =\sum_{k=0}^m  (-1)^kC^m_ke^{-kr_B^\alpha\La},
  $$
  we have
  $$
  \begin{aligned}
    & E_{11} \lesi \Big(\int_0^{r_B^\alpha}\Big[\int_{0}^{r_B^\alpha}\Big(\f{u}{t}\Big)^{1-\gamma}\f{t}{t+u}\Big\|(t+u)\La e^{-(t+u)\La}f\Big\|_{L^{2}(S_j(B),\f{dx}{|S_j(B)|})}\f{du}{u}\Big]^2\f{dt}{t}\Big)^{1/2}\\ & \ + \sum_{k=1}^m |C_k^m| \Big(\int_0^{r_B^\alpha}\Big[\int_{0}^{r_B^\alpha}\Big(\f{u}{t}\Big)^{1-\gamma}\f{t}{t+u+kr_B^\alpha}\Big\|(t+u+kr_B^\alpha)\La e^{-(t+u+kr_B^\alpha)\La}f\Big\|_{L^{2}(S_j(B),\f{dx}{|S_j(B)|})}\f{du}{u}\Big]^2\f{dt}{t}\Big)^{1/2}.
  \end{aligned}
  $$
  By Proposition \ref{thm-ptk} and Theorem \ref{thm-Tt}, together with $t+u\lesi r_B^{\alpha}$ and $(t+u+ kr_B^\alpha)^{1/\alpha}\simeq_k r_B$ for $u, t\in (0,r_B^\alpha]$ and $k\ge 1$, we get, for some $\tilde C_k^m\gtrsim |C_k^m|$,
  $$
  \begin{aligned}
    E_{11}&\lesi  \Big(\int_0^{r_B^\alpha}\Big[\int_{0}^{r_B^\alpha}\Big(\f{u}{t}\Big)^{1-\gamma}\f{t}{t+u}\Big(\f{r_B}{(t+u)^{1/\alpha}}\Big)^d\Big(\f{2^jr_B}{(t+u)^{1/\alpha}}\Big)^{-d-\beta} \|f\|_{L^p(B,\f{dx}{|B|})}\f{du}{u}\Big]^2\f{dt}{t}\Big)^{1/2}\\ 
          & \ \ + \sum_{k=1}^m \tilde C_k^m \Big(\int_0^{r_B^\alpha}\Big[\int_{0}^{r_B^\alpha}\Big(\f{u}{t}\Big)^{1-\gamma}\f{t}{r_B^\alpha} 2^{-j(d+\beta)}\|f\|_{L^p(B,\f{dx}{|B|})}\f{du}{u}\Big]^2\f{dt}{t}\Big)^{1/2}.
  \end{aligned}
  $$
  Using also $\beta/\alpha>1-\gamma$, we get
  $$
  \begin{aligned}
    E_{11}&\lesi   2^{-j(d+\beta)} \Big(\fint_B|f|^{p}dx\Big)^{1/p},
  \end{aligned}
  $$
  as desired. To estimate $E_{12}$, we use
  \begin{equation}\label{eq2-squarefunction}
    (I-e^{-r_B^\alpha\La})^m=\int_0^{r_B^\alpha}\dots \int_0^{r_B^\alpha} \La^me^{-(s_1+\dots+s_m)\La}d\vec{s},
  \end{equation}
  where $d\vec{s}=ds_1\dots ds_m$, and estimate
  $$
  \begin{aligned}
    E_{12}&\lesi \Big(\int_0^{r_B^\alpha}\Big[\int_{[0,r_B^\alpha]^m}\int^\vc_{r_B^\alpha}\Big(\f{u}{t}\Big)^{1-\gamma}\Big\|t\La^{m+1} e^{-(t+u+s_1+\ldots+s_m)\La}f\Big\|_{L^{2}(S_j(B),\f{dx}{|S_j(B)|})}\f{du}{u}d\vec{s}\Big]^2\f{dt}{t}\Big)^{1/2}.
  \end{aligned}
  $$
  In this case, $u\simeq t+u+s_1+\ldots+s_m\ge r_B^\alpha$. Hence, by Proposition \ref{thm-ptk} and Theorem \ref{thm-Tt}, we obtain
  \begin{equation}\label{eq-I12}
    \begin{aligned}
      E_{12}&\lesi  \Big(\int^{r_B^\alpha}_0\Big[\int_{[0,r_B^\alpha]^m}\int^\vc_{r_B^\alpha}\Big(\f{u}{t}\Big)^{1-\gamma} \f{t}{u^{m+1}}\Big(\f{r_B}{u^{1/\alpha}}\Big)^{d-1/p'}  \Big(1+\f{u^{1/\alpha}}{2^jr_B}\Big)^{1/2}\Big(1+\f{2^jr_B}{u^{1/\alpha}}\Big)^{-(d+\beta)}\\
            & \quad \quad \quad \quad \times  \|f \|_{L^{p}(B,\f{dx}{|B|})}\f{du}{u}d\vec{s}\Big]^2\f{dt}{t}\Big)^{1/2}.
    \end{aligned}
  \end{equation}
  We can see that
  $$
  \begin{aligned}
    \int^\vc_{r_B^\alpha}\Big(\f{u}{t}\Big)^{1-\gamma}& \f{t}{u^{m+1}}\Big(\f{r_B}{u^{1/\alpha}}\Big)^{d-1/p'} \Big(1+\f{u^{1/\alpha}}{2^jr_B}\Big)^{1/2}\Big(1+\f{2^jr_B}{u^{1/\alpha}}\Big)^{-(d+\beta)}\f{du}{u}\\
                                                      &\lesi  \int^\vc_{r_B^\alpha}\Big(\f{u}{t}\Big)^{1-\gamma} \f{t}{u^{m+1}}\Big(\f{r_B}{u^{1/\alpha}}\Big)^{d-1/2} \Big(\f{u^{1/\alpha}}{r_B}\Big)^{1/2}\Big(\f{2^jr_B}{u^{1/\alpha}}\Big)^{-(d+\beta)}\f{du}{u}\\	
                                                      &\lesi 2^{-j(d+\beta)}\Big(\f{r_B^\alpha}{t}\Big)^{1-\gamma}\f{t}{r_B^{\alpha(m+1)}},
  \end{aligned}
  $$
  for all $m>\beta+1$. Plugging this into \eqref{eq-I12} we obtain
  \[
    \begin{aligned}
      E_{12}&\lesi  2^{-j(d+\beta)}\|f\|_{L^{p}(B,\f{dx}{|B|})}\Big(\int^{r_B^\alpha}_0\Big[\int_{[0,r_B^\alpha]^m}\Big(\f{r_B^\alpha}{t}\Big)^{1-\gamma}\f{t}{r_B^{\alpha(m+1)}} d\vec{s}\Big]^2\f{dt}{t}\Big)^{1/2}\\
            &\lesi  2^{-j(d+\beta)} \Big(\fint_B|f|^{p}dx\Big)^{1/p}.
    \end{aligned}
  \]
  Collecting the estimates of $E_{11}$ and $E_{12}$, we have
  $$
  E_1\lesi 2^{-j(d+\beta)} \Big(\fint_B|f|^{p}dx\Big)^{1/p}.
  $$
  
  We now estimate $E_2$. To do this, we bound
  $$
  \begin{aligned}
    E_2&\lesi \Big(\int^\vc_{r_B^\alpha}\Big[\int_{0}^{r_B^\alpha}\Big(\f{u}{t}\Big)^{1-\gamma}\Big\|t\La e^{-(t+u)\La}(I-e^{-r_B^\alpha\La})^mf\Big\|_{L^{2}(S_j(B),\f{dx}{|S_j(B)|})}\f{du}{u}\Big]^2\f{dt}{t}\Big)^{1/2}\\
       & \ \ + \Big(\int^\vc_{r_B^\alpha}\Big[\int_{r_B^\alpha}^\vc\Big(\f{u}{t}\Big)^{1-\gamma}\Big\|t\La e^{-(t+u)\La}(I-e^{-r_B^\alpha\La})^mf\Big\|_{L^{2}(S_j(B),\f{dx}{|S_j(B)|})}\f{du}{u}\Big]^2\f{dt}{t}\Big)^{1/2}\\
       &=:E_{21}+E_{22}.
  \end{aligned}
  $$
  Similarly to \eqref{eq-I12}, 
  $$
  \begin{aligned}
    E_{21}&\lesi  \Big(\int_{r_B^\alpha}^\vc\Big[\int_{[0,r_B^\alpha]^m}\int_0^{r_B^\alpha}\Big(\f{u}{t}\Big)^{1-\gamma} \f{t}{t^{m+1}}\Big(\f{r_B}{t^{1/\alpha}}\Big)^{d-1/p'}  \Big(1+\f{t^{1/\alpha}}{2^jr_B}\Big)^{1/2}\Big(1+\f{2^jr_B}{t^{1/\alpha}}\Big)^{-(d+\beta)}\\
          & \quad \quad \quad \quad \times  \|f \|_{L^{p}(B,\f{dx}{|B|})}\f{du}{u}d\vec{s}\Big]^2\f{dt}{t}\Big)^{1/2},
  \end{aligned}
  $$
  Thus,
  \[
    E_{21}\lesi    2^{-j(d+\beta)} \Big(\fint_B|f|^{p}dx\Big)^{1/p}
  \]
  for all $m>\beta+1$. By the same reason,
  \[
    \begin{aligned}
      E_{22}&\lesi  \Big(\int_{r_B^\alpha}^\vc\Big[\int_{[0,r_B^\alpha]^m}\int^\vc_{r_B^\alpha}\Big(\f{u}{t}\Big)^{1-\gamma} \f{t}{(t+u)^{m+1}}\Big(\f{r_B}{(t+u)^{1/\alpha}}\Big)^{d-1/p'}  \Big(1+\f{(t+u)^{1/\alpha}}{2^jr_B}\Big)^{1/2}\\
            & \quad \quad \quad \quad \times  \Big(1+\f{2^jr_B}{(t+u)^{1/\alpha}}\Big)^{-(d+\beta)}
              \|f \|_{L^{p}(B,\f{dx}{|B|})}\f{du}{u}d\vec{s}\Big]^2\f{dt}{t}\Big)^{1/2}.
    \end{aligned}
  \]
  In addition,
  \[
    \begin{aligned}
      \int^\vc_{r_B^\alpha}&\Big(\f{u}{t}\Big)^{1-\gamma} \f{t}{(t+u)^{m+1}}\Big(\f{r_B}{(t+u)^{1/\alpha}}\Big)^{d-1/p'}  \Big(1+\f{(t+u)^{1/\alpha}}{2^jr_B}\Big)^{1/2}
                             \Big(1+\f{2^jr_B}{(t+u)^{1/\alpha}}\Big)^{-(d+\beta)} \f{du}{u}\\
                           &\lesi \int^\vc_{r_B^\alpha}\Big(\f{u}{t}\Big)^{1-\gamma} \f{t}{(t+u)^{m+1}}\Big(\f{r_B}{(t+u)^{1/\alpha}}\Big)^{d-1/2}  \Big(\f{(t+u)^{1/\alpha}}{r_B}\Big)^{1/2}\Big(\f{2^jr_B}{(t+u)^{1/\alpha}}\Big)^{-(d+\beta)} \f{du}{u}\\
                           &\lesi 2^{-j(d+\beta)}\int^\vc_{r_B^\alpha}\Big(\f{u}{t}\Big)^{1-\gamma} \f{t}{(t+u)^{m+1}}\Big(\f{r_B}{(t+u)^{1/\alpha}}\Big)^{-\beta-1}   \f{du}{u}\\
                           &\lesi 2^{-j(d+\beta)}\int^\vc_{0}\Big(\f{u}{t}\Big)^{1-\gamma} \f{t r_B^{-\beta-1}}{(t+u)^{m-\beta }} \f{du}{u}
                             \lesi 2^{-j(d+\beta)} \f{t r_B^{-\beta-1}}{t^{m-\beta }}
    \end{aligned}
  \]
  for all $m>\beta + 1$. Consequently,
  \[
    E_{22}\lesi  2^{-j(d+\beta)} \Big(\fint_B|f|^{p}dx\Big)^{1/p}.
  \]
  Combining the estimates for $E_{21}, E_{22}$, and $E_1$ yields
  $$
  \Big(\int_{S_j(B)}|S_{\La,\gamma}(I-e^{-r_B^\alpha\La})^mf|^{2}dx\Big)^{1/2}\lesi 2^{-j(d+\beta)} \Big(\fint_B|f|^{p}dx\Big)^{1/p}.
  $$
  This completes the proof of  \eqref{eq1-thmSquareFunctions}.
  
  \bigskip

  \textbf{Step 2: The proof of the $L^p$-boundedness for $p>2$}

  \medskip
  We use Theorem \ref{thm2-Auscher} with $T=S_{L_\lambda,\gamma}$ and $\mathcal{A}_{r_B}=I-(I-e^{-r_B^\alpha\La})^m$, $m>d+\beta + \alpha(1-\gamma)$. Thus, for any $q\in (2,\vc)$ it suffices to prove
  \begin{eqnarray}
    \label{e1-SLa}
    \Big( \fint_{B} \left| S_{\La,\gamma}(I-\mathcal{A}_{r_B})f\right|^{2}dx\Big)^{1/2} \le
    C \mathcal{M}_{2}(f)(x)
  \end{eqnarray}
  and
  \begin{eqnarray}
    \label{e2-SLa}
    \Big( \fint_{B} \big| S_{\La,\gamma}\mathcal{A}_{r_B}f\big|^{q}dx\Big)^{1/q} \le
    C \mathcal{M}_{2}(S_{\La,\gamma}f)(x)
  \end{eqnarray}
  for all balls $B$ with radius $r_B$, all $f \in C^{\infty}_c(\mathbb{R}^d) $ and all $x\in B$.
  To prove \eqref{e1-SLa}, we write
  $$
  \begin{aligned}
    \begin{aligned}
      \Big(\fint_{B}|S_{\La,\gamma}(I-e^{-r_B^\alpha\La})^mf|^{2}dx\Big)^{\frac12}
      &\le \sum_{j\ge 0}\Big(\fint_{B}|S_{\La,\gamma}(I-e^{-r_B^\alpha\La})^mf_j|^{2}dx\Big)^{\frac12}=:\sum_{j=0}^\vc F_j,
    \end{aligned}
  \end{aligned}
  $$
  where $f_j=f\one_{S_j(B)}$. For $j=0,1$, we use the $L^2$-boundedness of $S_{\La,\gamma}$ and $\mathcal{A}_{r_B}$ and obtain
  $$
  F_j\lesi \mathcal{M}_{2}(f)(x).
  $$
  Hence, it suffices to prove
  \begin{equation}\label{eq- Fj}
    F_j\lesi 2^{-j\beta}\Big(\fint_{S_j(B)}|f|^{2}dx\Big)^{1/2}
  \end{equation}
  for all $j\ge 2$. To do this, we write
  \[
    \begin{aligned}
      & \Big(\fint_{B}|S_{\La,\gamma}(I-e^{-r_B^\alpha\La})^mf_j|^{2}dx\Big)^{1/2} \\
      & \quad \le  \Big(\int_0^{r_B^\alpha}\left\|(t\La)^{\gamma}e^{-t\La}(I-e^{-r_B^\alpha\La})^mf_j\right\|_{L^{2}(B,\f{dx}{|B|})}^2\f{dt}{t}\Big)^{1/2}\\
      & \qquad +\Big(\int_{r_B^\alpha}^\vc\left\|(t\La)^{\gamma}e^{-t\La}(I-e^{-r_B^\alpha\La})^mf_j\right\|_{L^{2}(B,\f{dx}{|B|})}^2\f{dt}{t}\Big)^{1/2}.
    \end{aligned}
  \]
  At this stage, we can argue as in the proof of \eqref{eq1-thmSquareFunctions} in Step 1. However, in this case, we use \eqref{eq2-Tt} instead of \eqref{eq1-Tt}. By doing so, we arrive at the expression \eqref{eq- Fj}. As the proof follows a similar structure, we omit the details.
		
  It remains to prove \eqref{e2-SLa}. We first write
  $$
  \begin{aligned}
    & \Big(\int_B |S_{\La,\gamma}\mathcal A_{r_B}f(x)|^{q}dx\Big)^{\frac1q}
      =\Big[\int_B \Big( \int_0^\vc \Big|\sum_{k=1}^m C^m_ke^{-kr_B^\alpha\La}(t\La)^{\gamma}e^{-t\La}f(x)\Big|^2\f{dt}{t} \Big)^{\frac q2}dx\Big]^{\frac1q}\\
    & \quad \lesi\sum_{j\ge 0}\Big[\int_B \Big( \int_0^\vc \Big|\sum_{k=1}^m C^m_ke^{-kr_B^\alpha\La}[(t\La)^{\gamma}e^{-t\La}f\one_{S_j(B)}](x)\Big|^2\f{dt}{t} \Big)^{q/2}dx\Big]^{\frac1q}
  \end{aligned}
  $$
  which, along with Minkowski's inequality, Proposition \ref{thm-ptk}, and \eqref{eq2-Tt} in Theorem \ref{thm-Tt}, gives
  $$
  \begin{aligned}
    \Big(\fint_B &|S_{\La,\gamma}\mathcal A_{r_B} f(x)|^{q}dx\Big)^{1/q}\\
                 &\lesi\sum_{j\ge 0}  \Big( \int_0^\vc \Big\|\sum_{k=1}^me^{-kr_B^\alpha\La}[(t\La)^{\gamma}e^{-t\La}f\one_{S_j(B)}]\Big\|_{L^{q}(B,\f{dx}{|B|})}^2\f{dt}{t} \Big)^{1/2}\\
		&\lesi\sum_{j\ge 0} 2^{-j\beta} \left( \int_0^\vc \left\|(t\La)^{\gamma}e^{-t\La}f\right\|_{L^2(S_j(B),\f{dx}{|S_j(B)|})}^2\f{dt}{t} \right)^{1/2}\\
                 &\lesi\sum_{j\ge 0} 2^{-j\beta} \Big(\fint_{2^jB} |S_{\La,\gamma}f(x)|^{2}dx\Big)^{1/2}.
  \end{aligned}
  $$
  This implies \eqref{e2-SLa}. Hence the proof of Step 2 is completed.

  \bigskip
  
  Thus we have proved that the square function $S_{\La,\gamma}$ is bounded on $L^p(\Rd_+)$ for all $1<p<\vc$, i.e.,
  \[
    \|S_{\La,\gamma}f\|_p\lesi \|f\|_p.
  \]
  The reverse inequality follows from this and duality. We provide the details. By functional calculus, for any $g\in L^{p'}(\Rd_+)$, we have
  $$
  \begin{aligned}
    \int_{\Rd_+} f(x)g(x)dx&=c(\gamma)\int_{\Rd_+} \int_0^\vc(t\La)^{2\gamma}e^{-2t\La}f(x)g(x)\f{dt}{t}dx,
  \end{aligned}
  $$
  where $c(\gamma)= \int_0^\vc t^{2\gamma}e^{-2t}\f{dt}{t}$. By H\"older's inequality, we obtain
  $$
  \begin{aligned}
    \left|\int_{\Rd_+} f(x)g(x)dx\right|
    &=c(\gamma) \left|\int_{\Rd_+} \int_0^\vc(t\La)^{\gamma}e^{-t\La}f(x)(t\La)^{\gamma}e^{-t\La}g(x)\f{dt}{t}dx\right| \\
    &\lesi \int_{\Rd_+}S_{\La,\gamma}f(x)S_{\La,\gamma}g(x) dx\\
    &\lesi\|S_{\La,\gamma}f\|_{p}\|S_{\La,\gamma}g\|_{p'}.
  \end{aligned}
  $$
  By the  upper bound $\|S_{\La,\gamma}g\|_{p'}\lesi \|g\|_{p'}$,
  $$
  \left|\int_{\Rd_+} f(x)g(x)dx\right| \lesi \|S_{\La,\gamma}f\|_{p}\|g\|_{p'},
  $$
  and consequently,
  $$
  \|f\|_{p}\lesi \|S_{\La,\gamma}f\|_{p}.
  $$
  This completes the proof of Theorem~\ref{squarefunctions}.	
\end{proof}

\section{Bounds for differences of kernels}
\label{s:newboundsdifferenceskernels}

In this section, we prepare the proof of the reversed Hardy inequality (Theorem~\ref{thm-difference}), which relies on Schur tests and bounds for differences of heat kernels and derivatives of heat kernels.

Let $\alpha\in (0,2]$ and $\sigma$ as in \eqref{eq-sigma}. Set $q:=\min\{\sigma,(\alpha-1)_+\}$. For each $\beta\in (0,\alpha)$, we define
\[
\begin{aligned}
	L^{\alpha,\beta}_t(x,y):=& \big(\one_{x_d\vee y_d\le t^{1/\alpha}} +\one_{x_d\vee y_d\ge t^{1/\alpha}}\one_{|x-y|\ge (x_d\wedge y_d)/2}\big)\Big(1\wedge \f{x_d}{t^{1/\alpha}}\Big)^{q}\Big(1\wedge \f{y_d}{t^{1/\alpha}}\Big)^{q}\\
	&\times t^{-d/\alpha}\Big[\Big( \f{t^{1/\alpha}}{t^{1/\alpha}+|x-y|}\Big)^{d+\beta}\one_{\alpha<2} + \exp\Big(-\f{|x-y|^2}{t}\Big)\one_{\alpha=2}\Big]
\end{aligned}
\]
and
\[
\begin{aligned}
	M^{\alpha,\beta}_t(x,y):=&  \one_{x_d\vee y_d\ge t^{1/\alpha}} \one_{|x-y|\le (x_d\wedge y_d)/2}\\
	&\times \f{t^{1-d/\alpha}}{(x_d\vee y_d)^\alpha}\Big[\Big( \f{t^{1/\alpha}}{t^{1/\alpha}+|x-y|}\Big)^{d+\beta}\one_{\alpha<2} + \exp\Big(-\f{|x-y|^2}{t}\Big)\one_{\alpha=2}\Big].
\end{aligned}
\]

These two functions were used in \cite[Theorem~17]{FrankMerz2023} in the proof of the upper bound for the kernel of the difference $e^{-t\La} -e^{-tL_0}$. In this section, we will prove new bounds for the difference $tL_0 e^{-tL_0}- t\La e^{-t\La}$, see Proposition~\ref{prop-difference} below. We will often use the following auxiliary bound.

\begin{lem}\label{lem- composition of two kernels}
  Let $N\in\N$. Then for all $\beta\in (0,2]$ and all $s,t>0$ we have
  \[
    \int_{\R^N} \f{(st)^\beta}{(s+|x-z|)^{N+\beta}(t+|z-y|)^{N+\beta}}dz \simeq \f{(s+t)^{\beta}}{[(s+t)+|x-y|]^{N+\beta}}
  \]
  for all $x,y\in \R^N$.
\end{lem}

\begin{proof}
  The proof of the inequality ``$\lesi$'' was given in \cite[Lemma~22]{FrankMerz2023}. Here, we prove ``$\simeq$'' directly. It is well-known that 
  \[
    e^{-t(-\Delta_{\R^N})^{\beta/2}}(x,y)\simeq \f{t}{(t^{1/\beta}+|x-y|)^{N+\beta}}.
  \]
  Hence,
  \[
    \begin{aligned}
      \int_{\R^N} \f{(st)^\beta}{(s+|x-z|)^{N+\beta}(t+|z-y|)^{N+\beta}}dz
      &\simeq \int_{\R^N} e^{-s^\beta(-\Delta_{\R^N})^{\beta/2}}(x,z)e^{-t^\beta(-\Delta_{\R^N})^{\beta/2}}(z,y)dz\\
      &= e^{-(s^\beta+ t^\beta)(-\Delta_{\R^N})^{\beta/2}}(x,y)\\
      &\simeq \f{(s+t)^{\beta}}{[(s+t)+|x-y|]^{N+\beta}}
    \end{aligned}
  \]
  for all $s,t>0$ and $x,y\in \R^N$.
\end{proof}

The following technical lemma is inspired from \cite{FrankMerz2023}.

\begin{lem}\label{lem- difference for alpha < 2}
  Let $\alpha\in (0,2)$, $\min\{0,\alpha-1\}<\sigma<\alpha$, and $\beta\in ((\alpha-1)_+,\alpha)$. For $t>0$ and $x,y\in\R_+^d$, let
  \begin{equation*}
    T^{\alpha,\beta}_t(x,y) := \Big(1\wedge\f{x_d}{t^{1/\alpha}}\Big)^{\sigma}\Big(1\wedge\f{y_d}{t^{1/\alpha}}\Big)^{\sigma}t^{-d/\alpha}\Big(\f{t^{1/\alpha}}{t^{1/\alpha}+|x-y|}\Big)^{d+\beta} ,
  \end{equation*}
  and
  \[
    H^{\alpha,\beta}_t(x,y) := \Big(1\wedge\f{x_d}{t^{1/\alpha}}\Big)^{(\alpha-1)_+}\Big(1\wedge\f{y_d}{t^{1/\alpha}}\Big)^{(\alpha-1)_+}t^{-d/\alpha}\Big(\f{t^{1/\alpha}}{t^{1/\alpha}+|x-y|}\Big)^{d+\beta}.
  \]	
  Then there exists $C>0$ such that
  \begin{align}
    \label{eq:lem- difference for alpha < 2}
    \begin{split}
      U^{\alpha,\beta}_t(x,y)
      &:=\one_{x_d\vee y_d\ge t^{1/\alpha}} \one_{|x-y|\le (x_d\wedge y_d)/2}\int_{\Rd_+}\int_0^t H^{\alpha,\beta}_{t-s}(x,z)z_d^{-\alpha}T^{\alpha,\beta}_{s}(z,y)ds dz\\
      &\le C M^{\alpha,\beta}_t(x,y)
    \end{split}
  \end{align}
  for all $x,y\in \Rd_+$ and $t>0$.
\end{lem}

Note that $\sigma\ge(\alpha-1)/2$ implies $\sigma>\min\{0,\alpha-1\}$ when $\alpha\neq1$.

\begin{proof}
  Without loss of generality, we assume $y_d\ge x_d$ and $t=1$. Set 
  \[
    S:=\{(x,y)\in\R_+^d\times\R_+^d: x_d\vee y_d\ge 1, |x-y|\le (x_d\wedge y_d)/2\}.
  \]
  Since $|x_d-y_d|\le |x-y|\le (x_d\wedge y_d)/2$, we have $x_d\simeq y_d$. Hence, $x_d\simeq y_d\gtrsim 1$ for $(x,y)\in S$. We now write
  \[
    \begin{aligned}
      U^{\alpha,\beta}_t(x,y)
      = &\one_{S}(x,y) \int_{\R^{d-1}}\int_0^{x_d/2}\int_0^1 H^{\alpha,\beta}_{1-s}(x,z)z_d^{-\alpha}T^{\alpha,\beta}_{s}(z,y)ds dz\\
            &+\one_{S}(x,y) \int_{\R^{d-1}}\int_{x_d/2}^{\vc}\int_0^1 H^{\alpha,\beta}_{1-s}(x,z)z_d^{-\alpha}T^{\alpha,\beta}_{s}(z,y)ds dz\\
      =:&E_1 +E_2.
    \end{aligned}
  \]
  For the term $E_2$, since $z_d\ge x_d/2$ and $x_d\simeq y_d \gtrsim1$, by Lemma \ref{lem- composition of two kernels}, we have
  \[
    \begin{aligned}
      E_2&\lesi \f{1}{x_d^\alpha} \int_0^1 \int_{\Rd_+}(1-s)^{-d/\alpha}\Big(\f{(1-s)^{1/\alpha}}{(1-s)^{1/\alpha}+|x-z|}\Big)^{d+\beta}s^{-d/\alpha}\Big(\f{s^{1/\alpha}}{s^{1/\alpha}+|z-y|}\Big)^{d+\beta}dz ds\\
         &\lesi \f{1}{x_d^\alpha} \int_0^1  \Big(\f{1}{1+|x-y|}\Big)^{d+\beta} ds
           \lesi \f{1}{x_d^\alpha} \Big(\f{1}{1+|x-y|}\Big)^{d+\beta}
           \lesi M_1^{\alpha,\beta}(x,y).
    \end{aligned}
  \]	
  For the term $E_1$, we note that since $y_d\ge x_d \ge 2z_d$, we have $|y_d-z_d|\simeq y_d\simeq x_d \simeq |x_d-z_d|$. Therefore, in this case for $s\in (0,1)$,
  \[
    (1-s)^{1/\alpha}+|x-z|\simeq x_d + |x'-z'| \ \ \text{and} \ \ s^{1/\alpha}+|y-z|\simeq x_d + |y'-z'|. 
  \]
  This and Lemma \ref{lem- composition of two kernels} imply
  \[
    \begin{aligned}
      & \int_{\R^{d-1}} \f{(1-s)^{\beta/\alpha}}{((1-s)^{1/\alpha}+|x-z|)^{d+\beta}}
      \f{s^{\beta/\alpha}}{(s^{1/\alpha}+|x-z|)^{d+\beta}}dz'\\
      & \quad \simeq \int_{\R^{d-1}} \f{(1-s)^{\beta/\alpha}}{(x_d+|x'-z'|)^{d+\beta}}\f{s^{\beta/\alpha}}{(x_d+|y'-z'|)^{d+\beta}}dz'\\
      & \quad \lesi \f{x_d^{-\beta-1} s^{\beta/\alpha}(1-s)^{\beta/\alpha}}{(x_d+|x'-y'|)^{d+\beta}}
        \simeq \f{x_d^{-\beta-1}(1-s)^{\beta/\alpha}s^{\beta/\alpha}}{(1+|x-y|)^{d+\beta}}.
    \end{aligned}
  \]
  It follows that 
  \[
    \begin{aligned}
      E_1&\lesi \f{x_d^{-\beta-1}}{(1+|x-y|)^{d+\beta}}\int_0^{x_d/2}\int_0^{1}z_d^{-\alpha} (1-s)^{\beta/\alpha}s^{\frac\beta\alpha}\Big(1\wedge\f{z_d}{(1-s)^{\frac1\alpha}}\Big)^{(\alpha-1)_+}\Big(1\wedge\f{z_d}{s^{1/\alpha}}\Big)^{\sigma}dsdz_d.
    \end{aligned}
  \]
  By a simple calculation,
  \[
    \begin{aligned}
      & \int_0^{1}  (1-s)^{\beta/\alpha} s^{\beta/\alpha}\Big(1\wedge\f{z_d}{(1-s)^{1/\alpha}}\Big)^{(\alpha-1)_+}\Big(1\wedge\f{z_d}{s^{1/\alpha}}\Big)^{\sigma}ds
        = \int_0^{1/2}\ldots + \int_{1/2}^1\ldots \\
      & \quad \lesi \Big(1\wedge z_d\Big)^{\sigma+(\alpha-1)_+}.
    \end{aligned}
  \]
	
  Consequently,
  \[
    \begin{aligned}
      E_{1}&\lesi \f{x_d^{-\beta-1}}{(1+|x-y|)^{d+\beta}}\int_0^{x_d/2} z_d^{-\alpha} \Big(1\wedge z_d\Big)^{\sigma+(\alpha-1)_+} dz_d\\
           &\lesi \f{1}{x_d^{\beta +1}}\f{1}{(1+|x-y|)^{d+\beta}}\int_0^{x_d/2} z_d^{-\alpha}\Big(1\wedge z_d\Big)^{\sigma+(\alpha-1)_+} dz_d\\
           &= \f{1}{x_d^{\beta +1}}\f{1}{(1+|x-y|)^{d+\beta}}\Big[x_d^{-\alpha+1}\cdot \one_{\alpha<1}+\ln(x_d)\cdot \one_{\alpha=1}+ \one_{\alpha>1}\Big].
    \end{aligned}
  \]
  Since $x_d\gtrsim 1$, we use the fact $\ln(x_d)\le c_\epsilon x_d^\epsilon$ for every $\epsilon>0$ to further obtain
  \[
    \begin{aligned}
      E_{1} &\lesi   \f{1}{(1+|x-y|)^{d+\beta}}\Big[\Big(\f{1}{x_d}\Big)^{\beta+\alpha}\cdot \one_{\alpha<1}+\Big(\f{1}{x_d}\Big)^{\beta+\alpha-\epsilon}\cdot \one_{\alpha=1}+\Big(\f{1}{x_d}\Big)^{\beta+1}\cdot \one_{\alpha>1}\Big].
    \end{aligned}
  \]
  This completes our proof.
\end{proof}

\begin{lem}\label{lem- difference for alpha = 2}
  Let $\alpha=2$ and $\sigma$ as in \eqref{eq-sigma}. For $x,y\in\R_+^d$ and $t>0$, let
  \begin{equation*}
    T_t(x,y) := \Big(1\wedge\f{x_d}{\sqrt t }\Big)^{\sigma}\Big(1\wedge\f{y_d}{\sqrt t}\Big)^{\sigma}t^{-d/2}\exp\Big(-\f{|x-y|^2}{t}\Big)
  \end{equation*}
  and
  \[
    H_t(x,y) := \Big(1\wedge\f{x_d}{\sqrt t}\Big)\Big(1\wedge\f{y_d}{\sqrt t}\Big)t^{-d/2}\exp\Big(-\f{|x-y|^2}{t}\Big).
  \]  
  Then there exist $C,c>0$ such that
  \[
    \begin{aligned}
      U_t(x,y)&:=\one_{x_d\vee y_d\ge \sqrt t} \one_{|x-y|\le (x_d\wedge y_d)/2}\int_{\Rd_+}\int_0^t H_{t-s}(x,z)z_d^{-\alpha}T_{s}(z,y)ds dz\\
              &\le C M^{\alpha,\beta}_{ct}(x,y)
    \end{aligned}
  \]
  for all $x,y\in \Rd_+$ and $t>0$.
\end{lem}

\begin{proof}
  The proof is similar to that of Lemma \ref{lem- difference for alpha < 2}. The only difference is that instead of Lemma \ref{lem- composition of two kernels}, we use that for $N\ge 1$ and $c_1, c_2>0$, there exist $C, c>0$ such that for all $s,t>0$ and $x,y\in \R^N$, we have
  $$
  \int_{\R^N}\f{1}{(st)^{N/2}}\exp\Big(-\f{|x-z|^2}{c_1s}\Big) \exp\Big(-\f{|z-y|^2}{c_2t}\Big)dz\le C\f{1}{(s+t)^{N/2}}\exp\Big(-\f{|x-y|^2}{c (s+t)}\Big).
  $$
  Hence, we omit the details.
\end{proof}

We now prove new bounds for the integral kernel $Q_t(x,y)$ of the difference $tL_0 e^{-tL_0}- t\La e^{-t\La}$. These bounds will be instrumental to prove the reversed Hardy inequality expressed in terms of the square functions in \eqref{eq:squarefunctions2}.

\begin{prop}
  \label{prop-difference}
  Let $\alpha\in (0,2]$ and let $\lambda \ge 0$ when $\alpha<2$ and $\lambda\ge -1/4$ when $\alpha=2$. Let  $\sigma$ be defined by \eqref{eq-sigma}. Then for any $\beta\in (0,\alpha)$, there exists $C>0$ such that for all $x,y\in \Rd$ and $t>0$,
  \begin{equation}
    \label{eq1-DifferenceKernels}
    |Q_t(x,y)|\le C \big[ L^{\alpha,\beta}_t(x,y)+M^{\alpha,\beta}_t(x,y)\big].
  \end{equation}
\end{prop}

\begin{rem}
  \label{rem-prop-difference}
  Let $\alpha\in(0,2)$, $\lambda\in[\lambda_*,0)$ and assume that $\me{-tL_\lambda}(x,y)$ satisfies the bound in \eqref{eq:thm-heatkernelLa- alpha < 2} with $\sigma$ defined by \eqref{eq-sigma}. Then \eqref{eq1-DifferenceKernels} remains valid. This follows by the same arguments as in the proof below.
\end{rem}

\begin{proof}	
  We only give the proof for $\alpha\in (0,2)$ since the case $\alpha=2$ can be treated similarly. 
  Let $T^{\alpha,\beta}_t(x,y)$ and $ {H}^{\alpha,\beta}_t(x,y)$ be the two functions defined in Lemma \ref{lem- difference for alpha < 2}. Let $\widetilde{p}_{t}(x,y)$ and $\widetilde{p}_{t,1}(x,y)$ denote the kernels of $e^{-tL_0}$ and $-L_0e^{-tL_0}$, respectively. By  Proposition \ref{thm-ptk},
  \[
    |\widetilde{p}_{t}(x,y)|+ t|\widetilde{p}_{t,1}(x,y)|\lesi H^{\alpha,\beta}_t(x,y)
  \]
  and
  \[
    |p_{t}(x,y)|+ t|p_{t,1}(x,y)|\lesi T^{\alpha,\beta}_t(x,y)
  \]
  for all $t>0$ and $x,y\in \Rd_+$. We now consider two cases.

  \bigskip

  \noindent{ {\textbf{Case 1:} $x_d\vee y_d\le t^{1/\alpha}$ OR $x_d\vee y_d\ge t^{1/\alpha}$ and $|x-y|\ge (x_d\wedge y_d)/2$.}}

  Since in this case 
  \[
    H^{\alpha,\beta}_t(x,y)+T^{\alpha,\beta}_t(x,y)\lesi L^{\alpha,\beta}_t(x,y),
  \]
  we get 
  \[
    |Q_t(x,y)|\lesi L^\alpha_t(x,y).
  \]
  
  \bigskip

  \noindent{ {\textbf{Case 2:} $x_d\vee y_d\ge t^{1/\alpha}$ and $|x-y|< (x_d\wedge y_d)/2$.}}

  By Duhamel's formula,
  \[
    \begin{aligned}
      \tilde{p}_{t}(x,y)& - {p}_{t}(x,y)\\
                        &= \lambda \int_0^t\int_{\Rd_+}\tilde{p}_{t-s}(x,z){z_d^{-\alpha}}p_{s}(z,y)dzds\\
                        &=\lambda \int_0^{t/2}\int_{\Rd_+}\tilde{p}_{t-s}(x,z){z_d^{-\alpha}}p_{s}(z,y)dzds+\lambda \int_0^{t/2}\int_{\Rd_+}\tilde{p}_{s}(x,z){z_d^{-\alpha}}p_{t-s}(z,y)dzds.
    \end{aligned}
  \]
  Differentiating both sides with respect to $t$ and multiplying by $t$, we get    
  \begin{equation}
    \label{eq-Kato}
    \begin{aligned}
      {Q}_t(x,y)&=\lambda  t\int_{\mathbb{R}^d_+}\tilde{p}_{t/2}(x,z) {z_d^{-\alpha}}p_{t/2}(z,y)dz+\lambda  t\int_0^{t/2}\int_{\mathbb{R}^d_+}\tilde{p}_{t-s,1}(x,z) {z_d^{-\alpha}}p_{s}(z,y)dzds\\
                &\quad + \lambda t\int_{t/2}^t\int_{\mathbb{R}^d_+}\tilde{p}_{t-s}(x,z) {z_d^{-\alpha}}p_{s,1}(z,y)dzds\\
                &=I_1+I_2+I_3.
    \end{aligned}
  \end{equation}
  The term $I_1$ can be written as
  \[
    \begin{aligned}
      \lambda t\int_{\mathbb{R}^d_+}\tilde{p}_{t/2}(x,z) {z_d^{-\alpha}}p_{t/2}(z,y)dz &=6\lambda \int_{t/3}^{t/2}\int_{\mathbb{R}^d_+}\tilde{p}_{t/2}(x,z) {z_d^{-\alpha}}p_{t/2}(z,y)dzds.
    \end{aligned}
  \]	
  It is easy to see that  $H^{\alpha,\beta}_{t-s}(\cdot,\cdot)\simeq H^{\alpha,\beta}_{t/2}(\cdot,\cdot)\gtrsim \tilde{p}_{t/2}(\cdot,\cdot)$ and $T^{\alpha,\beta}_{s}(\cdot,\cdot)\simeq T^{\alpha,\beta}_{t/2}(\cdot,\cdot)\gtrsim p_{t/2}(\cdot,\cdot)$ for  $s\in [t/3,t/2]$. This, along with Lemma \ref{lem- difference for alpha < 2}, implies
  \[
    \begin{aligned}
      I_1
      &\lesi \int_{t/3}^{t/2}\int_{\mathbb{R}^d}H^{\alpha,\beta}_{t-s}(x,z) {z_d^{-\alpha}}T^{\alpha,\beta}_s(z,y)dzds
      \lesi \int_{0}^{t}\int_{\mathbb{R}^d_+}H^{\alpha,\beta}_{t-s}(x,z) {z_d^{-\alpha}}T^{\alpha,\beta}_s(z,y)dzds \\
      &\lesi L^{\alpha,\beta}_t(x,y)+M^{\alpha,\beta}_t(x,y).
    \end{aligned}
  \]
  
  For the second term, since $t-s\simeq t$ for $ {s}\in (0,t/2)$, we have
  \[
    t|\tilde{p}_{t-s,1}(x,z)|\simeq (t-s)|\tilde{p}_{t-s,1}(x,z)|\lesi H^{\alpha,\beta}_{t-s}(x,z).
  \]
  Therefore, by Lemma \ref{lem- difference for alpha < 2},
  \[
    \begin{aligned}
      I_2&\lesi \int_0^{t/2} \int_{\mathbb{R}^d_+}H^{\alpha,\beta}_{t-s}(x,z) {z_d^{-\alpha}}T^{\alpha,\beta}_s(z,y)dzds \lesi L^{\alpha,\beta}_t(x,y)+M^{\alpha,\beta}_t(x,y).
    \end{aligned}
  \]
  Similarly,
  \[
    \begin{aligned}
      I_3&\lesi \int_{t/2}^t \int_{\mathbb{R}^d_+}H^{\alpha,\beta}_{t-s}(x,z)z_d^{-\alpha}T^{\alpha,\beta}_s(z,y)dzds \lesi L^{\alpha,\beta}_t(x,y)+M^{\alpha,\beta}_t(x,y).
    \end{aligned}
  \]
  This completes our proof.
\end{proof}

\begin{rem}
  In \cite{FrankMerz2023}, similar upper bounds were obtained for the kernel of the difference $ e^{-tL_0} -  e^{-t\La}$. However, these are not suitable for our purpose since the square functions used here are different.
\end{rem}

\section{Proof of the reversed Hardy inequality (Theorem~\ref{thm-difference})}
\label{s:reversedhardy}

We now use the previous bounds for the difference of kernels to prove Theorem~\ref{thm-difference}, i.e., the reversed Hardy inequality, expressed in our new square functions.

\medskip
\emph{Step 1.}
By Proposition \ref{prop-difference},
\begin{align}
  \begin{split}
    \Big(\int_0^\vc &t^{-s}\left|\left(t\La e^{-t\La} -tL_0e^{-tL_0}\right)f(x)\right|^2\f{dt}{t}\Big)^{1/2}\\
                    &= \Big[\sum_{j\in \mathbb{Z}}\int_{2^{\alpha j}}^{2^{\alpha(j+1)}} t^{-s}\left|\left(t\La e^{-t\La} -tL_0e^{-tL_0}\right)f(x)\right|^2\f{dt}{t}\Big]^{1/2}\\
                      &\le \left[\sum_{j\in \mathbb{Z}}\int_{2^{\alpha j}}^{2^{\alpha(j+1)}} t^{-s}\left(\int_{\mathbb{R}^d_+} [L^{\alpha,\beta}_t(x,y)+M^{\alpha,\beta}_t(x,y)]|f(y)|dy\right)^2\f{dt}{t}\right]^{1/2}\\
                      &\le  \sum_{j\in \mathbb{Z}} 2^{-js \alpha/2} \int_{\mathbb{R}^d_+} \left[L^{\alpha,\beta}_{2^{\alpha(j+1)}}(x,y)+M^{\alpha,\beta}_{2^{j\alpha}}(x,y)\right] y_d^{\alpha s/2} \,\frac{|f(y)|}{y_d^{\alpha s/2}} dy
  \end{split}
\end{align}
where in the last inequality we used the embedding $\ell_1\hookrightarrow\ell_2$. Thus, it suffices to show the $L^p(\R_+^d)$-boundedness of the operator with kernel
\begin{align}
  \label{eq:thm-differenceaux1}
  \sum_{j\in\Z}2^{-j\alpha s/2}\left[L_{2^{\alpha(j+1)}}^{\alpha,\beta}(x,y) + M_{2^{j\alpha}}^{\alpha,\beta}(x,y)\right] y_d^{\alpha s/2}.
\end{align}
To that end, we use Schur tests. In the following, we will always replace the Gaussian decay for $\alpha=2$ by a polynomial decay of order $d+2$.

\smallskip
\emph{Step 2.}
We begin with the kernel coming from the $M_{2^{j\alpha}}^{\alpha,\beta}(x,y)$-part in \eqref{eq:thm-differenceaux1}. As discussed in the proof of Lemma~\ref{lem- difference for alpha < 2}, we have $x_d\simeq y_d$ on the support of $M_{2^{j\alpha}}^{\alpha,\beta}(x,y)$. Thus, the kernel
\begin{align}
  \begin{split}
    \sum_{j\in\Z}2^{-j\alpha s/2} M_{2^{j\alpha}}^{\alpha,\beta}(x,y)\, y_d^{\alpha s/2}
    \simeq \sum_{N\in 2^{\alpha \Z}} N^{-s/2} M_{N}^{\alpha,\beta}(x,y)\, (x_d y_d)^{\alpha s/4}
  \end{split}
\end{align}
in question is replaced by a symmetric one and we only have to perform a single Schur test instead of two. The Schur test is analogous to that in the proof of Step 2 in \cite[Theorem~5]{FrankMerz2023}, where the summation over dyadic numbers $2^{\alpha\Z}$ is replaced with an integral over $t$ with respect to the Haar measure $dt/t$. Indeed, the bound follows from
\begin{align}
  \begin{split}
    & \int\limits_{y_d\simeq x_d} \sum_{N\in2^{\alpha\Z}}N^{-\frac s2-\frac d\alpha}\,\frac{N}{x_d^\alpha}\left(\frac{N^{1/\alpha}}{N^{1/\alpha}+|x-y|}\right)^{d+\beta}\one_{N< x_d^\alpha}x_d^{\alpha s/2}\,dy \\
    & \quad \lesssim x_d^{\alpha s/2-\alpha}\sum_{2^{\alpha\Z}\ni N<x_d^\alpha}\int_{\R^d} \left(1\wedge\frac{N^{(d+\beta)/\alpha}}{(N^{1/\alpha}+|x-y|)^{d+\beta}}\right)\,dy
      \lesssim1.
  \end{split}
\end{align}

\smallskip
\emph{Step 3.}
We now consider the kernel coming from the $L_{2^{\alpha(j+1)}}^{\alpha,\beta}(x,y)$-part in \eqref{eq:thm-differenceaux1}. As in the proof of \cite[Theorem~5]{FrankMerz2023}, we will prove the stronger statement that the operator with kernel
\begin{align}
  \begin{split}
    \tilde L_{t}^{\alpha,\beta}(x,y)
    & := \left(\one_{x_d\vee y_d<t^{1/\alpha}} + \one_{x_d\vee y_d>t^{1/\alpha}}\one_{|x-y|>(x_d\wedge y_d)/2}\right) \, \left(1\wedge\frac{x_d}{t^{1/\alpha}}\right)^{-r} \left(1\wedge\frac{y_d}{t^{1/\alpha}}\right)^{-r} \\
    & \qquad \times t^{-d/\alpha} \left(\frac{t^{1/\alpha}}{t^{1/\alpha} + |x-y|}\right)^{d+\beta}
  \end{split}
\end{align}
is $L^p(\R_+^d)$-bounded, where $-r:=(q\wedge 0)\le0$ and $q=\min\{\sigma,(\alpha-1)_+\}$. As in the proof of \cite[Theorem~5]{FrankMerz2023}, we will show that
\begin{align}
  \label{eq:thm-differenceaux2}
  \sum_{N\in2^{\alpha\Z}} N^{-s/2} \tilde L_N^{\alpha,\beta}(x,y) y_d^{\alpha s/2}
  \lesssim \left( \frac{|x-y|\vee x_d\vee y_d}{\sqrt{x_d y_d}} \right)^{2r} \frac{(|x-y|\vee x_d\vee y_d)^{\alpha}}{(|x-y|\vee (x_d\wedge y_d))^{d+\alpha}}.
\end{align}
As with the $M_t^{\alpha,\beta}$-part, the proof is similar to that in \cite{FrankMerz2023}. Let us give the details. We distinguish between $N\lessgtr(x_d\vee y_d)^\alpha$. On the one hand,
\begin{align}
  \label{eq:thm-differenceaux3}
  \begin{split}
    & \sum_{2^{\alpha\Z}\ni N\ge(x_d\vee y_d)^\alpha} N^{-\frac s2} \tilde L_N^{\alpha,\beta}(x,y) y_d^{\alpha s/2} \\
    & \quad \simeq y_d^{\frac{\alpha s}{2}}(x_d y_d)^{-r} \sum_{2^{\alpha\Z}\ni N\ge(x_d\vee y_d)^\alpha} N^{-\frac s2+\frac{2r}{\alpha}-\frac d\alpha}\left(1\wedge\frac{N^{(\beta+d)/\alpha}}{|x-y|^{d+\beta}}\right) \\
    & \quad \lesssim y_d^{\frac{\alpha s}2} (x_dy_d)^{-r} \left[ (|x-y| \! \vee \! x_d \! \vee \! y_d)^{-\frac{\alpha s}2+ 2r-d} + \one_{x_d \vee y_d \le |x-y|} |x-y|^{-d-\alpha} (x_d \! \vee \! y_d)^{2r+\alpha-\frac{\alpha s}2} \right].
  \end{split}
\end{align}
The first term here comes from the $N$-summation from $(|x-y|\vee x_d\vee y_d)^\alpha$ to $\infty$. This sum converges since $-\frac s2+\frac{2r}\alpha-\frac d\alpha<0$. (Note that $s>0$ and $2r\le(1-\alpha)_+< 1$.) The second term comes from an upper bound on the sum between $(x_d\vee y_d)^\alpha$ and $|x-y|^\alpha$, in fact, from an upper bound on the integral between $0$ and $|x-y|^\alpha$. This sum converges since $-\frac s2 + \frac{2r}{\alpha}+1>0$.
	
We now turn to the contribution to $\tilde L_N^{\alpha,\beta}$ from $\{x_d\vee y_d\ge t^{1/\alpha}\}$. We have
\begin{align}
  \label{eq:thm-differenceaux4}
  \begin{split}
    & \sum_{2^{\alpha\Z}\ni N\le(x_d\vee y_d)^\alpha} N^{-\frac s2} \tilde L_N^{\alpha,\beta}(x,y) \, y_d^{\frac{\alpha s}2} \\
    & \simeq y_d^{\frac{\alpha s}2} \sum_{2^{\alpha\Z}\ni N\le(x_d\vee y_d)^\alpha} N^{-\frac s2 -\frac d\alpha} \left( 1 \wedge \frac{x_d\wedge y_d}{N^{1/\alpha}} \right)^{-r} 
      \left( 1\wedge \frac{N^{(\beta+d)/\alpha}}{|x-y|^{d+\beta}} \right) \one_{|x-y|\ge (x_d\wedge y_d)/2} \\
    & \le y_d^{\frac{\alpha s}2} |x-y|^{-d-\alpha} \sum_{2^{\alpha\Z}\ni N\le(x_d\vee y_d)^\alpha} N^{1-\frac s2} \left( 1 \wedge \frac{x_d\wedge y_d}{N^{1/\alpha}} \right)^{-r} \one_{|x-y|\ge (x_d\wedge y_d)/2} \\
    & \lesssim y_d^{\frac{\alpha s}2} |x-y|^{-d-\alpha} \left[ (x_d\wedge y_d)^{\alpha-\frac{\alpha s}2} + (x_d\wedge y_d)^{-r} (x_d\vee y_d)^{r+\alpha-\frac{\alpha s}2} \right] \one_{|x-y|\ge (x_d\wedge y_d)/2} \,.
  \end{split}
\end{align}
The first term here comes from the sum from $0$ to $(x_d\wedge y_d)^\alpha$. This converges since $s<2$. The second term comes from an upper bound on the sum from $(x_d\wedge y_d)^\alpha$ to $(x_d\vee y_d)^\alpha$, in fact, from an upper bound on the sum between $0$ and $(x_d\vee y_d)^\alpha$. This sum converges since $-\frac s2+\frac r\alpha>-1$. Combining \eqref{eq:thm-differenceaux3}--\eqref{eq:thm-differenceaux4} as in the proof of \cite[Theorem~5]{FrankMerz2023} eventually yields \eqref{eq:thm-differenceaux2}. Thus, the $L^p(\R_+^d)$-boundedness of $L_{2^{\alpha(j+1)}}^{\alpha,\beta}(x,y)$-part in \eqref{eq:thm-differenceaux1} follows from Proposition~\ref{schurkillip2} below. This concludes the proof of Theorem~\ref{thm-difference}. \qed

\begin{prop}
  \label{schurkillip2}
  Let $\alpha>0$ and $0\le r<\frac12$. Then the integral operator with integral kernel
  \begin{align*}
    \left( \frac{|x-y|\vee x_d\vee y_d}{\sqrt{x_d y_d}} \right)^{2r} \frac{(|x-y|\vee x_d\vee y_d)^{\alpha}}{(|x-y|\vee (x_d\wedge y_d))^{d+\alpha}}
  \end{align*}
  is bounded on $L^p(\R_+^d)$ for all $p\in(\frac{1}{1-r},\frac1r)$.
\end{prop}

\begin{proof}
  The proof is similar to that of \cite[Proposition~19]{FrankMerz2023}.
  We denote the kernel in the proposition by $k(x,y)$. By symmetry of $k(x,y)$ and a duality argument, it suffices to consider $p\le2$.
  
  \emph{Step 1.}
  As a preliminary step to the main argument, we integrate over the $\R^{d-1}$-variables. By \cite[(20)]{FrankMerz2023},
  \begin{equation}
    \label{eq:schurmarginal}
    \int_{\R^{d-1}} k(x,y)\,dy'
    \lesssim \left( \frac{x_d\vee y_d}{\sqrt{x_d y_d}} \right)^{2r} \frac{(x_d\vee y_d)^{\alpha}}{(|x_d-y_d|\vee (x_d\wedge y_d))^{1+\alpha}}.
  \end{equation}
  
  \smallskip
  \emph{Step 2.}
  We perform weighted Schur tests for the operator with kernel given by the right side of \eqref{eq:schurmarginal}. We use the weight
  $$
  w(x,y)=\left( \frac{x_d}{y_d} \right)^\beta
  \qquad\text{with}\ \beta\in(pr,p(1-r))\cap(p'r,p'(1-r)) \,.
  $$
  Since $r<1/2$ and $p\in(1/(1-r),1/r)$, it is possible to find such a $\beta$.
  
  For the first part of the Schur test, we use \eqref{eq:schurmarginal} to bound
  \begin{align*}
    \int_{\R^d_+} w(x,y)^{1/p} k(x,y)\,dy
    & \simeq \int_0^\infty \left( \frac{x_d}{y_d} \right)^{\beta/p} \left( \frac{x_d\vee y_d}{\sqrt{x_d y_d}} \right)^{2r} \frac{(x_d\vee y_d)^{\alpha}}{(|x_d-y_d|\vee (x_d\wedge y_d))^{1+\alpha}}\,dy_d \\
    & = \int_0^\infty t^{-\beta/p -r} \frac{(1\vee t)^{\alpha+2r}}{(|1-t|\vee (1\wedge t))^{1+\alpha}}\,dt \\
    & \simeq \int_0^\infty t^{-\beta/p-r} (1\wedge t^{-1+2r})\,dt
      <\infty.
  \end{align*}
  The finiteness of the last integral uses the assumptions $pr<\beta<p(1-r)$.
  
  For the second part of the Schur test, we note that, by symmetry, \eqref{eq:schurmarginal} remains valid with $dy'$ replaced by $dx'$. Thus, and by $p'r<\beta<p'(1-r)$,
  \begin{align*}
    \int_{\R^d_+} w(x,y)^{-1/p'} k(x,y)\,dx
    & \simeq \int_0^\infty \left( \frac{y_d}{x_d} \right)^{\beta/p'} \left( \frac{x_d\vee y_d}{\sqrt{x_d y_d}} \right)^{2r} \frac{(x_d\vee y_d)^{\alpha}}{(|x_d-y_d|\vee (x_d\wedge y_d))^{1+\alpha}}\,dx_d \\
    & = \int_0^\infty t^{-\beta/p' -r} \frac{(1\vee t)^{\alpha+2r}}{(|1-t|\vee (1\wedge t))^{1+\alpha}}\,dt
      <\infty,
  \end{align*}
  as before. The $L^p(\R^d_+)$-boundedness therefore follows from the Schur test.
\end{proof}

\section{Proof of the generalized Hardy inequality (Theorem~\ref{thm-HardyIneq})}
\label{s:generalizedhardy}

To prove Theorem~\ref{thm-HardyIneq}, we use the following Riesz kernel estimates.

\begin{lem}[{\cite[Theorem~11]{FrankMerz2023}}]
  \label{lem-kernel of fractional power of L lambda}
  Let $\alpha\in (0,2]$ and let $\lambda\ge 0$ when $\alpha\in (0,2)$ and $\lambda\ge -1/2$ when $\alpha=2$. Let $\sigma$ be defined by \eqref{eq-sigma} and let $s\in (0,\f{2d}{\alpha}\wedge \f{2(d+2\sigma)}{\alpha})$. Then,
  \begin{enumerate}[\rm (a)]
  \item For all $x,y \in \Rd_+$ with $|x-y|\le x_d\vee y_d$,
    \begin{align}
      \label{eq:riesz}
      L_\lambda^{-s/2}(x,y)
      \simeq |x-y|^{\f{\alpha s}{2}-d}\Big(1\wedge\f{x_d}{|x-y|}\wedge\f{y_d}{|x-y|}\Big)^\sigma.
    \end{align}
  \item For all $x,y \in \Rd_+$ with $x_d\vee y_d\le |x-y|$,
    \begin{align}
      \label{eq:rieszimprovedsim}
      \begin{split}
        L_\lambda^{-s/2}(x,y)
        & \simeq |x-y|^{\f{\alpha s}{2}-d}\Big(\f{x_d y_d}{|x-y|^2}\Big)^\sigma\\
        & \hspace{-5em} \times \Big[\one_{\alpha=2}+\Big(\one_{\sigma\le \f{\alpha}{2}(1+\f{s}{2})}+\Big(\ln\f{|x-y|}{x_d\vee y_d}\Big){\one_{\sigma=\f{\alpha}{2}(1+\f{s}{2})}}+\Big(\f{|x-y|}{x_d\vee y_d}\Big)^{2\sigma-\alpha(1+\f{s}{2})}\one_{\sigma>\f{\alpha}{2}(1+\f{s}{2})} \Big)\one_{\alpha<2}\Big].
      \end{split}
    \end{align}
  \end{enumerate}
\end{lem}

\begin{rem}
  \label{rieszrem}
  Let $\alpha\in(0,2)$, $\lambda\in[\lambda_*,0)$ and assume that $\me{-tL_\lambda}(x,y)$ satisfies the bound in \eqref{eq:thm-heatkernelLa- alpha < 2} with $\sigma$ defined by \eqref{eq-sigma}. Then \eqref{eq:riesz} and \eqref{eq:rieszimprovedsim} remain valid. Similarly, the upper (resp.~lower) bound in \eqref{eq:thm-heatkernelLa- alpha < 2} implies the upper (resp.~lower) bound in \eqref{eq:riesz} and \eqref{eq:rieszimprovedsim}. This was observed in \cite[Theorem~11]{FrankMerz2023}.
\end{rem}

Besides these Riesz kernel bounds, we use the following extension of \cite[Lemma~15]{FrankMerz2023}, where $(L_\lambda)^{s/2}C_c^\infty(\R_+^d)\subseteq L^2(\R_+^d)$, i.e., $C_c^\infty(\R_+^d)\subset {\dom} L_\lambda^{s/2}$ was shown for all $s,\alpha\in (0,2]$.

\begin{lem}
  \label{domain} 
  Let $1<p<\vc$, $\alpha\in (0,2]$, and let $\lambda \ge 0$ when $\alpha<2$ and $\lambda\ge -1/4$ when $\alpha=2$. Let $\sigma$ be defined by \eqref{eq-sigma} and suppose $s, \alpha\in (0,2]$. Then, $(L_\lambda)^{s/2}C_c^\infty(\R_+^d)\subseteq L^p(\R_+^d)$.
\end{lem}

\begin{proof}
  We first consider $s=2$.
  Let $\hat f(\xi)=(2\pi)^{-d/2}\int_{\R^d}\me{2\pi ix\cdot\xi}f(x)\,dx$ denote the Fourier transform of $f$. Then, for $\lambda_0=C(\alpha/2)$, we have (see, e.g., \cite{Bogdanetal2003})
  \[
    L_\lambda f = ((-\Delta)^{\alpha/2}\widetilde{f})|_{\Rd_+} + (\lambda-\lambda_0)x_d^{-\alpha}f,
  \]
  where $\widetilde{f}$ is the zero extension of $f$ to $\Rd$. Since $((-\Delta)^{\alpha/2}\widetilde{f})^\wedge (\xi) = |\xi|^\alpha (\widetilde{f})^\wedge \ (\xi)$, $((-\Delta)^{\alpha/2}\widetilde{f})$ is a Schwartz function and hence $((-\Delta)^{\alpha/2}\widetilde{f}) \in L^p(\Rd)$. On the other hand, since $x_d^{-\alpha}$ is bounded on the support of $f$, $x_d^{-\alpha}f\in L^p(\Rd_+)$. Consequently, $L_\lambda f\in L^p(\Rd_+)$ for all $1<p<\vc$.

  Now let $s\in (0,2)$. We use 
  \[
    \begin{aligned}
      \La^{s/2}f&=\f{1}{\Gamma(1-s/2)}\int_{0}^\vc u^{1-s/2} \La e^{-u\La}f\f{du}{u}\\
                &=\f{1}{\Gamma(1-s/2)}\int_{0}^1 u^{1-s/2}   e^{-u\La}\La f\f{du}{u} + \f{1}{\Gamma(1-s/2)}\int_{1}^\vc u \La e^{-u\La} f\f{du}{u^{1+s/2}},
    \end{aligned}
  \]
  which implies 
  \[
    \begin{aligned}
      \|\La^{s/2}f\|_p&=\f{1}{\Gamma(1-s/2)}\int_{0}^1 u^{1-s/2}   \|e^{-u\La}\La f\|_p \f{du}{u} + \f{1}{\Gamma(1-s/2)}\int_{1}^\vc \|u \La e^{-u\La} f\|_p\f{du}{u^{1+s/2}}.
    \end{aligned}
  \]
  From Lemma \ref{lem-Lp boundedness of Tt} and Proposition \ref{thm-ptk}, we have
  \[
    \|e^{-u\La}\|_{p\to p}+ \|u\La e^{-u\La}\|_{p\to p} \lesi 1
  \]
  uniformly in $u>0$.
  Therefore,
  \[
    \begin{aligned}
      \|\La^{s/2}f\|_p&=\f{1}{\Gamma(1-s/2)}\int_{0}^1 u^{1-s/2}   \|\La f\|_p \f{du}{u} + \f{1}{\Gamma(1-s/2)}\int_{1}^\vc \|f\|_p\f{du}{u^{1+s/2}}\\
                      &\lesi \|\La f\|_p + \|f\|_p.
    \end{aligned}
  \]
  Since $f, \La f \in L^p(\Rd_+)$, $\La^{s/2}f\in L^p(\Rd_+)$.
\end{proof}

\begin{rem}
  \label{domain Llambda}
  Let $\alpha\in(0,2)$, $\lambda\in[\lambda_*,0)$ and assume that $\me{-tL_\lambda}(x,y)$ satisfies the bound in \eqref{eq:thm-heatkernelLa- alpha < 2} with $\sigma$ defined by \eqref{eq-sigma}. Then Lemma \ref{domain} remains valid for for $\f{1}{1+\sigma}=:(p_{-\sigma})'<p< p_{-\sigma}:=-\f{1}{\sigma}$. This follows by the same arguments as in the proof above, taking into account Remark \ref{rem-thm-ptk}.
\end{rem}

We can now give the

\begin{proof}[Proof of Theorem~\ref{thm-HardyIneq}]
  The necessity of $(\f{\alpha s}{2}-\sigma)_+<\f{1}{p}<1+\sigma\wedge0$ for \eqref{eq:thm-HardyIneqPre} to hold is proved exactly as in \cite[Theorem~13]{FrankMerz2023}, so we omit the details.

  Thus, let $(\f{\alpha s}{2}-\sigma)_+<\f{1}{p}<1+\sigma\wedge0$ in the following.
  Formula \eqref{eq:thm-HardyIneq} follows from \eqref{eq:thm-HardyIneqPre} and Lemma~\ref{domain} by setting $f=(L_\lambda)^{s/2}g$ with $g\in C_c^\infty(\R_+^d)$. Thus, it suffices to show \eqref{eq:thm-HardyIneqPre} for all $f\in L^p(\R_+^d)$. We set
  \[
    K(x,y):=x_d^{-\alpha s/2} L_\lambda^{-s/2}(x,y)
  \]
  and show that the integral operator with kernel $K(x,y)$ is bounded on $L^p(\R_+^d)$. We will use weighted Schur tests for the kernel $K(x,y)$ in four regions corresponding to 
  \[
    |x-y|\le 4(x_d\wedge y_d), 4x_d \le |x-y|\le 4y_d, 	4y_d \le |x-y|\le 4x_d,
  \]
  and
  \[
    4(x_d\vee y_d) \le |x-y|.
  \]
	
  \bigskip
  
  \noindent
  \textbf{Case 1: $|x-y|\le 4(x_d\wedge y_d)$.}
  The Schur test in this region was already carried out in \cite{FrankMerz2023}. We give the proof for completeness.
  From Lemma \ref{lem-kernel of fractional power of L lambda} in this region, 
  \[
    K(x,y)\simeq  \f{x_d^{-\alpha s/2}}{|x-y|^{d-\f{\alpha s}{2}}}.
  \]
  Since $d-\alpha s/2 < d$, we have
  \[
    \begin{aligned}
      \int_{|x-y|\le 4(x_d\wedge y_d)}K(x,y) dy
      &\lesi \int_{|x-y|\le 4y_d}\f{x_d^{-\alpha s/2}}{|x-y|^{d-\f{\alpha s}{2}}} dy\\
      & \lesi 1.
    \end{aligned} 
  \]
  For the integral with respect to $dx$, we observe $y_d \ge x_d + |x_d - y_d|\le x_d + |x - y|\le 5x_d$. Therefore,
  \[
    \begin{aligned}
      \int_{|x-y|\le 4(x_d\wedge y_d)}K(x,y) dx
      &\lesi \int_{|x-y|\le 4 y_d}\f{y_d^{-\alpha s/2}}{|x-y|^{d-\f{\alpha s}{2}}} dx
        \lesi 1.
    \end{aligned} 
  \]
  
  \bigskip
  
  \noindent
  \textbf{Case 2: $4x_d \le |x-y|\le 4y_d$.} From Lemma \ref{lem-kernel of fractional power of L lambda},
  \[
    K(x,y) \simeq \f{x_d^{\sigma -\f{\alpha s}{2}}}{|x-y|^{d+\sigma -\f{\alpha s}{2}}}.
  \]
  Since $\f{s\alpha}{2}-\sigma<\f{1}{p}$ and $1+2\sigma -\alpha s/2>0$, we can choose $\beta$ such that 
  \begin{equation}\label{eq-beta}
    p(\alpha s/2 -\sigma)<\beta < \min\{p'(1+\sigma-\alpha s/2), p(1+\sigma)\}.
  \end{equation}
  It follows that 
  \begin{equation}\label{eq1-from beta}
    \sigma -\f{\alpha s}{2}+\beta/p>0, 
  \end{equation}
  \begin{equation}\label{eq2-from beta}
    \sigma -\f{\alpha s}{2}-\beta/p'>-1,
  \end{equation}
  and
  \begin{equation}\label{eq3-from beta}
    \sigma -\beta/p>-1.
  \end{equation}	
  Moreover, since $1+\sigma -\alpha s/2>0$ for $\sigma> \f{\alpha}{2}\big(1+\f{s}{2}\big)$, we can choose $\beta$ such that 
  \begin{equation}\label{eqs-beta}
    \max\{0,p(\alpha s/2 -\sigma)\}<\beta < \min\{p'(1+\sigma-\alpha s/2), p(1+\sigma)\}.
  \end{equation}
  as long as $\alpha<2$ and $\sigma> \f{\alpha}{2}\big(1+\f{s}{2}\big)$.
  (Note that this threshold for $\sigma$ also occurs in \cite[Theorem~11]{FrankMerz2023}.)
  We verify the $L^p$-boundedness using a Schur test with weight
  \[
    w(x,y)=\Big(\f{x_d}{|x-y|}\Big)^\beta,
  \]
  i.e., we show that
  \[
    \sup_{x\in\R_+^d}\int_{4x_d \le |x-y|\le 4y_d}w(x,y)^{1/p}K(x,y)dy\lesi 1
  \]	
  and
  \[
    \sup_{y\in\R_+^d}\int_{4x_d \le |x-y|\le 4y_d}w(x,y)^{-1/p'}K(x,y)dx\lesi 1. 
  \]	
  From \eqref{eq1-from beta}, we have
  \[
    \begin{aligned}
      \int_{4x_d \le |x-y|\le 4y_d}w(x,y)^{1/p}K(x,y)dy
      &\lesi \int_{4x_d \le |x-y| }\f{x_d^{\sigma -\f{\alpha s}{2}+\beta/p}}{|x-y|^{d+\sigma -\f{\alpha s}{2}+\beta/p}} dy \lesi 1.
    \end{aligned}
  \]
  
  For the second inequality, we have
  \[
    \begin{aligned}
      \int_{4x_d \le |x-y|\le 4y_d}w(x,y)^{-\beta/p'}K(x,y)dx&\lesi \int_{4x_d \le |x-y|\le 4y_d}\f{x_d^{\sigma -\f{\alpha s}{2}-\beta/p'}}{|x-y|^{d+\sigma -\f{\alpha s}{2}-\beta/p'}} dx.
    \end{aligned}
  \]
  In this region, we have $|x-y|\ge y_d -x_d\ge y_d -|x-y|/4$. Thus, $|x-y|\ge 4y_d/5$ and, by \eqref{eq2-from beta},
  \[
    \begin{aligned}
      \int_{4x_d \le |x-y|\le 4y_d}w(x,y)^{-\beta/p'}K(x,y)dx
      &\lesi \int_{4x_d \le |x-y|\le 4y_d}\f{x_d^{\sigma -\f{\alpha s}{2}-\beta/p'}}{y_d^{d+\sigma -\f{\alpha s}{2}-\beta/p'}} dx\\
      &\lesi \int_0^{ y_d}\int_{|x'-y'|\le 4y_d} \f{x_d^{\sigma -\f{\alpha s}{2}-\beta/p'}}{y_d^{d+\sigma -\f{\alpha s}{2}-\beta/p'}} dx'dx_d\\
      &\lesi \int_0^{y_d} \f{x_d^{\sigma -\f{\alpha s}{2}-\beta/p'}}{y_d^{1+\sigma -\f{\alpha s}{2}-\beta/p'}} dx_d \lesi 1.
    \end{aligned}
  \]

  \bigskip
  \noindent \textbf{Case 3: $4y_d\le |x-y|\le 4x_d$.} In this case,
  \[
    K(x,y) \simeq \f{x_d^{-\f{\alpha s}{2}}y_d^\sigma}{|x-y|^{d+\sigma -\f{\alpha s}{2}}}.
  \]
  Since $\f{1}{p}<1+\sigma$, we can choose $\gamma$ such that 
  \[
    \sigma p' <\gamma<p(1+\sigma),
  \]
  which implies 
  \begin{equation}
    \label{eq-gamma}
    \sigma -\gamma/p > -1 \ \  \text{and} \ \ \sigma +\gamma/p'>0.
  \end{equation}
  We will verify a Schur test with weight
  \[
    w(x,y)=\Big(\f{|x-y|}{y_d}\Big)^\gamma.
  \]
  We first show that 
  \[
    \int_{4y_d \le |x-y|\le 4x_d}w(x,y)^{1/p}K(x,y)dy\lesi 1. 
  \]	
  To do this, we write
  \[
    \begin{aligned}
      \int_{4y_d \le |x-y|\le 4x_d}w(x,y)^{1/p}K(x,y)dy
      &\simeq \int_{4y_d \le |x-y|\le 4x_d}\Big(\f{|x-y|}{y_d}\Big)^{\gamma/ p}\f{x_d^{-\f{\alpha s}{2}}y_d^\sigma}{|x-y|^{d+\sigma -\f{\alpha s}{2}}}dy\\
      &\simeq \int_{4y_d \le |x-y|\le 4x_d} \f{x_d^{-\f{\alpha s}{2}}y_d^{\sigma-\gamma/p}}{|x-y|^{d+\sigma -\gamma/p-\f{\alpha s}{2}}}dy\\
    \end{aligned}
  \]
  Similarly to \textbf{Case 2}, in this region, $|x-y|\ge x_d -y_d\ge x_d -|x-y|/4$, which implies that $|x-y|\ge 4x_d/5$. Hence, $|x-y|\simeq x_d$ in this region. Consequently,
  \[
    \begin{aligned}
      \int_{4y_d \le |x-y|\le 4x_d}w(x,y)^{1/p}K(x,y)dy
      &\lesi \int_0^{ x_d}\int_{|x'-y'|\le 4x_d} \f{x_d^{-\f{\alpha s}{2}}y_d^{\sigma-\gamma/p}}{x_d^{d+\sigma -\gamma/p-\f{\alpha s}{2}}} dy'dy_d\\
      &\lesi \int_0^{ x_d}  \f{x_d^{d-\f{\alpha s}{2}-1}y_d^{\sigma-\gamma/p}}{x_d^{d+\sigma -\gamma/p-\f{\alpha s}{2}}} dy_d\\
      &\lesi \int_0^{ x_d}  x_d^{\gamma/p-\sigma -1}y_d^{\sigma-\gamma/p}  dy_d
        \lesi 1,
    \end{aligned}
  \]
  where in the last inequality we used \eqref{eq-gamma}.

  It remains to show that 
  \[
    \int_{4y_d \le |x-y|\le 4x_d}w(x,y)^{-1/p'}K(x,y)dx\lesi 1. 
  \]	
  Since in this case $|x-y|\simeq x_d$, by \eqref{eq-gamma} we have
  \[
    \begin{aligned}
      \int_{4y_d \le |x-y|\le 4x_d}w(x,y)^{-1/p'}K(x,y)dx
      &\simeq \int_{  |x-y|\ge 4y_d}\Big(\f{|x-y|}{y_d}\Big)^{-\gamma/ p'}\f{y_d^\sigma}{|x-y|^{d+\sigma}}dx\\
      &\lesi \int_{  |x-y|\ge 4y_d} \f{ y_d^{\sigma+\gamma/p'}}{|x-y|^{d+\sigma+\gamma/p'}}dx
        \lesi 1.
    \end{aligned}
  \]

  \bigskip
  \noindent
  \textbf{Case 4: $4(x_d\vee y_d) \le |x-y|$.} In this case,
  \[
    \begin{aligned}
      K&(x,y)\simeq x_d^{-\alpha s/2}|x-y|^{\f{\alpha s}{2}-d}\Big(\f{x_d y_d}{|x-y|^2}\Big)^\sigma\\
       &\times \Big[\one_{\alpha=2}+\Big(\one_{\sigma\le \f{\alpha}{2}(1+\f{s}{2})}+\Big(\ln\f{|x-y|}{x_d\vee y_d}\Big){\one_{\sigma=\f{\alpha}{2}(1+\f{s}{2})}}+\Big(\f{|x-y|}{x_d\vee y_d}\Big)^{2\sigma-\alpha(1+\f{s}{2})}\one_{\sigma>\f{\alpha}{2}(1+\f{s}{2})} \Big)\one_{\alpha<2}\Big].
    \end{aligned}
  \]
  We will perform a Schur test with weight
  \[
    w(x,y)= \Big(\f{x_d}{y_d}\Big)^{\beta},
  \] 
  where $\beta $ is defined by \eqref{eq-beta}.

  Since $|x-y|\ge 4(x_d\vee y_d)$, $|x-y|\simeq |x'-y'|$. Then we have
  \[
    \begin{aligned}
      \int_{\Rd_+} &w(x,y)^{1/p}K(x,y)dy \\
                   &\lesi \sum_{j\ge 1} \int_{0}^{2^jx_d}\int_{2^j x_d<|x'-y'|\le 2^{j+1}x_d}w(x,y)^{1/p}K(x,y)dy'dy_d\\
                   &\lesi \sum_{j\ge 1}\int_{0}^{2^jx_d} x_d^{-\alpha s/2} (2^jx_d)^{\f{\alpha s}{2}-1}\Big(\f{x_d y_d}{(2^jx_d)^2}\Big)^\sigma\Big(\f{x_d}{y_d}\Big)^{\beta/p}\\
                   & \ \ \times \Big[\one_{\alpha=2}+\Big(\one_{\sigma\le \f{\alpha}{2}(1+\f{s}{2})}+\Big(\ln\f{2^jx_d}{x_d\vee y_d}\Big){\one_{\sigma=\f{\alpha}{2}(1+\f{s}{2})}}+\Big(\f{2^jx_d}{x_d\vee y_d}\Big)^{2\sigma-\alpha(1+\f{s}{2})}\one_{\sigma>\f{\alpha}{2}(1+\f{s}{2})} \Big)\one_{\alpha<2}\Big]dy_d. 
    \end{aligned}
  \]
  By a straightforward calculation, we get
  \[
    \begin{aligned}
      \int_{\Rd_+} &w(x,y)^{1/p}K(x,y)dy \\
                   &\lesi \sum_{j\ge 1}2^{j(\alpha s/2 -1 -2\sigma) }\int_{0}^{2^jx_d} x_d^{-1-\sigma +\beta/p} y_d^{\sigma -\beta/p} dy_d\\
                   & \ \ \times \Big[\one_{\alpha=2}+\Big(\one_{\sigma\le \f{\alpha}{2}(1+\f{s}{2})}+\ln(2^j){\one_{\sigma=\f{\alpha}{2}(1+\f{s}{2})}}+2^{j(2\sigma-\alpha(1+\f{s}{2}))}\one_{\sigma>\f{\alpha}{2}(1+\f{s}{2})} \Big)\one_{\alpha<2}\Big]. 
    \end{aligned}
  \]
  From \eqref{eq3-from beta}, 
  \[
    \begin{aligned}
      \int_{\Rd_+}& w(x,y)^{1/p}K(x,y)dy \\
                  &\lesi \sum_{j\ge 1}2^{j(\alpha s/2 -1 -2\sigma)} 2^{j(1+\sigma -\beta/p)}\\
                  & \ \ \times \Big[\one_{\alpha=2}+\Big(\one_{\sigma\le \f{\alpha}{2}(1+\f{s}{2})}+\ln(2^j){\one_{\sigma=\f{\alpha}{2}(1+\f{s}{2})}}+2^{j(2\sigma-\alpha(1+\f{s}{2}))}\one_{\sigma>\f{\alpha}{2}(1+\f{s}{2})} \Big)\one_{\alpha<2}\Big]\\
                  &\lesi \sum_{j\ge 1}2^{j(\alpha s/2  -\sigma -\beta/p)}\\
                  & \ \ \times \Big[\one_{\alpha=2}+\Big(\one_{\sigma\le \f{\alpha}{2}(1+\f{s}{2})}+\ln(2^j){\one_{\sigma=\f{\alpha}{2}(1+\f{s}{2})}}+2^{j(2\sigma-\alpha(1+\f{s}{2}))}\one_{\sigma>\f{\alpha}{2}(1+\f{s}{2})} \Big)\one_{\alpha<2}\Big].
    \end{aligned}
  \]
  This, along with \eqref{eq1-from beta}, implies 
  \[
    \begin{aligned}
      \int_{\Rd_+}& w(x,y)^{1/p}K(x,y)dy \\
                  & \lesi 1 + \sum_{j\ge 1}2^{j(\alpha s/2  -\sigma -\beta/p)}2^{j(2\sigma-\alpha(1+\f{s}{2}))}\one_{\sigma>\f{\alpha}{2}(1+\f{s}{2})}  \one_{\alpha<2}\\
                  & \lesi 1 + \sum_{j\ge 1}2^{j( \sigma-\alpha -\beta/p)}\one_{\sigma>\f{\alpha}{2}(1+\f{s}{2})}  \one_{\alpha<2} \lesi 1,
    \end{aligned}
  \]
  where in the last inequality we used \eqref{eqs-beta} and the fact $\sigma< \alpha$.
  
  For the second inequality, similarly to the first inequality
  \[
    \begin{aligned}
      \int_{\Rd_+} &w(x,y)^{-1/p'}K(x,y)dx \\
                   &\lesi \sum_{j\ge 1} \int_{0}^{2^jy_d}\int_{2^j y_d<|x'-y'|\le 2^{j+1}y_d}w(x,y)^{-1/p'}K(x,y)dx'dx_d\\
                   &\lesi \sum_{j\ge 1}\int_{0}^{2^jy_d} x_d^{-\alpha s/2} (2^jy_d)^{\f{\alpha s}{2}-1}\Big(\f{x_d y_d}{(2^jy_d)^2}\Big)^\sigma\Big(\f{x_d}{y_d}\Big)^{-\beta/p'} \\
                   & \ \ \times \Big[\one_{\alpha=2}+\Big(\one_{\sigma\le \f{\alpha}{2}(1+\f{s}{2})}+\Big(\ln\f{2^jy_d}{x_d\vee y_d}\Big){\one_{\sigma=\f{\alpha}{2}(1+\f{s}{2})}}+\Big(\f{2^jy_d}{x_d\vee y_d}\Big)^{2\sigma-\alpha(1+\f{s}{2})}\one_{\sigma>\f{\alpha}{2}(1+\f{s}{2})} \Big)\one_{\alpha<2}\Big]dx_d. 
    \end{aligned}
  \]
  By a simple calculation, we get
  \[
    \begin{aligned}
      \int_{\Rd_+} &w(x,y)^{1/p}K(x,y)dx \\
                   &\lesi \sum_{j\ge 1}2^{j(\alpha s/2 -1 -2\sigma )}\int_{0}^{2^jy_d} x_d^{-\alpha s/2+\sigma -\beta/p'} y_d^{\alpha s/2 -1 -\sigma +\beta/p'} dx_d\\
                   & \ \ \times \Big[\one_{\alpha=2}+\Big(\one_{\sigma\le \f{\alpha}{2}(1+\f{s}{2})}+\ln(2^j){\one_{\sigma=\f{\alpha}{2}(1+\f{s}{2})}}+2^{j(2\sigma-\alpha(1+\f{s}{2}))}\one_{\sigma>\f{\alpha}{2}(1+\f{s}{2})} \Big)\one_{\alpha<2}\Big]. 
    \end{aligned}
  \]
  
  This, along with \eqref{eq2-from beta}, yields
  \[
    \begin{aligned}
      \int_{\Rd_+}& w(x,y)^{1/p}K(x,y)dx \\
                  &\lesi \sum_{j\ge 1}2^{j(\alpha s/2 -1 -2\sigma +(\gamma-\beta)/p)} 2^{j(1+\sigma -\gamma/p)}\\
                  & \ \ \times \Big[\one_{\alpha=2}+\Big(\one_{\sigma\le \f{\alpha}{2}(1+\f{s}{2})}+\ln(2^j){\one_{\sigma=\f{\alpha}{2}(1+\f{s}{2})}}+2^{j(2\sigma-\alpha(1+\f{s}{2}))}\one_{\sigma>\f{\alpha}{2}(1+\f{s}{2})} \Big)\one_{\alpha<2}\Big]\\
                  &\lesi \sum_{j\ge 1}2^{j(\alpha s/2  -\sigma -\beta/p)}\\
                  & \ \ \times \Big[\one_{\alpha=2}+\Big(\one_{\sigma\le \f{\alpha}{2}(1+\f{s}{2})}+\ln(2^j){\one_{\sigma=\f{\alpha}{2}(1+\f{s}{2})}}+2^{j(2\sigma-\alpha(1+\f{s}{2}))}\one_{\sigma>\f{\alpha}{2}(1+\f{s}{2})} \Big)\one_{\alpha<2}\Big].
    \end{aligned}
  \]
  At this stage, arguing similarly to the first inequality we also come up with
  \[
    \int_{\Rd_+} w(x,y)^{1/p}K(x,y)dx\lesi 1.
  \]
  This completes the proof of \eqref{eq:thm-HardyIneqPre} under the assumption $(\f{\alpha s}{2}-\sigma)_+<\f{1}{p}<1+\sigma\wedge0$.
\end{proof}

\section{Proof of density of $(L_\lambda)^{s/2}C_c^\infty(\R^d_+)$ in $L^p(\R_+^d)$ (Theorem~\ref{density})}
\label{s:density}

The strategy to prove Theorem~\ref{density} is the same as in \cite[Theorem~25]{FrankMerz2023}.
We first prove Theorem \ref{density} for $f$ of a special form, namely, $f\in L_\lambda^{s/2} \me{-tL_\lambda} C_c^\infty(\R^d_+)$ for some $0<t<\infty$. To do so, we use the following pointwise bounds on functions in $\me{-tL_\lambda} C_c^\infty(\R^d_+)$. For a function $u$ on a set $\Omega$ and $0<\beta\le 2$, we define its $\beta$-th H\"older seminorm as
\begin{equation}
  \label{eq:holder}
  [u]_{C^\beta(\Omega)} :=
  \begin{cases}
    \sup_{x,y\in\Omega} \frac{|u(x)-u(y)|}{|x-y|^\beta} & \text{if}\ 0<\beta\le 1 \,, \\
    \sup_{x,y\in\Omega} \frac{|\nabla u(x)-\nabla u(y)|}{|x-y|^{\beta-1}} & \text{if}\ 1<\beta\le 2 \,.
  \end{cases}
\end{equation}

\begin{lem}[{\cite[Lemma~27]{FrankMerz2023}}]
  \label{pointwise}
  Let $\alpha$, $\lambda$ and $\sigma$ be as in Theorem \ref{density}. Let $0<t<\infty$ and $\psi\in \me{-tL_\lambda} C_c^\infty(\R^d_+)$. Then, for all $x\in\R^d_+$,
  \begin{align}
    \label{eq:heatbound1}
    |\psi(x)| & \lesssim (1\wedge x_d)^\sigma(1\wedge |x|^{-d-\alpha}) \,, \\
    \label{eq:heatbound2}
    |L_\lambda\psi(x)| & \lesssim (1\wedge x_d)^\sigma (1\wedge |x|^{-d-\alpha}) \,, \\
    \label{eq:heatbound3}
    |(-\Delta)^{\alpha/2} \psi(x)| & \lesssim (1\wedge x_d)^{\sigma-\alpha} (1\wedge |x|^{-d-\alpha}) \,, \\
    \label{eq:heatbound4}
    [\psi]_{C^\beta(B_{\ell_x}(x))} & \lesssim (1\wedge x_d)^{\sigma-\beta}  (1\wedge |x|^{-d-\alpha})
                                      \quad \text{with}\ \ell_x:=1\wedge\tfrac{x_d}{2} \,,\ 0<\beta<\alpha.
  \end{align} 
\end{lem}

As noted in \cite{FrankMerz2023} the decay in the bounds for $\alpha=2$ can be vastly improved, but it is convenient for us to have a unified statement.
To give the proof of Theorem~\ref{density}, we use, for two parameters $0<r\le 1\le R<\infty$, the commutator bounds in Corollary~\ref{cutoffcomb} below. These bounds involve functions satisfying
\begin{equation}
  \label{eq:asschi1}
  0\le\chi\le 1 \,,
  \qquad
  \chi(x) = 1 \ \text{if}\ |x|\le R \,,
  \qquad
  \chi(x) = 0 \ \text{if}\ |x|\ge 2R \,,
  \qquad
  |\nabla\chi|\lesssim R^{-1} \,,
\end{equation}
and
\begin{equation}
  \label{eq:asstheta1}
  0\le\theta\le 1 \,,
  \qquad
  \theta(x) = 0 \ \text{if}\ x_d\le r \,,
  \qquad
  \theta(x) = 1 \ \text{if}\ x_d\ge 2r \,,
  \qquad
  |\nabla\theta|\lesssim r^{-1} \,,
\end{equation}
and, if $\alpha\ge 1$, also the following bounds on their Hessians, namely
\begin{equation}
  \label{eq:asschi2}
  |D^2\chi|\lesssim R^{-2},
\end{equation}
and
\begin{equation}
  \label{eq:asstheta2}
  |D^2\theta|\lesssim r^{-2}.
\end{equation}

\begin{cor}
  \label{cutoffcomb}
  Let $p\in(1,\vc)$ and $0<\alpha<2$. Let $0<r\le 1\le R<\infty$, assume that $\chi$ and $\theta$ satisfy \eqref{eq:asschi1} and \eqref{eq:asstheta1} and, if $\alpha\ge 1$, also \eqref{eq:asschi2} and \eqref{eq:asstheta2}. Let $\frac{\alpha-1}{2}\le \sigma<\alpha$, assume that $v$ satisfies
  \begin{equation}
    \label{eq:assv1}
    |v(x)|\le (1\wedge |x|^{-d-\alpha}) (1\wedge x_d)^\sigma
    \qquad\text{for all}\ x\in\R^d_+ \, 
  \end{equation}
  and, if $\alpha\ge 1$, also
  \begin{equation}
    \label{eq:assv2}
    [v]_{C^\beta(B_{\ell_x}(x))} \le (1\wedge |x|^{-d-\alpha}) \, (1\wedge x_d)^{\sigma-\beta}
    \qquad\text{for all}\ x\in\R^d_+\ \text{with}\ \ell_x := 1\wedge \tfrac{x_d}2
  \end{equation}
  with some $\beta>\alpha-1$. Then
  $$
  \| [(-\Delta)^{\alpha/2},\chi\theta] v \|_{L^p(\R^d_+)} \lesssim r^{\sigma-\alpha+1/p} + R^{-\alpha-d/p'} \,.
  $$
\end{cor}

\begin{proof}
  The proof is exactly as in \cite[Corollary~23]{FrankMerz2023}. For an explicit constant $c_{d,\alpha}>0$, we estimate the $L^p$-norm of the right-hand side of
  \begin{align}
    \label{eq:cutoffcombaux1}
    c_{d,\alpha}[(-\Delta)^{\alpha/2},\chi\theta] v(x) = \theta(x) \int_{\R^d}  \frac{\chi(x)-\chi(y)}{|x-y|^{d+\alpha}} v(y)\,dy + \int_{\R^d} \frac{\theta(x)-\theta(y)}{|x-y|^{d+\alpha}} \chi(y)v(y)\,dy.
  \end{align}
  For the first summand on the right-hand side, we estimate $|\theta(x)|\leq1$ and use the pointwise bounds in \cite[Lemmas~20]{FrankMerz2023}, i.e.,
  \begin{align}
    \begin{split}
      & \left|\int_{\R^d} \frac{\chi(x)-\chi(y)}{|x-y|^{d+\alpha}}v(y)\,dy\right|
      \lesssim \one_{|x|\leq R} R^{-d-2\alpha} + \one_{|x|>R} |x|^{-d-\alpha} \\
      & \quad + \one_{\alpha\geq1}\one_{|x|\simeq R}
        R^{-d-\alpha-1} \left( (1\wedge x_d)^{-(\sigma-\alpha+1)_-} + \one_{\sigma=\alpha-1}\ln\tfrac1{1\wedge x_d} + \one_{\alpha=1} \ln R \right) \\
      & \quad \text{for all}\ x\in\R^d_+
    \end{split}
  \end{align}
  and all $v$ satisfying the assumptions of the present corollary. The $L^p$-norm of this right-hand side is $\mathcal{O}(R^{-\alpha-d/p'})$ for $R>1$. For the second summand on the right-hand side of \eqref{eq:cutoffcombaux1}, we use \cite[Lemmas~21]{FrankMerz2023}, i.e.,
  \begin{align}
    \label{eq:cutoffcombaux2}
    \begin{split}
      \left|\int_{\R^d} \frac{\theta(x)-\theta(y)}{|x-y|^{d+\alpha}}(\chi v)(y)\,dy\right|
      \lesssim \left(r^{\sigma-\alpha}\wedge \frac{r^{\sigma+1}}{x_d^{1+\alpha}}\right) (1+x_d)^{1+\alpha} (1\wedge |x|^{-d-\alpha})
      \quad \text{for all}\ x\in\R_+^d
    \end{split}
  \end{align}
  whenever $\chi v$ satisfies the assumptions of the present corollary. By \cite[Corollary~23]{FrankMerz2023} the function $\chi v$ indeed satisfies these assumptions when $v$ does. Since the $L^p$-norm of the right-hand side of \eqref{eq:cutoffcombaux2} is $\mathcal{O}(r^{\sigma-\alpha+1/p})$ for $r<1$, the proof is concluded.
\end{proof}

\begin{proof}[Proof of Theorem \ref{density}]
  \emph{Step 1.} We first prove this theorem for $f\in (L_\lambda)^{s/2}\me{-tL_\lambda} C_c^\infty(\R^d_+)$ for some $0<t<\infty$. Let $0<t<\infty$ and let $\psi\in \me{-tL_\lambda} C_c^\infty(\R^d)$. For parameters $0<r\le 1\le R<\infty$ to be determined, let $\chi$ and $\theta$ be as in Corollary~\ref{cutoffcomb} and
  $$
  \phi := \chi\theta \psi.
  $$
  Then, by \eqref{eq:heatbound1},
  $$
  \|\phi - \psi\|_{L^p(\R^d_+)} \le \| \one_{x_d\le 2r} \psi\|_{L^p(\R^d_+)} + \| \one_{|x|> R} \psi\|_{L^p(\R^d_+)} \lesssim r^{\sigma+1/p} + R^{-\alpha -d/p'} \,.
  $$
  Moreover,
  $$
  \|L_\lambda(\phi - \psi)\|_{L^p(\R^d_+)} \le \|(1-\chi\theta)L_\lambda\psi\|_{L^p(\R^d_+)} + \| [(-\Delta)^{\alpha/2},\chi\theta] \psi \|_{L^p(\R^d_+)} 
  $$
  and, by \eqref{eq:heatbound2},
  $$
  \|(1-\chi\theta)L_\lambda\psi\|_{L^p(\R^d_+)} \lesssim r^{\sigma+1/p} + R^{-\alpha -d/p'} \,.
  $$
  For $\alpha<2$ we apply Corollary \ref{cutoffcomb} and find
  $$
  \| [(-\Delta)^{\alpha/2},\chi\theta] \psi \|_{L^p(\R^d_+)}  \lesssim r^{\sigma-\alpha+1/p} + R^{-\alpha-d/p'} \,.
  $$
  The same bound holds for $\alpha=2$ as well, as follows by writing
  $$
  [-\Delta,\chi\theta] \psi = -2\nabla(\chi\theta)\cdot\nabla\psi - \Delta(\chi\theta)\psi
  $$
  and using the pointwise bounds \eqref{eq:heatbound1} and \eqref{eq:heatbound4}. Thus, for all $\alpha\le 2$,
  $$
  \|L_\lambda(\phi - \psi)\|_{L^p(\R^d_+)} \lesssim r^{\sigma-\alpha+1/p} + R^{-\alpha-d/p'} \,.
  $$
  Since $0< s\le 2$ we have, by interpolation,
  $$
  \|L_\lambda^{s/2}(\phi - \psi)\|_{L^p(\R^d_+)} \le \|\phi - \psi\|_{L^p(\R^d_+)}^{1-s/2} \|L_\lambda(\phi - \psi)\|_{L^p(\R^d_+)}^{s/2} \,.
  $$
  Inserting the above bounds, we get
  $$
  \|L_\lambda^{s/2}(\phi - \psi)\|_{L^p(\R^d_+)} \lesssim r^{\sigma+1/p-\alpha s/2} + R^{-\alpha-d/p'}.
  $$
  Since, by assumption $s<2(1/p+\sigma)/\alpha$, this tends to zero as $r\to 0$ and $R\to\infty$.
  
  \medskip
  
  \emph{Step 2.} We now prove Theorem \ref{density} in the general case.
  Let $f\in L^p(\R^d_+)$ and $\epsilon>0$. By Lemma \ref{Calderon-reproducing formula}, there exist  $t:=t_1/2$ and $T:=t_2/2$ such that 
  \begin{equation}
    \label{eq1-density proof}
    \| (\me{-2tL_\lambda} - \me{-2TL_\lambda})f - f \|_{L^p(\R^d_+)} \le \epsilon \,.
  \end{equation}
  On the other hand,
  $$
  \me{-2tL_\lambda}-\me{-2TL_\lambda}f = L_\lambda^{s/2}(\me{-tL_\lambda}+\me{-TL_\lambda}) L_\lambda^{-s/2}(\me{-tL_\lambda}-\me{-TL_\lambda})f.
  $$
  By Lemma \ref{lem-boundedness of LsetL},
  $$
  \|L_\lambda^{-s/2}(\me{-tL_\lambda}-\me{-TL_\lambda})f\|\lesi (T^{s/2}-t^{s/2})\|f\|_p.
  $$
  Since $C_c^\infty(\R^d_+)$ is dense in $L^p(\R^d_+)$, there is a $k\in C^\infty_c(\R^d_+)$ such that
  \begin{equation}
    \label{eq2-density proof}
    \| k - L_\lambda^{-s/2}(\me{-tL_\lambda}-\me{-TL_\lambda}) f \|_{L^p(\R^d_+)} \le \f{\epsilon}{\|L_\lambda^{s/2}(\me{-tL_\lambda}+\me{-TL_\lambda})\|_{p\to p}}.
  \end{equation}
  We define $\psi:= L_\lambda^{s/2}(\me{-tL_\lambda}+\me{-TL_\lambda})k$. Then, according to Step 1 (applied both to $L_\lambda^{s/2} \me{-tL_\lambda}k$ and to $L_\lambda^{s/2} \me{-TL_\lambda}k$) there is a $\phi\in C_c^\infty$ such that
  \begin{equation}\label{eq3-density proof}
    \| L_\lambda^{s/2} \phi - \psi \|_{L^p(\R^d_+)} \le \epsilon \,.
  \end{equation}
  Therefore,
  \[
  \begin{aligned}
    \|L_\lambda^{s/2} \phi -f\|_p
    &\le \|L_\lambda^{s/2} \phi -\psi\|_p +\|\psi-L_\lambda^{-s/2}(\me{-tL_\lambda}+\me{-TL_\lambda})L_\lambda^{-s/2}(\me{-tL_\lambda}-\me{-TL_\lambda})\me{-TL_\lambda}) f \|_p\\
    & \ \ + \|L_\lambda^{-s/2}(\me{-tL_\lambda}-\me{-TL_\lambda})L_\lambda^{-s/2}(\me{-tL_\lambda}-\me{-TL_\lambda}) f -f\|_p.
  \end{aligned}
  \]
  By \eqref{eq3-density proof},
  \[
    \|L_\lambda^{s/2} \phi -\psi\|_p\le \epsilon.
  \]
  Using \eqref{eq2-density proof} and Lemma \ref{lem-boundedness of LsetL},
  \[
    \begin{aligned}
      & \|\psi-L_\lambda^{-s/2}(\me{-tL_\lambda}-L_\lambda^{-s/2}(\me{-tL_\lambda}+\me{-TL_\lambda})\me{-TL_\lambda})\me{-TL_\lambda}) f \|_p\\
      & \quad = \|L_\lambda^{s/2}(\me{-tL_\lambda}+\me{-TL_\lambda})k-L_\lambda^{-s/2}(\me{-tL_\lambda}-L_\lambda^{-s/2}(\me{-tL_\lambda}-\me{-TL_\lambda})\me{-TL_\lambda}) f \|_p\\
      & \quad \le \|L_\lambda^{s/2}(\me{-tL_\lambda}+\me{-TL_\lambda})\|_{p\to p}\| k - L_\lambda^{-s/2}(\me{-tL_\lambda}-\me{-TL_\lambda}) f \|_{L^p(\R^d_+)}
        \le \epsilon.
    \end{aligned}
  \]
  Finally, by \eqref{eq1-density proof},
  \[
    \begin{aligned}
      \|L_\lambda^{-s/2}(\me{-tL_\lambda}-\me{-TL_\lambda})L_\lambda^{-s/2}(\me{-tL_\lambda}-\me{-TL_\lambda}) f -f\|_p
      &=\|(\me{-2tL_\lambda} - \me{-2TL_\lambda})f - f\|_p
        \le \epsilon.
    \end{aligned}
  \]
  Consequently, for any $\epsilon>0$ there is a $\phi\in C_c^\infty$ such that
  \[
  \|L_\lambda^{-s/2}\phi -f\|_p\le 3\epsilon.
  \]  
  This concludes the proof.
\end{proof}



\begin{thebibliography}{KMV{\etalchar{+}}18}

\bibitem[Aus07]{Auscher2007}
Pascal Auscher.
\newblock On necessary and sufficient conditions for {$L^p$}-estimates of
  {R}iesz transforms associated to elliptic operators on {$\mathbb{R}^n$} and
  related estimates.
\newblock {\em Mem. Amer. Math. Soc.}, 186(871):xviii+75, 2007.

\bibitem[BBC03]{Bogdanetal2003}
Krzysztof Bogdan, Krzysztof Burdzy, and Zhen-Qing Chen.
\newblock Censored stable processes.
\newblock {\em Probab. Theory Related Fields}, 127(1):89--152, 2003.

\bibitem[BD11]{BogdanDyda2011}
Krzysztof Bogdan and Bart{\l}omiej Dyda.
\newblock The best constant in a fractional {H}ardy inequality.
\newblock {\em Math. Nachr.}, 284(5-6):629--638, 2011.

\bibitem[BD23]{BuiDAncona2023}
The~Anh Bui and Piero D'Ancona.
\newblock Generalized {H}ardy operators.
\newblock {\em Nonlinearity}, 36(1):171--198, 2023.

\bibitem[BDM26]{Buietal2026}
The~Anh Bui, Xuan~Thinh Duong, and Konstantin Merz.
\newblock Equivalence of {S}obolev norms for {K}olmogorov operators with
  scaling-critical drift.
\newblock {\em Nonlinearity}, 39(1):015021, January 2026.

\bibitem[BM25]{BogdanMerz2025}
Krzysztof Bogdan and Konstantin Merz.
\newblock Subordinated {B}essel heat kernels.
\newblock {\em Theory Probab. Math. Statist.}, (112):67--83, 2025.

\bibitem[BN22]{BuiNader2022}
The~Anh Bui and Georges Nader.
\newblock Hardy spaces associated to generalized {H}ardy operators and
  applications.
\newblock {\em NoDEA Nonlinear Differential Equations Appl.}, 29(4):Paper No.
  40, 40, 2022.

\bibitem[BS02]{BorodinSalminen2002}
Andrei~N. Borodin and Paavo Salminen.
\newblock {\em Handbook of {B}rownian Motion --- Facts and Formulae}.
\newblock Probability and its Applications. Birkh\"{a}user Verlag, Basel,
  second edition, 2002.

\bibitem[CK03]{ChenKumagai2003}
Zhen-Qing Chen and Takashi Kumagai.
\newblock Heat kernel estimates for stable-like processes on {$d$}-sets.
\newblock {\em Stochastic Process. Appl.}, 108(1):27--62, 2003.

\bibitem[CKS10]{Chenetal2010}
Zhen-Qing Chen, Panki Kim, and Renming Song.
\newblock Two-sided heat kernel estimates for censored stable-like processes.
\newblock {\em Probab. Theory Related Fields}, 146(3-4):361--399, 2010.

\bibitem[CKSV20]{Choetal2020}
Soobin Cho, Panki Kim, Renming Song, and Zoran Vondra\v{c}ek.
\newblock Factorization and estimates of {D}irichlet heat kernels for non-local
  operators with critical killings.
\newblock {\em J. Math. Pures Appl. (9)}, 143:208--256, 2020.

\bibitem[Dav90]{Davies1990}
E.~B. Davies.
\newblock {\em Heat Kernels and Spectral Theory}, volume~92 of {\em Cambridge
  Tracts in Mathematics}.
\newblock Cambridge University Press, Cambridge, 1990.

\bibitem[DR96]{DuongRobinson1996}
Xuan~T. Duong and Derek~W. Robinson.
\newblock Semigroup kernels, {P}oisson bounds, and holomorphic functional
  calculus.
\newblock {\em J. Funct. Anal.}, 142(1):89--128, 1996.

\bibitem[FM23]{FrankMerz2023}
Rupert~L. Frank and Konstantin Merz.
\newblock On {S}obolev norms involving {H}ardy operators in a half-space.
\newblock {\em J. Funct. Anal.}, 285(10):Paper No. 110104, 2023.

\bibitem[FMS21]{Franketal2021}
Rupert~L. Frank, Konstantin Merz, and Heinz Siedentop.
\newblock Equivalence of {S}obolev norms involving generalized {H}ardy
  operators.
\newblock {\em International Mathematics Research Notices}, 2021(3):2284--2303,
  February 2021.

\bibitem[FMS23a]{Franketal2023}
Rupert~L. Frank, Konstantin Merz, and Heinz Siedentop.
\newblock Relativistic strong {S}cott conjecture: a short proof.
\newblock In {\em Density Functionals for Many-Particle Systems---Mathematical
  Theory and Physical Applications of Effective Equations}, volume~41 of {\em
  Lect. Notes Ser. Inst. Math. Sci. Natl. Univ. Singap.}, pages 69--79. World
  Sci. Publ., Hackensack, NJ, March 2023.

\bibitem[FMS23b]{Franketal2023T}
Rupert~L. Frank, Konstantin Merz, and Heinz Siedentop.
\newblock The {S}cott conjecture for large {C}oulomb systems: a review.
\newblock {\em Lett. Math. Phys.}, 113(1):Paper No. 11, 79, 2023.

\bibitem[FMSS20]{Franketal2020P}
Rupert~L. {Frank}, Konstantin {Merz}, Heinz {Siedentop}, and Barry {Simon}.
\newblock Proof of the strong {S}cott conjecture for {C}handrasekhar atoms.
\newblock {\em Pure Appl. Funct. Anal.}, 5(6):1319--1356, December 2020.

\bibitem[GT21]{GrzywnyTrojan2021}
Tomasz {Grzywny} and Bartosz {Trojan}.
\newblock {Subordinated Markov processes: sharp estimates for heat kernels and
  Green functions}.
\newblock {\em arXiv e-prints}, page arXiv:2110.01201, October 2021.

\bibitem[Har19]{Hardy1919}
G.~H. Hardy.
\newblock Notes on some points in the integral calculus {LI}: On {H}ilbert's
  double-series theorem, and some connected theorems concerning the convergence
  of infinite series and integrals.
\newblock {\em Messenger of Mathematics}, 48:107--112, 1919.

\bibitem[Har20]{Hardy1920}
G.~H. Hardy.
\newblock Note on a theorem of {H}ilbert.
\newblock {\em Mathematische Zeitschrift}, 6(3--4):314--317, 1920.

\bibitem[JM25]{JakubowskiMaciocha2025}
Tomasz Jakubowski and Pawe{\l} Maciocha.
\newblock Heat kernel estimates of fractional {S}chr\"{o}dinger operators with
  {H}ardy potential on half-line.
\newblock {\em Potential Anal.}, 63(1):125--146, 2025.

\bibitem[KMP06]{Kufneretal2006}
Alois Kufner, Lech Maligranda, and Lars-Erik Persson.
\newblock The prehistory of the {H}ardy inequality.
\newblock {\em Amer. Math. Monthly}, 113(8):715--732, 2006.

\bibitem[KMV{\etalchar{+}}17]{Killipetal2017}
Rowan Killip, Changxing Miao, Monica Visan, Junyong Zhang, and Jiqiang Zheng.
\newblock The energy-critical {NLS} with inverse-square potential.
\newblock {\em Discrete Contin. Dyn. Syst.}, 37(7):3831--3866, 2017.

\bibitem[KMV{\etalchar{+}}18]{Killipetal2018}
R.~Killip, C.~Miao, M.~Visan, J.~Zhang, and J.~Zheng.
\newblock Sobolev spaces adapted to the {S}chr\"odinger operator with
  inverse-square potential.
\newblock {\em Math. Z.}, 288(3-4):1273--1298, 2018.

\bibitem[KMVZ17]{Killipetal2017T}
Rowan Killip, Jason Murphy, Monica Visan, and Jiqiang Zheng.
\newblock The focusing cubic {NLS} with inverse-square potential in three space
  dimensions.
\newblock {\em Differential Integral Equations}, 30(3-4):161--206, 2017.

\bibitem[KVZ16a]{Killipetal2016Q}
Rowan Killip, Monica Visan, and Xiaoyi Zhang.
\newblock Quintic {NLS} in the exterior of a strictly convex obstacle.
\newblock {\em Amer. J. Math.}, 138(5):1193--1346, 2016.

\bibitem[KVZ16b]{Killipetal2016}
Rowan Killip, Monica Visan, and Xiaoyi Zhang.
\newblock Riesz transforms outside a convex obstacle.
\newblock {\em Int. Math. Res. Not. IMRN}, (19):5875--5921, 2016.

\bibitem[Mer21]{Merz2021}
Konstantin Merz.
\newblock On scales of {S}obolev spaces associated to generalized {H}ardy
  operators.
\newblock {\em Math. Z.}, 299(1):101--121, 2021.

\bibitem[Mer22]{Merz2022}
Konstantin Merz.
\newblock On complex-time heat kernels of fractional {S}chr{\"o}dinger
  operators via {P}hragm{\'e}n-{L}indel{\"o}f principle.
\newblock {\em J. Evol. Equ.}, 22(3):Paper No. 62, 30, 2022.

\bibitem[MS22]{MerzSiedentop2022}
Konstantin Merz and Heinz Siedentop.
\newblock Proof of the strong {Scott} conjecture for heavy atoms: the {Furry}
  picture.
\newblock {\em Annales Henri Lebesgue}, 5:611--642, 2022.

\bibitem[MSZ23]{Miaoetal2023}
Changxing Miao, Xiaoyan Su, and Jiqiang Zheng.
\newblock The {$W^{s,p}$}-boundedness of stationary wave operators for the
  {S}chr\"{o}dinger operator with inverse-square potential.
\newblock {\em Trans. Amer. Math. Soc.}, 376(3):1739--1797, 2023.

\bibitem[OK90]{OpicKufner1990}
B.~Opic and A.~Kufner.
\newblock {\em Hardy-type inequalities}, volume 219 of {\em Pitman Research
  Notes in Mathematics Series}.
\newblock Longman Scientific \& Technical, Harlow, 1990.

\bibitem[SWW25]{Songetal2025}
Renming Song, Peixue Wu, and Shukun Wu.
\newblock Heat kernel estimates for regional fractional {L}aplacians with
  multi-singular critical potentials in {$C^{1,\beta}$} open sets.
\newblock {\em Stochastic Process. Appl.}, 189:Paper No. 104727, 21, 2025.

\end{thebibliography}

\newcommand{\etalchar}[1]{$^{#1}$}
\def\cprime{$'$}

\end{document}